\documentclass{article}

\usepackage[tbtags]{amsmath}
\usepackage{amsfonts,amssymb,mathrsfs,amscd,bbm}
\usepackage[matrix,arrow,curve]{xy}
\usepackage{color}

\usepackage{lscape}
\usepackage{rotating}
\usepackage {epstopdf}
\usepackage{wasysym}
\usepackage{mathrsfs}
\usepackage{graphicx}

\usepackage{hypbmsec}

\newtheorem{theorem}{Theorem}
\newtheorem{proposition}{Proposition}
\newtheorem{lemma}{Lemma}
\newtheorem{corollary}{Corollary}
\newtheorem{conjecture}{Conjecture}

\newtheorem{proof}{Proof}
\newtheorem{definition}{Definition}
\newtheorem{remark}{Remark}
\newtheorem{example}{Example}
\newtheorem{problem}{}

\numberwithin{equation}{section}
\allowdisplaybreaks

\newcounter{aa}

\let\mathcal=\mathscr

\begin{document}

\author{Anatoly~M.~Vershik\footnote{St.~Petersburg Department
of Steklov Mathematical Institute;
St.~Petersburg State University
Institute for Information Transmission Problems.
E-mail: avershik@gmail.com. Partially supported by the Russian Science Foundation
 (grant No.~14-11-00581).}}

\date{24.01.2017}


\title{The theory of filtrations of subalgebras, standardness and independence}

\maketitle

\begin{flushright}{In memory of my friend Kolya K{\rm.} {\rm (1933--2014)}\\[2mm]
``\textit{Independence is the best quality,\\
the best word in all languages}.''\\[0.3mm]
{\fontsize{8}{9pt}\selectfont {\rm J.~Brodsky. From a~personal letter.}}}
\end{flushright}

\begin{abstract}
The survey is devoted to the combinatorial and metric theory of filtrations, i.\,e., decreasing sequences of $\sigma$-algebras in measure spaces or  decreasing sequences of subalgebras of certain algebras. One of the key notions, that of standardness, plays the role of a~generalization of the notion of the independence of a~sequence of random variables. We discuss the possibility of obtaining a~classification of filtrations, their invariants, and various links to problems in algebra, dynamics, and combinatorics.

Bibliography: 101 titles.
\end{abstract}


\setcounter{tocdepth}{2} \tableofcontents

\section{Introduction}
\label{sec1}

This survey deals with an area in which measure theory is intertwined with combinatorics and asymptotic analysis; its applications to dynamics, algebra, and other problems will be partly touched upon in this paper, as well as in subsequent publications. The survey contains a~lot of new problems from the dynamic theory of graphs and representation theory of groups and algebras.

\subsection{A~simple example and a~difficult question}

\label{ssec1.1} We begin with an elementary example that illustrates the notion of filtration and  problems of combinatorial measure theory.

Consider the space of all one-sided sequences of zeros and ones:
$$
X=\bigl\{\{x_n\}_{n=1}^{\infty}\bigr\},\qquad
x_n=0 \vee 1,\quad
n=1,2,\dots,
$$
i.\,e., the infinite product of the two-point space. We will regard~$X$ as a~dyadic (Cantor-like) compactum in the weak topology, and also as a~Borel space. In~$X$ we introduce the ``tail''  equivalence relation and the tail filtration of $\sigma$-subalgebras of sets. Namely, two sequences~$\{x_k\}$,~$\{x'_k\}$ are

-- $n$-equivalent if $x_{k+n}=x'_{k+n}$ for all $k\geqslant 0$; and

-- equivalent with respect to the tail equivalence relation if they are $n$-equivalent for some~$n$.

Denote by~${\mathfrak A}_n$, $n=0,1,2,\dots$,  the $\sigma$-algebra of Borel subsets in~$X$ that along with every point~$x$ contain all points $n$-equivalent to~$x$. In other words, ${\mathfrak A}_n$,
$n=0,1,2,\dots$, is the $\sigma$-subalgebra of Borel sets determined by conditions on the coordinates with indices~$\ge n$. The decreasing sequence
$$
{\mathfrak A}_0 \supset {\mathfrak A}_1 \supset {\mathfrak A}_2\supset \cdots
$$
is called the \textit{tail Borel filtration of the space~$X$
regarded as an infinite direct product},
$$
X=\prod_{n=1}^{\infty}[0;1].
$$

If we endow the space~$X$ with an arbitrary Borel measure~$\mu$, then the same sequence of $\sigma$-subalgebras, which are now understood as $\sigma$-subalgebras of classes of sets coinciding  $\operatorname{mod} 0$ (with respect to the given measure~$\mu$), provides an example of a~filtration in a~standard measure space.

For instance, take~$\mu$ to be the Bernoulli measure equal to the infinite product of the measures
 $\theta=(1/2,1/2)$; the resulting filtration is called the
 \textit{dyadic Bernoulli filtration}. By the famous Kolmogorov's zero--one (``all-or-none'') law, the intersection $\bigcap\limits_n{\mathfrak A}_n$
is the trivial $\sigma$-algebra~$\mathfrak N$, which contains only two classes of sets, of measure either zero or one.

For every positive integer~$n$, the space $(X,\mu)$ can be represented as a~direct product of measure spaces:
$$
(X,\mu)=\biggl(\,\prod_1^n\{(0;1),\theta\}\biggr) \times (X_n,\mu_n),
$$
where $(X_n,\mu_n)=\displaystyle\prod_{n+1}^{\infty}\{(0;1),\theta\}$.

Now we are going to state the main problem. Assume that we have a~dyadic compactum~$X'$ with a~Borel probability measure~$\mu'$ satisfying Kolmogorov's zero--one law, and for every positive integer~$n$ there is an isomorphism of measure spaces
$$
(X',\mu') \simeq \biggl(\,\prod_1^n\{(0;1),\theta\}\biggr) \times
(X'_n,\mu'_n),
$$
where $(X'_n,\mu'_n)$ are some measure spaces.

\textit{Can we claim that there exists an isomorphism~$T$ of measure spaces,
$T(X',\mu')=(X,\mu)$, for which $T(X'_n,\mu')=(X_n,\mu)$ for all~$n$?}

In the language to be described below, this is the question of whether a~finitely Bernoulli dyadic ergodic filtration is unique up to isomorphism.

It is similar to the following question from the theory of infinite tensor products of $C^*$-algebras and perhaps looks even more natural.

Consider a~$C^*$-algebra~$\mathscr A$ and assume that there is a~decreasing sequence
${\mathscr A}_n \subset \mathscr A$ of $C^*$-subalgebras of~$\mathscr A$ whose intersection is the space of constants,
$\bigcap\limits_n {\mathscr A}_n=\{\operatorname{Const}\}$, such that for all~$n$
$$
{\mathscr A}\simeq M_2(\mathbb C)^{{\otimes}n}\otimes{\mathscr A}_n
$$
(here $\simeq$ means an isomorphism of $C^*$-algebras).

\textit{Is it true that there exists an isomorphism}
$$
{\mathscr A}\simeq\prod_{1}^{\infty}
{\vphantom{\prod}}^{\!\otimes} M_2(\mathbb C)
$$
\textit{that sends~${\mathscr
A}_n$ to the subalgebra
$\displaystyle\prod_{n+1}^{\infty}{\vphantom{\prod}}^{\!\otimes}
M_2(\mathbb C)\subset {\mathscr A}$ for all~$n$?}

These problems were raised by the author about fifty years ago. In the original terms, the first of them looked as follows: is it true that every  filtration such that its finite segments are isomorphic to finite segments of a~Bernoulli filtration and the intersection of all its $\sigma$-algebras is trivial is isomorphic to a~Bernoulli filtration? The problem originated from ergodic theory, as we will discuss  below. In~\cite{59}, \textit{both questions were answered in the negative}; moreover, it turned out that there is a~continuum of measure spaces (respectively, $C^*$-algebras) with  dyadic structures that are finitely isomorphic but pairwise nonisomorphic globally. This discovery launched an important field, which, however, still has not gained sufficient attention: the asymptotic theory of filtrations. For this survey, we have selected the most important facts, both known for a~long time and new, and, most importantly, new problems in this area. The theory of filtrations has applications to measure theory, random processes, and dynamical systems, as well as to representation theory of groups and algebras. The main problem concerns the complexity of the asymptotic behavior at infinity of monotone sequences of algebras and a~possible classification of such sequences. The second question, about tensor products of
$C^*$-algebras, will be considered in another paper; here we have mentioned it to emphasize the parallelism between problems from very different areas of mathematics. Both these problems belong to asymptotic algebraic analysis.

Conceptually new ideas, as compared with the previous research on filtrations, appeared in connection with the theory of graded graphs (Bratteli diagrams); this is one of the central topics of this survey. Apparently, this connection has not been noticed earlier, though  the tail filtration is a~most important object in the theory of  $\operatorname{AF}$-algebras. From the standpoint of the theory of graded graphs, filtrations were considered in the author's recent paper~\cite{82}; here, on the contrary, we focus on the measure-theoretic
aspect of the problem and use notions and techniques  related to graphs.

\subsection{Three languages of measure theory}
\label{ssec1.2}
Let us return to the measure-the\-o\-re\-tic statement of the first problem.
If we consider the algebra~$L^{\infty}(X,\mu)$ of all classes of measurable bounded functions on the space
$(X,\mu)$ coinciding $\operatorname{mod} 0$, we obtain a~filtration of subalgebras:
$$
L^{\infty}(X,\mu) \equiv{\mathscr A}_0\supset {\mathscr A}_1 \supset
{\mathscr A}_2 \supset \cdots,
$$
where~${\mathscr A}_n$ is the subalgebra in~$L^{\infty}(X,\mu)$ consisting of all functions that depend on coordinates with indices~$\ge n$ ($n=0,1,2,\dots$).

Now let us describe this example in the language of partitions. Let~$\xi_n$ be the partition of
$(X,\mu)$ into the classes of sequences in which the coordinates with indices greater than~$n$ coincide; then ${\mathfrak A}_n$ is the $\sigma$-algebra of sets measurable with respect to the partition~$\xi_n$ (i.\,e., sets composed of elements of this partition). And ${\mathscr A}_n$ is the space of functions measurable with respect to~$\xi_n$. The fact that the sequence of partitions is decreasing means that almost every element of~$\xi_n$ is a~union of some elements of~$\xi_{n-1}$, $n=1,2,\dots$\,. (The partial order in the space of measurable partitions is discussed in Sec.~\ref{ssec2.2}.)

In these terms, ${\mathscr A}_n$ is the algebra~$L^{\infty}(X/{\xi_n},\mu_n)$, where $X/{\xi_n}$ is the quotient of~$(X,\mu)$ by~$\xi_n$ and
$\mu_{\xi_n}$ is the quotient measure on~ $X/{\xi_n}$.

Thus, there is a~functorial equivalence of the three languages of the theory of filtrations described above: the language of filtrations of $\sigma$-subalgebras of a~measure space, the language of filtrations of subspaces of measurable functions, and the language of filtrations of measurable partitions (precise definitions will be given below).

By filtrations we will mean either infinite decreasing sequences of subalgebras of a~commutative algebra with involution  (for instance, $L^{\infty}(X,\mu)$), or (equivalently) infinite decreasing sequences of $\sigma$-subalgebras of sets in a~standard measure space~$(X,\mu)$, or, finally (in yet another equivalent geometric language which we will use most frequently), infinite decreasing sequences of measurable partitions.

All  three languages are equivalent; the difference between them is terminological. The first context (filtrations of subalgebras of an algebra with involution) is the most general one, it admits an important generalization achieved by abandoning the commutativity of the ambient algebra and passing to $\operatorname{AF}$-algebras; as mentioned above, this generalization is not touched upon here, though it is closely related to the content of the paper.

We will also consider decreasing sequences of $\sigma$-subalgebras of sets in a~standard Borel space, i.\,e., Borel filtrations. In this case, no measure on the space is given, and the problem arises of describing all measures that agree with the given filtration. This statement  covers almost all problems related to invariant measures in different contexts.

{

\makeatletter

\def\@ssect#1#2#3#4#5{\@tempskipa #3\relax
   \ifdim \@tempskipa>\z@
     \begingroup #4\@hangfrom{\hskip #1}{\interlinepenalty \@M #5\par}%
     \endgroup
   \else \def\@svsechd{#4\hskip #1\relax #5.}\fi
    \@xsect{#3}}

\newbox\@subsection@box@

\def\@sect#1#2#3#4#5#6[#7]#8{%
\setbox\@subsection@box@=\hbox{#8}%
     \ifnum #2>\c@secnumdepth
     \let\@svsec\@empty\else
     \refstepcounter{#1}\protected@edef\@svsec{\csname the#1\endcsname.}%
     \fi
     \@tempskipa #5\relax
      \ifdim \@tempskipa>\z@
        \begingroup #6\relax
          \@hangfrom{\hskip #3\relax\textup{\@svsec}}{\interlinepenalty \@M #8\par}%
        \endgroup
       \csname #1mark\endcsname{#7}\addcontentsline
         {toc}{#1}{\ifnum #2>\c@secnumdepth \else
                      \protect\numberline{\csname the#1\endcsname}\fi
                    #7}\else
        \def\@svsechd{#6\hskip #3\relax  
                   \textup{\@svsec}%
                   \ifdim\wd\@subsection@box@>0pt\enspace #8\fi
                   \csname #1mark\endcsname
                      {#7}\addcontentsline
                           {toc}{#1}{\ifnum #2>\c@secnumdepth \else
                             \protect\numberline{\csname the#1\endcsname}\fi
                       #7}}\fi
     \@xsect{#5}}

\makeatother

\subsection{Where do filtrations appear?}
\label{ssec1.3}
Filtrations appear in the theory of random processes, potential theory and the theory of Markov processes, ergodic theory, topological and metric dynamics.

The most interesting applications belong to the theory of graded graphs (Bratteli diagrams), the theory of
 $\operatorname{AF}$-algebras, and asymptotic combinatorics. Here are several general examples.

}

\smallskip
\textbf{A.~The ``past'' of random processes.}
Consider an arbitrary random process, for instance, a~process~$\{x_m\}_m$ with real values and discrete time,
$m \leqslant 0$ (it is convenient to index the random variables by the negative integers).

The $\sigma$-algebra~${\mathfrak A}_n$ consists of all measurable sets that can be described by the ``past'' of the process, i.\,e., by the random variables with indices
 $m <-n$. In this context, $\{{\mathfrak A}_n\}_n$ is the
\textit{past filtration of the random process}. In a~similar way, we can define the ``future'' filtration  of a~random process. A~case of importance in ergodic theory is that of a~stationary filtration:
$$
\{\xi_n\}_{n=0}^{\infty}\simeq  \{\xi_n/{\xi_1}\}_{n=1}^{\infty},
$$
where  $\simeq$ stands for metric isomorphism, i.\,e., means the existence of a~measure-pre\-serv\-ing automorphism that sends the first filtration to the second one. In this case, the process~$\{x_m\}_m$ described above is stationary, i.\,e., the corresponding probability measure is invariant under the left shift ${T\{x_m\}}_m=\{x_{m-1}\}_m$,
$m<0$. In other words, the shift executes the passage to the quotient by the first partition, and the quotient filtration is isomorphic to the original one. There are several papers on the theory of stationary filtrations; for applications of the standardness criterion, see Sec.~\ref{sec5}. However, one should keep in mind that by no means all invariants of a~stationary filtration are endomorphism invariants. A~remarkable example of such an invariant is Ornstein's ``very weak Bernoulli'' (VWB) condition.

\smallskip
\textbf{B.~Filtrations in approximation theory.}
Consider an ergodic automorhism of a~measure space~$(X,\mu)$:
$T\colon X \rightarrow X$, $T\mu=\mu$. By the classical Rokhlin's theorem, it can be approximated in the uniform topology by periodic automorphisms. Periodic approximations can be amended in such a~way that the partitions into their trajectories decrease, i.\,e., form a~\textit{filtration of partitions into the trajectories of periodic approximations}.

It is this construction that leads to the so-called adic realization of an automorphism~\cite{67},~\cite{68}. Many important metric invariants of the automorphism can be extracted from the properties of this filtration; this will also be discussed below.

This example can be extended to amenable groups and amenable actions of arbitrary groups (see~\cite{37}). Consider a~dynamical system in a~space~$X$ with an invariant measure~$\mu$ in which ``time'' is a~countable amenable group~$G$. In other words, the group~$G$ is represented by automorphisms of the space $(X,\mu)$, i.\,e., measurable measure-preserving transformations:
\begin{align*}
T\colon G &\to \operatorname{Aut}(X,\mu),
\\
g &\mapsto T_g\colon X\to X,
\\
&\qquad\quad\;\;\mu\, \mapsto\, \mu.
\end{align*}
A~fundamental theorem of ergodic theory says that if the group~$G$ is amenable, i.\,e., has an invariant mean, then the orbit partition of~$G$ is tame (hyperfinite), i.\,e., there exists a~(non-unique) filtration that determines a~decreasing sequence of partitions which tends to the partition into the trajectories of the group action. There is a~particularly natural link between the theory of filtrations and the theory of actions of locally finite groups: in this case,
$G=\bigcup\limits_n G_n$ is an inductive limit of finite groups, and for every action of~$G$ there is a~canonically defined sequence of partitions into the trajectories of the finite subgroups~$G_n$.

Some invariants of filtrations constructed in this way, e.\,g., entropy  (see below), are also invariants of the action.

The adic realization of a~group action is a~kind of alternative to the generally accepted symbolic realization of this action by a~group of shifts in a~space of functions on the group. It can be regarded as a~generalization of the odometer to an arbitrary graded graph. This realization determines an approximation of the group action by finite groups, and thus it is a~far-reaching generalization of Rokhlin's lemma on periodic approximation.

\smallskip
\textbf{C.~Invariant and central measures.} \label{ssecC} Assume that we have an $\mathbb N$-graded
locally finite graph
$\Gamma=\displaystyle\coprod_n \Gamma_n$,
$\Gamma_0=\{\varnothing\}$, and let~$T(\Gamma)$ be the space of infinite paths in~$\Gamma$ starting from the vertex~$\varnothing$. The space~$T(\Gamma)$, being an inverse spectrum of finite spaces (of finite paths), is a~Cantor-like topological space, and hence is endowed with a~Borel structure. In~$T(\Gamma)$,  the
 \textit{tail filtration} is defined: the $n$th $\sigma$-algebra consists of all sets that do not change if we take an arbitrary path from such a~set and modify its initial segment up to level~$n$. If one endows  the space of paths with a~Borel measure, then the tail filtration turns into the filtration from Example~A above. But  here a~new problem arises, related to
\textit{measures with given cotransition probabilities}; namely, the problem of describing the central measures (see
Sec.~\ref{sec7}). This problem appeared in the theory of
$C^*$-algebras ($\operatorname{AF}$-algebras)
and locally finite groups (the list of traces and characters), in the theory of Markov processes and invariant Markov measures (``entrance and exit boundaries'' in the sense of Dynkin), in the theory of invariant measures for dynamical systems. This circle of problems has equivalent combinatorial and geometric statements interesting in their own right. Note that this class of examples of filtrations has important additional specific features, which we will use in what follows.

One of the results of this paper is as follows (Sec.~\ref{ssec4.1}).

\begin{theorem}\label{th1}
Every ergodic locally finite filtration is isomorphic to a~Markov filtration.
\end{theorem}

This means that one can study the most interesting class of filtrations by considering only one-sided Markov chains (in general, inhomogeneous in time and with an arbitrary state space). In particular, this provides a~fruitful reduction of the theory of $\operatorname{AF}$-algebras to the probabilistic theory of Markov chains, as well as the reverse transition from the theory of Markov chains to the theory of Bratteli diagrams with their algebraic underpinning.
Actually, this  theorem is a~refinement of the theorem on the existence of an adic realization for actions of amenable groups. It is also appropriate to mention that theorems on tail filtrations of path spaces of graded graphs are closely related to filtrations of commutants of finite-dimensional subalgebras of
 $\operatorname{AF}$-algebras.

\smallskip

\textbf{D.~Statistical physics and inverse limits of simplices.}
\label{ssecD}
Finally, the most popular reason for considering filtrations is provided by the modern formalism of statistical physics, both classical or quantum. Consider a~lattice, or even an arbitrary countable graph~$\Gamma$, and represent it as a~union of finite subsets (volumes):
$$
\Gamma=\bigcup_n \Gamma_n.
$$
Consider the space of configurations (e.\,g., subsets) on~$\Gamma$, i.\,e., the space of functions taking values in some alphabet (e.\,g., the alphabet $(0,1)$). Endow it with a~structure of a~Borel space by taking a~Borel base consisting of the cylinders, i.\,e., the sets of configurations determined by conditions on a~finite part of a~configuration (its restriction to some volume~$\Gamma_n$) and arbitrary outside this volume. Then we have a~filtration in this Borel space: its $n$th $\sigma$-subalgebra consists of the cylinder sets determined by conditions on the restriction of a~configuration to the volume~$\Gamma_n$. In other words, two configurations lie in the same element of the partition if they coincide outside~$\Gamma_n$ for some~$n$. If we have a~probability measure (statistic) on the space of configurations, then we obtain a~filtration of this measure space. A~popular scheme for studying phase transitions and other properties of statistics, adopted in recent years after the work by R.~L.~Dobrushin and others~\cite{5},~\cite{36}, is as follows: one defines not a~measure, but only a~system of conditional probabilities of the restriction of a~configuration to an arbitrary volume, under the condition that its restriction outside this volume is fixed, and studies the measures that have these conditional measures. This is a~method of defining measures in infinite-dimensional spaces alternative to the widely accepted Kolmogorov's method via joint distributions. This scheme for studying filtrations is briefly described below (see Sec.~\ref{sec3}, especially Subsections~\ref{ssec3.5} and
\ref{ssec3.6}).

We emphasize that the above list of areas related to filtrations does not pretend to be exhaustive, and this survey, as well as the list of references, is far from being complete. Especially interesting and little studied links are those with combinatorics; possible material on this topic can be drawn from~\cite{54},~\cite{55}.
We mean that an asymptotic combinatorial problem always involves an inductive family of problems with a~parameter (the number of objects, the collection of dimensions, etc.), but often one can also define a~filtration, whose properties are hidden quite deep, and without an analysis of these properties one cannot fully reveal the true nature of the problem. An example of such a~property is standardness, to be discussed below.

\subsection{Finite and infinite in classification problems; standardness as a~generalization of independence}
\label{ssec1.4}
In mathematics and physics there is a~lot of classification problems involving ``finite'' and ``infinite'' invariants, and usually the core of the classification problem is to find the latter ones. The problem under study, that of metric classification of filtrations, seems to be typical for the class of problems involving analytic, algebraic, and combinatorial components and, moreover, a~nontrivial relation between finite and infinite invariants. Perhaps, the ideas described here will be useful in other problems.

In what follows, classification and types of filtrations are understood in the sense of measure-preserving transformations, and equivalence is meant with respect to the group of such transformations. Instead of  ``equivalent objects (filtrations),'' we will more often say ``isomorphic objects.''

Every infinite filtration of algebras (or similar objects)
$$
\{{\mathfrak A}_n\}_{n=0}^{\infty}={\mathfrak A}_0 \supset
{\mathfrak A}_1\supset \cdots
$$
can be regarded as an infinite collection of its finite fragments, i.\,e., finite filtrations:
$$
\{{\mathfrak A}_k\}_{k=0}^{n},\qquad
n=0,1,2,\dots\,.
$$

Assume that we can give a~classification, with respect to some equivalence relation, of finite filtrations of given length~$n$ for all finite~$n$. More exactly, assume that we can construct a~topological or Borel  \textit{module space~${\mathscr M}_n$}, $n=0,1,2,\dots$, i.\,e., the space of complete invariants of finite filtrations of a~given length. We will say that two infinite filtrations are \textit{finitely isomorphic if for every~$N$ their fragments of length~$N$
are equivalent}, i.\,e., determine the same point in the space~${\mathscr M}_N$, $N=0,1,2,\dots$\;.

Of course, the fact that infinite filtrations are finitely isomorphic does not in general imply that they are isomorphic. In other words, even if the module space of infinite filtrations can be well defined, it is not in general the union of the module spaces of finite fragments of filtrations. We will consider this phenomenon in more detail below.

The main subject of this paper and the new results presented here are in one way or another related to the problem of classification of infinite filtrations and to selection of classes most important for applications. There exist infinite filtrations whose isomorphism type is uniquely determined by the finite isomorphism type. For example, if the group of symmetries of a~finite fragment of a~filtration is trivial, then the (isomorphism) type of this fragment already determines the type of the whole filtration.
Thus a~meaningful theory should deal only with the class of filtrations for which the group of symmetries of every infinite ``tail'' is infinite; partly because of this, we consider only locally finite filtrations (for a~definition, see Sec.~\ref{sec2}): for them, the corresponding groups are infinite, and this feature of these filtrations is already reminiscent of the Markov property. In the class of locally finite filtrations, finite isomorphism does not impose substantial restrictions on the ``infinite properties'' of filtrations, in particular, it does not even determine whether the filtration is ergodic or not.

Attempts to give~a classification of infinite sequences should be preceded by building a~classification of finite filtrations (i.\,e., finite decreasing sequences of subalgebras), and in the categories we are interested in, it either is known, or can easily be obtained. This classification was initiated by V.\,A.~Rokhlin (in the case of one $\sigma$-algebra or one partition), and one can use this pattern to construct invariants of finite filtrations (see, e.\,g.,~\cite{21}). We retell the classification of arbitrary finite decreasing sequences of partitions in the convenient language of graphs and trees. In these cases, the invariants, i.\,e., the module spaces, are quite manageable: these are spaces of finite measured trees. This is the content of a~part of Sec.~\ref{sec2}.

But what else should we add to finite isomorphism to obtain full isomorphism? In other words, what ``infinite invariants''  are not determined by finite ones? And do we really always need them? Are there classes of filtrations for which finite isomorphism already implies isomorphism? Note that the analysis of finite invariants does not allow one to deduce whether or not they suffice to determine the type of a~filtration.

This is one of the main questions, which is rather typical for any discussions about finite and infinite.

The answer to the above question is positive: the desired class is the class of so-called
 \textit{standard filtrations}. In the case of homogeneous (e.\,g., dyadic) filtrations, it consists, as follows immediately from the standardness criterion (see~\cite{59} and Sec.~\ref{sec4} below), of Bernoulli filtrations, i.\,e., past filtrations of  arbitrary sequences of independent random variables.

For general locally finite filtrations, even the statement of the problem is not so obvious and reads as follows. Let us try to generalize the problem of classification of filtrations, i.\,e., endow filtrations with a~certain structure in such a~way that the behavior of finite invariants of a~filtration in the augmented problem allows one to determine whether the ordinary finite invariants uniquely determine the type of this filtration.

Filtrations for which this condition is satisfied will be called \textit{standard}: this class generalizes the notion of independence (\,= Bernoulli property).

But how do this generalized problem and these generalized finite invariants look? The answer is that one should construct a~classification of finite fragments of filtrations  together with some additional data; for instance, this can be fixed measurable functions defined on the base space, but the most appropriate choice is to take a~\textit{metric on the measure space.} Then the measured trees which are invariants of finite fragments of filtrations will be also endowed with a~metric, and these new terms allow one to state the  \textit{criterion of asymptotic behavior of invariants}, called the \textit{standardness criterion};  the validity of its condition is exactly equivalent to the fact that the finite invariants completely determine the type of a~filtration. Moreover, the choice of a~specific metric does not affect the validity of this condition, it should only be admissible (see \cite{92},~\cite{93}). This is precisely the main idea, which we describe in Sec.~\ref{sec5}.
Thus the module space is exactly the space of finite invariants, and the standardness criterion is the condition meaning that the invariants of the conditional filtrations (on the elements of the $n$th partition) converge in measure as $n$ tends to infinity. Invariants of finite filtrations whose all partitions have finite elements  are measures on spaces of trees endowed with a~measure and a~metric. Below we also state this criterion in terms similar to those of martingale theorems.

At first sight, it seems that one can continue this procedure of selecting good classes, extending the problem of classification of filtrations and obtaining further classes subject to classification. But this is not the case: outside the standardness class there is essentially one class, indivisible and very interesting, that of \textit{totally nonstandard filtrations}, for which there is still no complete classification (and perhaps it does not exist), but there are nontrivial invariants, both resembling old invariants, such that scale and (secondary) entropy, and, possibly, new ones, such as ``higher zero--one laws,'' which provide far-reaching generalizations of Kolmogorov's laws. Perhaps, the terminology may change, because standardness can be understood both in the hard sense (as minimality) and in the mild sense (as proper standardness); the future will show a~convenient terminology.

It is interesting to compare this with Ornstein's theory of the very weak Bernoulli property (VWB),
$d$- and $f$-metrics, etc.; indeed, the VWB property is similar to standardness, the difference being only in the choice of
metrics on the space of conditional measures, as well as in the fact that ergodic theory deals only with standard filtrations.

We show that finite classification of filtrations is a~tame problem, i.\,e., the corresponding module space is a~reasonable Borel space. But the general problem of metric classification of arbitrary filtrations is ``wild'' in the accepted sense and seems to be just as difficult as the problem of classification of automorphisms of spaces and similar problems.

\subsection{A~summary of the paper}
\label{ssec1.5}

This is mainly a~survey paper which partially summarizes the research on filtrations from the introduction (in 1969) of the notion of standardness (more exactly, the discovery of examples of non-Bernoulli filtrations finitely isomorphic to Bernoulli ones) and the standardness criterion. We omitted several important topics directly related to the theory of filtrations, e.\,g., the notion of ``cosy'' filtrations in the sense of Tsirelson, recent papers by S.~Laurent, etc.  Our main purpose is, first, to extend basic results on standardness to inhomogeneous and non finitely Bernoulli filtrations and, second, to combine the theory of filtrations with the theory of graded graphs (Bratteli diagrams), and thus with the theory of $\operatorname{AF}$-algebras, combinatorics of graphs and Markov chains. Hence our main interest is in locally finite and locally representable filtrations, i.\,e., filtrations in which fibers (elements of partitions) are finite for every partition and the number of different types of conditional measures is also finite. A~detailed introduction, including Rokhlin's theory of one partition, is given in Sec.~\ref{sec2}. Links with the theory of graded graphs and Markov chains are described in Secs.~\ref{sec3} and~\ref{sec4}. It should be mentioned that the existence of these links enriches both the theory of filtrations and the theory of graphs, Markov chains, and algebras. Examples that have appeared in recent years brilliantly demonstrate this. There have appeared many interesting and complicated graphs originating from the theory of adic (i.\,e., ``graphic'') approximations of dynamical systems (the graphs of ordered and unordered pairs, graphs of words, etc.) which provide examples of nonstandard filtrations;  on the other hand, graph theory is a~source of new problems related to algebras and new realizations of filtrations. The theorem on a~Markov realization of a~locally finite filtration is proved in Sec.~\ref{sec4}. The key notion of the theory is standardness. For homogeneous and finitely Bernoulli ergodic filtrations, it coincides with independence, i.\,e., a~standard filtration is the ``past'' of a~Bernoulli scheme. For general locally finite filtrations, this notion is more sophisticated, and we give its definition and a~criterion for testing it. In particular, we state the standardness criterion in terms of random processes, namely, as a~strengthening of the martingale convergence theorem. This  is a~theorem on diagonal (in a~certain sense) convergence. There appears a~scale of strengthenings, which we inteprete as a~scale of zero--one laws: at one end of this scale, we have Kolmogorov's zero--one law, and on the other end, standardness or (in the homogeneous case) Bernoulli property. These questions are discussed in Sec.~\ref{sec5}. It should be said at once that in the case of general filtrations one can define different degrees of ``nonstandardness,'' or different degrees of closeness to standard filtrations. Nonstandardness in the homogeneous (e.\,g., dyadic) case is of most interest. In Sec.~\ref{sec6} we analyze several key examples, in particular, a~scheme of random walks in random scenery. But this work is still in progress, so this section is of preliminary nature. Finally, Sec.~\ref{sec7} deals with filtrations in Borel spaces and the problem of enumerating all measures with a~given equipment. This subject has wide applications, we describe its links with the theory of random walks on groups and the theory of traces of algebras and characters of groups.

I~considered it useful to supplement the survey with several historical remarks on the interesting history of the theory of filtrations.

\section{The definition of filtrations in measure spaces}
\label{sec2}

\subsection{Measure theory: a~Lebesgue space, basic facts}
\label{ssec2.1}
We consider a~certain category of measure spaces, that of  \textit{Lebesgue spaces in the sense of Rokh\-lin~\cite{45}},  or standard measure spaces. A~Lebesgue space is a~space endowed with a~separable $\sigma$-algebra of sets that is complete with respect to some (and hence every) basis; on this $\sigma$-algebra, a~probability $\sigma$-additive measure is defined. The separability of a~$\sigma$-algebra means the existence of a~countable basis of measurable sets that separates the points of the space; the definition of a~basis includes the condition that an arbitrary measurable set~$A$ should be contained in some set~$B$ from the Borel hull of the basis with $\mu(B\setminus A)=0$ (the weaker condition $\mu(B\bigtriangleup A)=0$ is not sufficient; we do not dwell on this here, see~\cite{45}).

All objects and notions, such as spaces, bases, completeness, sets, functions, partitions, maps, etc., should be understood $\operatorname{mod} 0$, i.\,e., up to modifications on sets of measure zero. A~Lebesgue space is a~class of Lebesgue spaces coinciding $\operatorname{mod} 0$, a~measurable function is a~class of functions coinciding $\operatorname{mod} 0$, etc.

Such an agreement makes it necessary to check that all notions, definitions, statements, proofs are well defined with respect to $\operatorname{mod} 0$, the fact that is usually ignored by all authors; however,  experienced authors have never got into trouble because of this neglect.

Morphisms in the category of Lebesgue spaces are measurable (in a~clear sense) measure-preserving maps, more exactly, classes (in the image and inverse image) of maps coinciding $\operatorname{mod} 0$. In more detail, a~measurable map of a~Lebesgue space into an arbitrary ``measurable'' space,   i.\,e., a~space endowed with a~$\sigma$-algebra (for instance, into a~standard Borel space with countable base, the interval $[0,1]$), is a~map such that the inverse image of a~measurable set is measurable. If the image space is also endowed with a~measure and the measure of an image is equal to the measure of the inverse image, then  the map is called measure-preserving. If the same is true for the inverse map, then it is called an isomorphism (in the case where the image and the inverse image coincide, an automorphism) of measure spaces.

Note that the image of a~Lebesgue space under a~measurable map is a~Le\-bes\-gue space (cf.\ the corresponding theorem for compact sets).

Two objects corresponding to each other under an isomorphism and the inverse isomorphism are called isomorphic (and sometimes equivalent).

The following classification theorem due to Rokhlin refined an earlier von Neumann's~\cite{96} theorem, which was of the same type, but was based on a~more complicated axiomatics of measure spaces.

\textit{There exists a~unique, up to isomorphism, Lebesgue space with continuous (i.\,e., having no atoms) probability measure}. The interval $[0,1]$ endowed with the Lebesgue measure is a~universal example of a~Lebesgue space with continuous measure. Every separable compact space with continuous Borel probability measure is another such example. The most convenient universal example of a~Lebesgue space is the countable product of two-point spaces, i.\,e., the dyadic compactum with an arbitrary continuous Borel measure. Thus, in the category of Lebesgue spaces, the most interesting objects (with continuous measure) are isomorphic.

Note that the construction of the theory of Lebesgue spaces in the sense of Rokhlin in~\cite{45} essentially reproduces the theory of metric compacta and, even more specifically, of the dyadic compactum, with necessary modifications taking into account the specifics of measure theory.

Passing to classification of general Lebesgue spaces, i.\,e., those with an arbitrary measure, consists in adding at most countably many atoms of positive measure.

Thus a~complete (metric) invariant of a~general Lebesgue space is an ordered sequence of nonnegative numbers
$$
m_1\geqslant m_2 \geqslant m_3 \geqslant \dots \geqslant 0,\qquad
\sum_i m_i \leqslant 1,
$$
namely, the sequence of sizes of all atoms, i.\,e., the points of positive measure.

The zero sequence corresponds to a~space with continuous measure, and a~sequence with unit sum,
 $\displaystyle\sum_i m_i=1$, to a~space with atomic measure.

\subsection{Measurable partitions, filtrations}
\label{ssec2.2}
A~measurable partition~$\xi$ (see Fig.~\ref{fig1}) of a~Lebesgue space $(X,\mu)$ is the partition of~$X$ into the inverse images of points under a~measurable map
$f\colon X\to R$, where $R$ is an arbitrary standard Borel space, for instance (without loss of generality), the interval
 $[0,1]$. The fact that a~set~$C$ is an element (block) of a~partition~$\xi$ is denoted as follows: $C\in\xi$.
Partitions are also regarded  $\operatorname{mod} 0$.

\begin{figure}[h!]
\centering
\includegraphics{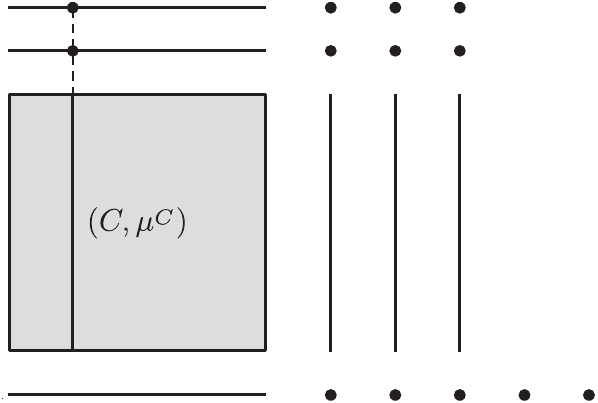}
\caption{A~measurable partition}
\label{fig1}
\end{figure}

A~fundamental fact, which goes back to works on probability theory of the 1930s (J.~Doob and others), but was rigorously proved by V.\,A.~Rokhlin at the initiative of A.\,N.~Kolmogorov (some hints on the necessity of a~rigorous study of conditional measures are contained in his famous book~\cite{33}), is as follows.

\begin{theorem}[{\rm(V.\,A.~Rokhlin~\cite{45})}]
\label{th2} Every measurable partition~$\xi$ of a~Lebesgue space
$(X,\mu)$ can be equipped with a~unique $\operatorname{mod} 0$ canonical system of measures~$\{\mu^C\}$, $C\in \xi$ (in a~more popular language, a~system of conditional measures), on almost all elements of~$\xi$ in such a~way that

1) almost all, with respect to the measure $\mu$, elements $(C,\mu^C)$ endowed with the conditional measures are Lebesgue spaces; the quotient space $X/\xi$ endowed with the quotient measure~$\mu_{\xi}$ is also a~Lebesgue space;

2) for every measurable set $A \in X$, the function  $\mu^C(A\cap C)$ is measurable (as a~function on
 $(X/\xi,\mu_{\xi})$), and the following (Fubini, repeated integration) formula holds:
$$
\mu(A)=\int_{X/{\xi}}\mu^C( A\cap C)\,d\mu_{\xi}(C).
$$
\end{theorem}

Recall that a~measurable partition~$\xi$ gives rise to a~unique $\sigma$-algebra~${\mathfrak A}_{\xi}$ of sets measurable with respect to~$\xi$, i.\,e., composed of elements of~$\xi$. In Lebesgue spaces, the converse is also true: every $\sigma$-subalgebra~${\mathfrak A}$ of measurable sets in a~space
$(X,\mu)$ gives rise to a~unique ($\operatorname{mod} 0$) measurable partition~$\xi$, and ${\mathfrak A}_{\xi}=\mathfrak A$.

In the spaces~$L^1$ and~$L^2$, we have the \textit{operator of conditional expectation} corresponding to a~measurable partition~$\xi$,
i.\,e., the projection~$P_{\xi}$ (in~$L^2$, the orthogonal projection) to the subspace of functions constant on elements of~$\xi$; this is the operator of averaging over the elements of~$\xi$ endowed with the conditional measures.

From an alternative point of view, the theorem on conditional measures is simply a~theorem on an integral representation of the operator of conditional expectation:
$$
(P_{\xi}f)(C)=\int_C f(x)\,d\mu^C(x).
$$

The left-hand side can be understood as a~function on the quotient space, but also as a~function on~$X$ belonging to the subalgebra of functions constant on elements of~$\xi$.

The projection to this subalgebra is well defined, and the proof of the existence of an integral representation  usually proceeds as  follows: first one endows the space~$X$ with a~compact topology, or even a~metric, then considers the projection to the space of continuous functions, and, finally, approximates functions in $L^1(X,\mu)$ by continuous functions.

Introducing a~partial order on the classes of partitions coinciding $\operatorname{mod} 0$ turns the collection of partitions into a~lattice, where $\xi\succ \xi'$ means that elements of~$\xi$ are finer than elements of~$\xi'$, i.\,e., that an element of~$\xi'$ is composed of elements of~$\xi$; it is this order that agrees with the order of inclusion on the spaces of measurable functions associated with partitions.
The greatest partition is the partition into singletons $\operatorname{mod} 0$, and the smallest one is the trivial partition (with only one nonempty component). Note that this order is opposite to that adopted in combinatorics, where the greatest partition is the trivial one and the smallest partition is that into singletons.

\subsection{Classification of measurable partitions}
\label{ssec2.3} The purpose of the upcoming sections is to describe metric invariants (i.\,e., give a~classification) of partitions, and then of finite filtrations, i.\,e., finite decreasing sequences of measurable partitions. The case of one partition is investigated in the well-known paper~\cite{45}, and we cite the corresponding result (with a~slightly modified statement) only for convenience of further generalizations. Note that this is perhaps the main original result of~\cite{45}.

\begin{theorem}[{\rm(V.\,A.~Rokhlin~\cite{45})}]
\label{th3}
A~metric invariant of a~measurable partition~$\xi$ in a~continuous Lebesgue space $(X,\mu)$ is a~metric invariant of the map
$$
V_{\xi}\colon(X/\xi,\mu_{\xi}) \to \Sigma,
$$
where $(X/\xi,\mu_{\xi})$ is the quotient space, $\mu_{\xi}$ is the quotient  of the measure~$\mu$ under the projection $X \to X/\xi$, and the map~$V_{\xi}$ sends every element $C \in \xi$ to the ordered sequence of measures of the atoms $m_1(C)\geqslant
m_2(C)\geqslant\cdots$, regarded as a~point of the simplex~$\Sigma$ of all ordered series of nonnegative numbers with sum at most~$1$.
\end{theorem}

Thus, two measurable partitions of Lebesgue spaces are isomorphic, i.\,e., correspond to each other under an isomorphism of measure spaces, if and only if the corresponding maps~$V$
are metrically isomorphic.

But we can go further and ask about the metric invariants of the map~$V_{\xi}$. The answer is as follows: these invariants are the $V_{\xi}$-image of the measure~$\mu_{\xi}$,
which is a~Borel measure on the simplex~$\Sigma$, and the so-called multiplicities (we will not need them, so we refer the interested reader to the paper~\cite{70}, where the invariants of measurable maps are described in detail, in generalization of V.\,A.~Rokhlin's paper~\cite{46} on classification of functions).

Speaking rather loosely, a~measurable partition is a~random measure space whose all realizations are regarded as disjoint subsets in the same space.

Before turning to the case of several partitions, we describe a~number of classes and examples of partitions.

1. \textit{Partitions that have independent complements}.
Assume that
 almost all elements~$C$ of a~partition~$\xi$ are isomorphic as spaces endowed with the conditional measures~$\mu^C$, i.\,e., the $V_{\xi}$-image of the measure~$\mu_{\xi}$ is the $\delta$-measure at the point of the simplex~$\Sigma$ corresponding to the (common) type of the conditional measures~$\mu^C$.
In this, and only this, case the partition~$\xi$ has an \textit{independent complement~$\xi^-$}, and the space~$(X,\mu)$ can be decomposed into a~direct product of measure spaces:
$$
(X,\mu)=(X/{\xi}, \mu_{\xi})\times (X/{\xi^-},\mu_{\xi^-}),
$$
where the partition~$\xi^-$ is the unique (up to isomorphism, but not geometrically) independent complement to~$\xi$. Almost every element of~$\xi$ has a~nonempty intersection with almost every element of~$\xi^-$, and the independence condition is satisfied:
$$
\mu(A\cap B)=\mu_{\xi}(A)\mu_{\xi^-}(B)
$$
for any sets~$A$ and~$B$ measurable with respect to~$\xi$
and~$\xi^-$, respectively.

\begin{corollary}[{\rm(V.\,A.~Rokhlin~\cite{45})}]
\label{cor1}
A~Lebesgue space with continuous measure has a~unique, up to isomorphism, partition with continuous conditional measures.
\end{corollary}

This is the partition of the square $[0,1]^2$ endowed with the Lebesgue measure into vertical intervals. A~(non-unique) independent complement is the partition into horizontal intervals.

2.\;\textit{Finite partitions}. Partitions with finitely many elements (blocks) are called finite; the complete invariant of finite partitions, in the case where $(X,\mu)$ is a~continuous measure space, is the collection of measures of the elements; in this case,
$V_{\xi}\mu=\delta_{0}$, where $0\in\Sigma$ is the zero sequence.

3.\;\textit{Atomic partitions}. Let~$\xi$ be a~partition whose almost all elements are finite, more exactly, almost all measures~$\mu^C$ are atomic measures with finitely many atoms; such partitions are called discrete (though it would be more appropriate to call them ``cofinite''). A~measurable partition~$\xi$ of a~Lebesgue space $(X,\mu)$ with continuous measure will be called \textit{atomic} if the number of metric types of conditional measures is finite; in other words, if the $V_{\xi}$-image of~$\mu_{\xi}$ is a~measure on~$\Sigma$ with finitely many atoms.

4.\;A~\textit{semihomogeneous partition} is an atomic partition with finite blocks and uniform conditional measures on almost every block. The invariant of a~semihomogeneous partition is the collection of positive integers equal to the numbers of points in the blocks and the measures of the sets of all blocks with a~given number of points. A~\textit{homogeneous} (dyadic, triadic, etc.)\ partition is a~partition with the same number of points in all blocks; in this case, the
 $V_\xi$-image of $\mu_\xi$ is the $\delta$-measure at the point
 $$
(1/n,\dots (n) \dots,1/n)\in \Sigma.
$$

\subsection{Classification of finite filtrations}
\label{ssec2.4}
The study and classification of (finite or infinite) collections of $\sigma$-subalgebras, or, which is equivalent, measurable partitions, in Lebesgue spaces with continuous measures is perhaps the most general and difficult geometric problem of measure theory and probability theory. Already the case of two measurable partitions in general position is of great interest and involves combinatorial and algebraic difficulties; we will consider it elsewhere.

In this section, we first give a~classification of finite monotone sequences of measurable partitions, i.\,e., finite filtrations. It is quite simple. However, the main subject of the paper is the analysis of infinite filtrations.

The study of finite filtrations in a~space with continuous measure should be preceded by the study of such filtrations in the following finite measure space:
$$
(C,m),\qquad
C=\{c_1,c_2,\dots,c_k\},\quad
m(c_i)=m_i,\quad
i=1,\dots,k.
$$
A~filtration of length~$n$ in a~finite space is a~\textit{hierarchy}: $\xi_0$ is the partition of~$C$ into singletons; $\xi_1$ is a~partition of~$C$; $\xi_2$ combines some blocks of~$\xi_1$ into larger blocks,~etc.; the last partition~$\xi_{n-1}$ consists of a~single block, which is the whole space~$C$. It is convenient to think of a~hierarchy as a~tree  of rank~$n$, i.\,e., with~$n$ levels, endowed with a~measure (see Fig.~\ref{fig2}).

\begin{figure}[t!]
\centering
 \includegraphics{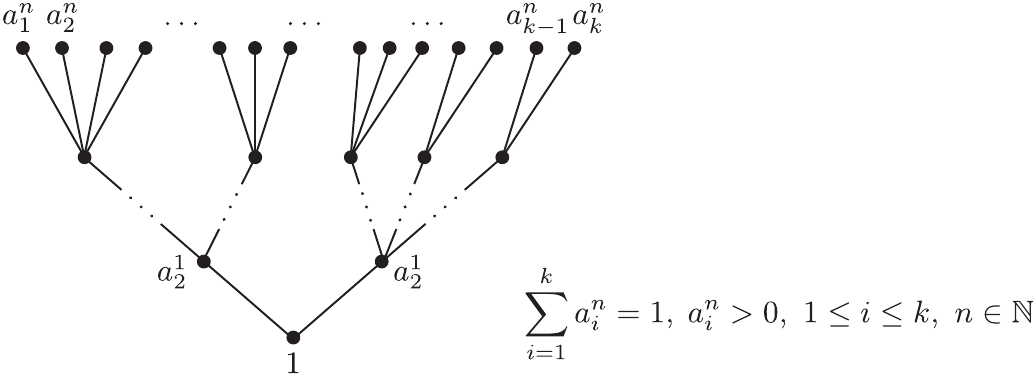}
\caption{A~finite filtration of a~finite measure space}
\label{fig2}
\end{figure}

For $n=1$, the tree has two levels: several vertices of the zero level and a~single vertex of the first level. The vertices of the zero level are linearly ordered according to the measures of the  points.

In the general case, the vertices of the~$r$th level are the elements of the partition~$\xi_r$, $r=0,1,\dots,n-1$, and an edge joining a~vertex of level $r-1$ with a~vertex of level~$r$ means an inclusion between the corresponding elements of partitions. The vertices of the zero level, i.\,e., the points of the original space, are ``leaves'' of the tree. They have the measures~$m_i$, $i=1,\dots,k$; let us introduce a~partial linear order on the leaves belonging to the same element of the partition~$\xi_1$ according to their measures (up to an arbitrary order on leaves of equal measure). Given these measures, one can recover, by a~simple summation, the conditional measures on the vertices of the first level, i.\,e., on the elements of the partition~$\xi_1$, and continue in the same way for the elements $\xi_2,\dots,\xi_{n-1}$ of the subsequent levels.

\begin{definition}
\label{def1} Denote by $\operatorname{Tree}_n$ the space of finite single-root trees of rank~$n$ (i.\,e., trees with~$n$ levels labeled by the numbers $0,1,\dots,n-1$ and a~unique vertex of the last level as a~root) endowed with a~strictly positive probability measure on the leaves, i.\,e., on the vertices of the zero level, and hence on the vertices of all levels; the values of the measure induce a~linear order on the vertices up to an arbitrary order of vertices with equal measure.
\end{definition}

One may say that $\operatorname{Tree}_n$ is the space of finite filtrations of length~$n$ on finite measure spaces.

In what follows, we will equip these trees with another important structure, a~metric, and the resulting category of trees endowed with a~measure and a~metric will be a~powerful tool in solving our problems.

Now we are ready to describe a~complete system of invariants of a~finite filtration of an arbitrary Lebesgue space. We consider only filtrations in which all measurable partitions are atomic:
$$
\{\xi_0\succ \xi_1 \succ \xi_2 \succ\dots\succ \xi_n\}=\{\xi_k\}_{k=1}^n,
$$
where $\xi_0=\varepsilon$ is the partition into singletons $\operatorname{mod} 0$.

\goodbreak

The fact that such a~sequence is decreasing means that almost every element~$C_n$ of the last partition~$\xi_n$ is finite and composed of finitely many elements of the partition~$\xi_{n-1}$, i.\,e.,
$$
C_n=\bigcup C^i_{n-1},\qquad C^i_{n-1}\in \xi_{n-1},
$$
the elements~$C^i_{n-1}$ are unions of elements of~$\xi_{n-2}$,
etc.

In other words, almost every element $C_n\in \xi_n$ supports a~hierarchy, i.\,e., a~finite filtration with a~measure. Thus we have a~homomorphism
$$
t_n\colon X/\xi_n \to \operatorname{Tree}_n,
$$
which sends almost every element $C_n\in \xi_n$, endowed with the conditional measures of its points, to the measured tree corresponding to the finite filtration on this element induced by the filtration on the whole space. Denote the $t_n$-image of the measure~$\mu_{\xi_n}$ on~$X/\xi_n$ by~$\omega_{\tau}$.

\begin{theorem}
\label{th4}
The complete system of isomorphism invariants of a~finite filtration
$\tau \equiv \{\xi_k\}_{k=0}^n$ of atomic partitions of a~Lebesgue space with continuous measure coincides with the system of invariants of the homomorphism~$t_n$ to the space~$\operatorname{Tree}_n$ of measured trees,
which consists of the measure~$\omega_{\tau}$ on~$\operatorname{Tree}_n$ obtained as the
$t_n$-image of the measure~$\mu_{\xi_n}$ on $X/\xi_n$ and the multiplicity function. If the filtration has finitely many types of conditional measures, then the complete invariant is simply the atomic measure~$\omega_{\tau}$ on the set
$\operatorname{Tree}_{n+1}$ of trees of rank $n+1$ (all multiplicities are continual).

In the case $n=1$ (i.\,e., the case of one atomic partition~$\xi_1$), this is a~measure on ordered probability vectors, in accordance with Theorem~\ref{th3}.
\end{theorem}

The proof that these are invariants is obvious, since the construction used only invariants of partitions, namely,  systems of conditional measures. The completeness follows from the fact that two nonisomorphic filtrations must differ at some set of positive measure that is measurable with respect to the partition~$\xi_n$, and hence the images of the corresponding measures on~$\operatorname{Tree}_{n+1}$ do not coincide. Repeating what has been said in the case of one partition, one may think of the invariant of a~finite filtration as a~random tree endowed with a~measure.

It is not difficult to extend this procedure to classification of finite filtrations with arbitrary measurable (rather than only atomic) partitions: one should pass from finite trees to countable, or even continual, trees and general measures on them.

Earlier, classification of finite filtrations was considered in different terms (see, e.\,g.,~\cite{21}); the language of trees or hierarchies proposed above was not used. The classification itself is not very interesting; it could be useful if one could apply it to classification of infinite filtrations. But one cannot simply construct the inductive limit of the spaces $\operatorname{Tree}_n$ with respect to natural embeddings and consider measures on the resulting space as invariants of filtrations, and a~direct attempt to pass to the limit is of no use for classification of infinite filtrations. The reason is that when we pass to the limit, the bijectivity of the correspondence between classes of filtrations and finite invariants
is lost. In fact, the information related to finite filtrations useful for classification of infinite filtrations is somewhat different: finite fragments of a~filtration should be considered not by themselves, as we did above, but together with some functions or metrics, as we will do below.

\subsection{Filtrations we consider and how one can define them}
\label{ssec2.5}

\subsubsection{Classes and properties of filtrations}
\label{sssec2.5.1}
First of all, we mention the most interesting class of filtrations (to be considered in more detail in another paper). This is the class of filtrations in which all partitions have purely continuous conditional measures. Is is of great importance for random processes and the theory of  $C^*$-algebras, but it is not related to combinatorial problems.

The key notion is that of ergodicity of filtrations, called also Kolmogorov property,
regularity, etc.

A~filtration
$$
\tau=\{{\mathfrak A}_n\} \simeq \{\xi_n\} \simeq
\{L^{\infty}(\xi_n)\}
$$
is called \textit{ergodic} if
\begin{enumerate}
\item[--] the intersection of the $\sigma$-algebras~${\mathfrak A}_n$ is the trivial $\sigma$-algebra~${\mathfrak N}$:
$$
\bigcap_n {\mathfrak A}_n={\mathfrak N},
$$
\end{enumerate}
or
\begin{enumerate}
\item[--] the measurable intersection of the partitions~$\xi_n$ is the trivial partition
$$
\bigwedge_n \xi_n=\nu,
$$
\end{enumerate}
or
\begin{enumerate}
\item[--] the intersection of the algebras of functions constant on elements of the partitions~$\xi_n$ over all~$n$ consists of the constants:
$$
\bigcap_n L^{\infty}(X/\xi_n)=\{\operatorname{Const}\}.
$$
\end{enumerate}

In our further considerations, we deal with filtrations in which all partitions have atomic conditional measures without continuous components (filtrations with continuous conditional measures are left for another paper).

We distinguish the following classes: \textit{locally finite filtrations, homogeneous filtrations, and semihomogeneous filtrations}.

\begin{definition}
\label{def2} (i) A~filtration $\tau=\{\xi_n\}_{n=0}^{\infty}$
of a~Lebesgue space $(X,\mu)$ with continuous measure~$\mu$ is called
\textit{locally finite} if all partitions~$\xi_n$,
$n=0,1,2,\dots$, are atomic, i.\,e., almost all conditional measures on the elements of~$\xi_n$ are atomic, and the number of types of conditional measures is finite, in particular, for every~$n$ the number of atoms in almost all elements of~$\xi_n$ is finite (but, possibly, depends on~$n$).

(ii) A~locally finite filtration is called \textit{semihomogeneous} if the conditional measures of almost all elements of all partitions~$\xi_n$,
$n=0,1,2,\dots$, are uniform distributions (i.\,e., the measures of all points of a~given element are equal).\footnote{But the conditional measures of elements of the quotient partitions $\xi_n/\xi_{n-1}$ are not necessarily uniform.}

(iii) A~locally finite filtration is called \textit{homogeneous} if for every~$n$ the conditional measures of almost all elements of~$\xi_n$ are uniform and equal. If the number of elements in~$\xi_n$ is~$r_n$,
then the filtration is called
$\{r_n\}_n$-adic; for $r_n\equiv 2$, dyadic, and for
$r_n\equiv 3$, triadic.

\end{definition}

Relaxing the finiteness condition in the definition of a~locally finite filtration (for example, allowing elements of partitions to have countably many atoms or an arbitrary number of types of conditional measures) does not raise serious new difficulties and essentially new phenomena in the theory of filtrations, thus we restrict ourselves only to locally finite filtrations.

Consider also the following special classes.

\begin{definition}
\label{def3}
We say that a~filtration~$\tau$ is

(a) a~\textit{Bernoulli} filtration  if it is the tail filtration for an arbitrary sequence of independent random variables; in other words, if there exists a~sequence of independent partitions~$\eta_n$, $n=1,2,\dots$, in a~Lebesgue space with continuous measure such that
$$
\xi_n=\bigvee_{k=n}^{\infty} \eta_k,\qquad
n=0,1,2,\dots\,;
$$

(b) more generally, a~\textit{Markov} filtration if it is the tail filtration of a~one-sided  Markov chain with discrete time and arbitrary state space.
\end{definition}

Let us turn our attention to the following important class of Markov filtrations which generalizes random processes with independent increments.

Consider a~group~$G$ and  a~finitely supported probability measure~$m$ on~$G$. The Markov chain~$y_n$, $n=0,1,2,\dots$, corresponding to the (say, left) random walk on~$G$ with transition probability~$m$ generates a~Markov filtration. The ergodicity of this filtration is equivalent to the triviality of the Poisson--Furstenberg boundary. And the triviality criterion for the boundary is that the entropy vanishes (see~\cite{26}).

More generally, assume that the group~$G$ acts on a~Lebesgue space $(Z,\nu)$
with continuous (or even $\sigma$-finite) invariant  measure~$\nu$:
$$
\forall\,g \quad
T_g\nu=\nu.
$$
 This gives rise to a~one-sided Markov process~$\{z_n\}_n$ with the state space
$(Z,\nu)$ and transition probabilities
$$
\operatorname{Prob}(z|u)=m(g),\quad\text{where } z=T_gu,
$$
and the corresponding filtration, which will be considered in the upcoming sections. If~$m$ is the uniform measure on the generators of the group, then the filtration is homogeneous. Note that the state space of this process is continual, but the filtration is locally finite.

As we will see in the next section, in this case the filtration can be defined by a~sequence of finite partitions
$$
\eta_n,\quad
n=1,2,\dots,\qquad
\bigwedge_n \eta_n=\epsilon
$$
(where $\epsilon$ is the partition into singletons).
The filtration~$\{\xi_n\}_n$ is defined as follows:
$$
\xi_n=\bigvee_{k=n+1}^{\infty} \eta_k,\qquad
n=0,1,2,\dots\,.
$$

\subsubsection{The basis method of defining filtrations}
\label{sssec2.5.2}
The most efficient method of defining filtrations, both in measure  and Borel spaces, is to introduce first an increasing sequence of finite partitions (a~basis),
and then define a~filtration as the sequence of products of ``tails.'' We will call it the basis method.

Consider a~sequence of finite partitions (a~basis)
$$
\{\eta_n\}_{n=1}^{\infty},\qquad
\bigwedge_n \eta_n=\epsilon
$$
(where $\epsilon$ is the partition into singletons) of a~measure or Borel space.
A~filtration~$\{\xi_n\}_n$ is defined as follows:
$$
\xi_n=\bigvee_{k=n+1}^{\infty} \eta_k,\qquad
n=0,1,2,\dots\,.
$$

The basis method of defining filtrations naturally arises in connection with random processes with time~${\mathbb Z}_+$ and in other probabilistic situations (for example, $\eta_n$ is the partition corresponding to the $\sigma$-algebra of sets determined by the state of the process at time~$n$). We will use it systematically in Sec.~\ref{sec3} when defining tail filtrations on path spaces of graphs.

We emphasize that geometrically one and the same filtration can be defined via different bases, and the problem of (metric or Borel) isomorphism of filtrations defined by the basis method is quite nontrivial. Moreover, the conditional measures of the partitions~$\xi_n$ (i.\,e., equipments, see Sec.~\ref{ssec3.3}) in the case of measure spaces can be defined only by passing to the limit.

The relationship between two sequences of partitions, the increasing one (a~basis),
$$
\{\zeta_n\}_n:\quad
\zeta_n=\bigvee_{k=1}^n\eta_k, \quad
n=1,2,\dots,
$$
and the decreasing one (a~filtration),
$$
\{\xi_n\}_n:\quad
\xi_n=\bigvee_{k=n+1}^{\infty} \eta_k,\quad
n=0,1,2,\dots,
$$
which look as dual objects, is by no means symmetric: the theory of filtrations is much deeper  than the theory of increasing sequences, and the most important difference is in the asymptotic behavior. Namely, passing to the limit in the theory of filtrations is much finer than in the case of increasing sequences. In Sec.~\ref{sec8}, we show an example  from A.\,N.~Kolmogorov's paper with a~commentary by V.\,A.~Rokhlin, which demonstrates the absence of continuity of the product when passing to the limit in a~filtration. This fact distinguishes filtrations of subalgebras of functions from filtrations of subspaces in a~Hilbert space, where there is a~complete symmetry between increasing and decreasing sequences. The reason, of course, is  that the lattice of subalgebras (partitions), in contrast to the lattice of subspaces, has no canonical involution similar to passing to the orthogonal complement, which sends increasing sequences to decreasing ones and vice versa.

\subsection{The equivalence relation associated with a~filtration, cocycles and associated measures}
\label{ssec2.6}
Along with the canonical intersection of $\sigma$-algebras, which turns out to be trivial for ergodic filtrations, there is another set-theoretic intersection, and the corresponding well-defined equivalence relation on the measure space itself. Assume that we have a~locally finite filtration (or even just a~filtration with discrete elements) $\tau=\{\xi_n\}_n$.
Consider the monotone limit of the measurable partitions~$\xi_n$;
denote it by $\bigcap\limits_n \xi_n={\overline \xi}_{\tau}$
(as opposed to the measurable intersection $\bigwedge \xi_n$).

\begin{proposition}
\label{pr1} The partition $\bigcap\limits_n \xi_n$ is well defined as the monotone limit of the decreasing sequence of partitions~$\{\xi_n\}_n$. It is not measurable, since in the ergodic case there are no nonconstant measurable functions constant on  elements of this partition. Nevertheless, there is a~canonical object corresponding to this partition in a~function space (e.\,g., in~$L^2(X,\mu)$):
\begin{equation}
\label{eq2.1}
{\mathfrak C}=\biggl\{f\in L^2(X,\mu):\exists n, \
\int_{C} f(x)\,d\mu^C(x)=0\mbox{  for a.\,e.\  }C\in\xi_n\biggr\}.
\end{equation}
Clearly, the integral in~\eqref{eq2.1} vanishes also for all $m>n$. The linear space
${\mathfrak C}\subset L^2(X,\mu)$ is not closed, its closure is the orthogonal complement to the intersection
$\bigcap\limits_n L^{\infty}(X/\xi_n,\mu_{\xi_n})$ (in the ergodic case, the closure is the space of all functions with zero integral).
\end{proposition}

\begin{proof}
1. If we modify all partitions~$\xi_n$, $n=1,2,\dots$, at a~set of measure zero, then
${\overline \xi}_{\tau}$ will also change only on a~set of measure zero (it is important here that the partitions~$\xi_n$ decrease and their elements are finite).

2. The space~${\mathfrak C}$ is uniquely determined by the filtration and does not change under automorphisms of this filtration.

The proposition is proved.
\end{proof}

The space~${\mathfrak C}$ can be extended to an invariant functional analog of the intersection~${\overline\xi}_{\tau}$ itself rather than the filtration. To this end, in the definition of~${\mathfrak C}$ one should take functions~$f$  for which the zero integral condition is satisfied for almost all elements~$C$ of arbitrary measurable partitions greater than the intersection: $\xi\succ {\overline
\xi}_{\tau} $. This condition no longer changes when a~filtration is replaced with another filtration having the same intersection of partitions. This space is a~functional analog of a~tame partition.\footnote{This observation was made by the author in~\cite{88},~\cite{83}. Clearly, classification of tame partitions is equivalent to classification of nonclosed subspaces of the above form with respect to automorphisms, that is, unitary real multiplicative operators.}

The most important property of the intersection of partitions is that its (countable) elements support a~projective conditional measure, namely, for almost all pairs of points~$x$,~$y$ of its elements, the ratio of conditional measures $\dfrac{\mu^C(x)}{\mu^C(y)}$\,, where~$C$ is an element of an arbitrary measurable partition $\xi \succ \overline \xi$, does not depend on the choice of~$\xi$. This ratio is a~well-defined cocycle, which will be discussed in more detail in the next section (see Sec.~\ref{ssec3.5}).  In the case where~$\xi$ is the orbit partition for a~transformation with quasi-invariant measure (or for an amenable group), this cocycle is called the Radon--Nikodym cocycle. See also~\cite{94}.

The paradoxical role of the cocycle defined on the intersection of measurable partitions (on the tail equivalence relation) and constructed from a~given measure~$\mu$ on~$X$ (in the next section, a~measure is constructed on paths in a~graph) is that it can simultaneously be a~cocycle for other measures singular with respect to~$\mu$. Thus it turns out to be related to another notion $\operatorname{mod} 0$ and, consequently, to other classes of functions coinciding $\operatorname{mod} 0$. In other terms, this paradox can be stated as follows: let
$\tau=\{\xi_n\}_n$ be a~filtration; then the conditional measures of the partitions~$\xi_n$ for all~$n$ do not  uniquely determine the measure~$\mu$, there can exist other \textit{associated measures} that are singular with respect to~$\mu$, but have the same conditional measures with respect to the partitions~$\xi_n$.

The notion of cocycle and the problem of describing all measures with a~given cocycle appeared in the theory of Markov processes (works of E.\,B.~Dynkin~\cite{9, 10, 11}), in ergodic theory
(K.~Schmidt~\cite{52}), in the theory of graded graphs (works of the author and S.\,V.~Kerov, see~\cite{86}, and their followers). The problem of constructing all measures with a~given cocycle is a~far-reaching generalization of the traditional problem of describing all invariant measures for group actions. For more details, see~\cite{82} and Sec.~\ref{sec6}.

\section{Tail filtrations in graded graphs and Markov chains}
\label{sec3}

In this section we introduce  a~new circle of notions, which brings into the theory of filtrations a~substantially different point of view on basic notions and provides it with an enormous amount of examples. We use $\mathbb N$-graded graphs and multigraphs (Bratteli diagrams); the theory of such graphs and their links with
$C^*$-algebras were intensively studied in the 1970s--1980s (by O.~Bratteli, G.~Elliott, E.~Effros, D.~Handelman, and others, see~\cite{18},~\cite{86} and the references therein). New connections of this theory to dynamics and ergodic theory were discovered by the author in the 1970s and became a~source of new problems both in representation theory and dynamics; in part, they became the subject of asymptotic representation theory developed in the same years by the author together with S.\,V.~Kerov and their followers.

In this survey and several previous papers by the author, attention  is drawn, apparently for  the first time,  to the role of filtrations of $\sigma$-algebras and filtrations of subalgebras of  $C^*$-algebras in the analysis of the structure of  algebras themselves. But here we are primarily interested in the measure-theoretic structure of filtrations appearing in connection with graded graphs.

\subsection{Path spaces of graded graphs (Bratteli diagrams) and Markov compacta}
\label{ssec3.1}
Consider a~locally finite, infinite, $\mathbb N$-graded graph~$\Gamma$ (in other words, a~Bratteli diagram). The set of vertices of~order~$n$, $n=0,1,2,\dots$, will be denoted by~$\Gamma_n$
and called the $n$th \textit{level} of the graph;
$$
\Gamma=\coprod_{n\in \mathbb N} \Gamma_n,
$$
where $\Gamma_0$ consists of the unique initial vertex~$\{\varnothing\}$. We assume that every vertex has at least one successor and every vertex except the initial one has at least one predecessor. In what follows, we also assume that edges of~$\Gamma$
are simple.\footnote{Considering multigraphs, i.\,e., graphs with multiple edges,  in general gives nothing new for our purposes, since cotransition probabilities (equipments) introduced below replace and generalize the notion of multiplicities of edges. However, in what follows we use the language of multigraphs to simplify some statements.}

A~graded graph~$\Gamma$ gives rise canonically to a~locally semisimple algebra~$A(\Gamma)$ over~$\mathbb C$,
however here we consider neither this algebra no the relationship of the notions introduced below with this algebra and its representations.

A~path is a~(finite or infinite) sequence of adjacent edges of the graph (for graphs without multiple edges, this is the same as a~finite or infinite sequence of adjacent vertices). The number of paths from the initial vertex to a~given vertex is finite, which means exactly that the graph is locally finite. The space of all infinite paths in a~graph~$\Gamma$ is denoted by~$T(\Gamma)$; in a~natural sense, it is an inverse limit of the spaces~$T_n(\Gamma)$ of finite paths from the initial vertex to the vertices of the $n$th level. Thus
$T(\Gamma)$ is a~Cantor-like compactum.

Cylinder sets in the space $T(\Gamma)$ are sets determined by conditions on initial segments of paths up to level~$n$; these sets are clopen and define a~base of a~topology. One can naturally define the notion of the \textit{tail equivalence relation~$\tau_{\Gamma}$} on~$T(\Gamma)$: two infinite paths are equivalent if they eventually coincide; we also say that these two paths lie in the same block of the tail partition.

The \textit{tail filtration} (for the moment, a~Borel one, without any measure)
$$
\Xi(\Gamma)=\{{\mathfrak A}_0 \supset {\mathfrak A}_1 \supset
\dotsb\}
$$
is the decreasing sequence of $\sigma$-algebras of Borel sets where the $\sigma$-algebra~${\mathfrak A}_n$, $n \in \mathbb N$, consists of all Borel subsets in~$T(\Gamma)$  that together with every path~$\gamma$ contain  all paths coinciding with~$\gamma$ up to level~$n$. In a~clear sense, the $\sigma$-algebra~${\mathfrak A}_n$ is complementary to the finite $\sigma$-algebra of cylinder sets of order~$n$.

\subsection{Trees of paths}
\label{ssec3.2}
Let us make a~general remark on the path space of a~graded graph: the paths joining a~given vertex with the initial one form a~tree whose root is the given vertex, and each path (a~leaf of the tree) is endowed with a~measure.

The connection between the tree constructed from the paths of a~graph and that constructed from a~measurable partition is obvious.

\begin{definition}
\label{def4}
Every vertex~$v$ of level~$n$ in a~graded graph determines in a~natural way a~tree of rank~$n$, which is composed of the paths leading from the initial vertex to~$v$; these trees agree in a~natural sense: the tree corresponding to a~vertex~$u$ is embedded as a~subtree into the tree corresponding to a~vertex~$v$ if $v$ succeeds~$u$. Leaves of the tree correspond to paths in the graph.
\end{definition}

\begin{proposition}
\label{pr2}
The measured trees  constructed in Sec.~\ref{ssec2.4} from finite fragments of an infinite filtration~$\tau$
are exactly the trees constructed from the vertices of the graph~$\Gamma(\tau)$. (The possibility of combining them into an inductive limit that determines the type of the filtration up to finite isomorphism was discussed above. However, we do not need this, since the type of the filtration is determined by the minimal model itself.)
\end{proposition}

Thus an infinite path in a~graded graph gives rise to an inductive limit of a~sequence of nested trees corresponding to successive vertices of this path. The sequence of the roots of these trees has no limit, and the increasing sequence of the nested sets of leaves of the trees constitutes the set of all infinite paths in the graph cofinal with the given one; the number of such paths, as well as the number of vertices at every level of the trees, goes to infinity. The measures on the leaves have no limit, but there arises a~cocycle of ratios of measures. Thus we have obtained a~rootless tree whose all levels are infinite. It is easier to view this tree as an infinite sequence of hierarchies on the infinite set of its leaves supporting the cocycle of ratios of conditional measures.

Of course, the construction depends on the choice of an infinite path. Here one should take into account whether the filtration is ergodic or not, but  we will not need to study this problem.

At the same time, it is clear that this construction cannot be used to advance the classification of finite filtrations from Sec.~\ref{ssec2.4} towards a~classification of infinite filtrations, since it takes into account only finite invariants of filtrations, which do not suffice for this purpose. However, these data suffice for classification of filtrations up to
 \textit{finite isomorphism}, which means the existence of a~complete system of Borel invariants for this classification.

\subsection{An equipment of a~multigraph and cotransition probabilities}
\label{ssec3.3}
We introduce an additional structure on a~(multi)graph, a~\textit{system of cotransition probabilities}, or an \textit{equipment},
$$
\Lambda=\{\lambda=\lambda_v^u; \ u\in \Gamma_n,\ v\in \Gamma_{n+1},\
(u,v)\in \operatorname{edge}(\Gamma_n,\Gamma_{n+1}), \ n=0,1,2,\dots\},
$$
by associating with each vertex $v \in \Gamma_n$ a~probability vector whose component~$\lambda_v^u$ is the probability of an edge $u\prec v$ that comes to~$v$ from the previous level; thus
$\displaystyle\sum_{u\colon u\prec v}\lambda_v^u=1$; $\lambda_v^u\geqslant 0$.
In the case of a~multigraph, probabilities are assigned to each edge from the set of edges joining vertices~$u$,~$v$ such that~${u\prec v}$.

\begin{definition}
\label{def5}
An \textit{equipped multi(graph)} is a~pair $(\Gamma,\Lambda)$ consisting of a~(mul\-ti)\-graph~$\Gamma$ and a~system~$\Lambda$ of cotransition probabilities on the edges of~$\Gamma$. An equipment allows one to define the probability of a~path leading from~$\varnothing$ to a~given vertex as the product of cotransition probabilities over all edges constituting the path.

The most important special case of an equipment (i.\,e., of a~system~$\Lambda$ of cotransition probabilities), which is called the
 \textit{canonical}, or
\textit{central}, equipment and is studied in combinatorics, representation theory, and algebraic situations, is as follows:
$$
\lambda_v^u=\frac{\dim(u)}{\sum_{u:u\prec v} \dim(u)}\,,
$$
where $\dim(u)$ is the number of paths from the initial vertex~$\varnothing$
to a~vertex~$u$ (in representation-theoretic terms, this is the dimension of the representation of the algebra~$A(\Gamma)$ corresponding to the vertex~$u$).
\end{definition}

One can easily see that for the central equipment, the measure on the set of paths leading from~$\varnothing$ to a~given vertex is uniform, and the cotransition probability to achieve a~vertex~$v$ from a~preceding vertex~$u$ is proportional to the fraction (among all the paths that lead from~$\varnothing$ to~$v$) of those paths that pass through~$u$.

The canonicity of this system of cotransition probabilities is in the fact that it is determined only by the graph itself. The corresponding Markov measures on the path space~$T(\Gamma)$ are called
\textit{central measures}; up to now, the study of Bratteli diagrams was restricted to only these measures. In terms of the theory of $C^*$-algebras, central measures are traces on the algebra~$A(\Gamma)$, and ergodic central measures are indecomposable traces. For more details on the case of central measures, see~\cite{77}, Sec.~7 below, and the extensive bibliography which can be found in the papers from our (incomplete) list of references.

Measures on the path space of a~graph that agree with the canonical equipment are called central, and in the theory of stationary Markov chains they are called measures of maximal entropy.

\begin{figure}[t!]
\centering
 \includegraphics{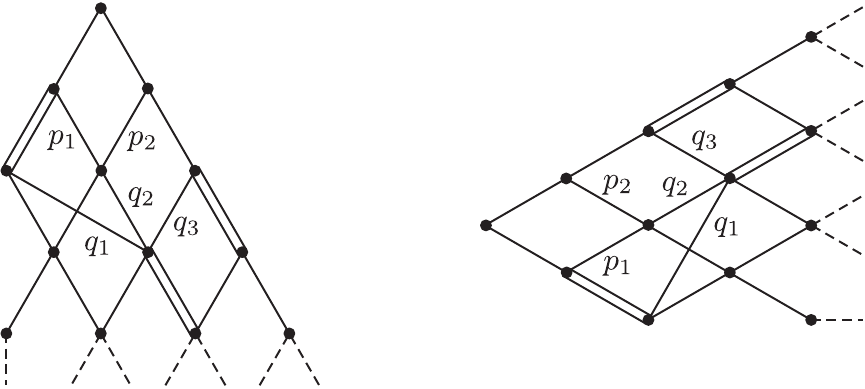}
\caption{A~graph of paths and a~Markov chain: rotation by $90^\circ$}
\label{fig3}
\end{figure}

\subsection{The Markov compactum of paths}
\label{ssec3.4}
The term ``cotransition probabilities'' is borrowed from the theory of Markov chains~\cite{9}: if we regard a~graph, rotated by~$90^\circ$  (see Fig.~\ref{fig3}), as the set of states of a~Markov chain starting from the state~$\varnothing$
at time $t=0$, and the numbers of levels as moments of time, then $\Lambda=\{\lambda_v^u\}$ should be viewed as the system of cotransition probabilities for the Markov chain:
$$
\operatorname{Prob}\{x_{t}=u\mid x_{t+1}=v\}=\lambda_v^u.
$$

Recall the definition of a~Markov compactum in the generality we need.

\begin{definition}
\label{def6}
Consider a~sequence~$\{X_n\}_{n=1}^{\infty}$ of finite sets, and associate with it a~sequence of ``multi-edges,'' i.\,e., a~collection of matrices~$M_n$ of size
($|X_n|=:$) $d_n\times d_{n+1}$ ($:=|X_{n+1}|$) in which an element~$m_{u,v}$,
$u \in X_n$, $v \in X_{n+1}$, is a~nonnegative integer equal to the number of edges joining~$u$ and~$v$; to each such edge  we assign a~nonzero number.

A~\textit{path (trajectory)} is a~sequence of adjacent edges.
If all the multiplicities~$m_{u,v}$ are equal to zero or one, then a~path is a~sequence of adjacent vertices.

The space of all paths (trajectories) endowed with the natural topology of an inverse spectrum of sets of finite paths is called a~\textit{Markov (topological) compactum}. A~Markov compactum is called stationary if the state spaces and the matrices do not depend on~$n$:
$$
X_n\equiv X_1,\qquad
M_n\equiv M_1,\quad
n=1,2,\dots;
$$
in this case, one can consider not only the one-sided, but also the two-sided Markov compactum.

A~Markov compactum with an equipment (or with cotransitions) is a~Markov compactum in the sense of the above definition in which elements~$m_{u,v}$ of the matrix~$M_n$ are vectors whose coordinates are positive numbers not exceeding~$1$ corresponding to the probabilities of edges going from  $u\in X_n$ to~$v\in X_{n+1}$, so that the sum of all probabilities along a~row (over all edges going to~$v$) is equal to~$1$.
\end{definition}

Sometimes, a~Markov compactum (either equipped or not) will be called a~Mar\-kov chain. We are going to consider Markov measures on a~(usually, nonstationary) Markov compactum.

One can see from the definition that the path space of a~graded (multi)graph is a~Markov compactum and an equipment of the graph defined above is an equipment of the Markov compactum. In what follows, we do not distinguish between these two languages: that of Markov compacta and that of path spaces of graded (multi)graphs. To begin with, these two (seemingly quite different) areas of mathematics, the theory of Markov processes and the theory of approximately finite-dimensional algebras (together with its combinatorial ``lining,''  the analysis of graded multigraphs), are in fact different versions of the same theory, and the multigraph corresponding to a~Markov chain is exactly a~Bratteli diagram. However, the settings of the problems are different, being algebraic in one case and probabilistic in the other one, and the traditions of drawing graphs are different too.

It may seem that the difference between these two seemingly distinct theories can be removed by rotating the figure by~$90^\circ$. However, it is important to overcome this difference conceptually and to understand to what extent both areas, the theory of graded graphs and
$\operatorname{AF}$-algebras on the one hand and the theory of Markov chains on the other hand, can be enriched by exchanging their dissimilar problems and intrinsic methods.

For example, I have long ago raised perhaps paradoxical problems about the $K$-functor of a~Markov chain and ergodic $K$-theory, or about Shannon-type theorems for $\operatorname{AF}$-algebras (see~\cite{87}).

Note that the analysis of asymptotic properties of the sequence of matrices~$M_n$ introduced above is also of great independent interest from the technical point of view. One may say that the whole theory of Markov compacta of paths in graded graphs we consider here is a~part of asymptotic theory of infinite products of Markov matrices, which is  important in itself and also related to the theory of invariant measures, phase transitions, etc.

\subsection{A~cocycle of an equipped graph}
\label{ssec3.5}
The notion of equipment is related to the notion of cocycle on the path space of a~graph. Consider a~function~$c(\,\cdot\,{,}\,\cdot\,)$ on the set of pairs of cofinal (i.\,e., eventually coinciding) paths on the graph with values in the set of nonnegative real numbers. Assume that
$$
c(t,s)c(s,t)=1,\quad
c(t,s)=c(t,z)c(z,s),\quad
c(t,t)=1
$$
for all pairs of cofinal paths~$(s,t)$. Sometimes, instead of ``$2$-cocycle'' one says simply ``cocycle on an equivalence relation''; the definition makes sense for an arbitrary equivalence relation. A~cocycle is said to be trivial, or cohomologous to~$1$, if there exists a~nonnegative function~$g$ on paths such that
$$
c(t,s)=\frac{g(t)}{g(s)}\,.
$$
A~simple example: for a~measurable partition of a~space, the function
\begin{equation}
\label{eq3.1}
c(t,s)=\frac{\mu^C(t)}{\mu^C(s)}
\end{equation}
(the ratio of the conditional measures of points lying in the same element of the partition) is a~trivial cocycle.

\begin{lemma}
\label{lem1}
For every locally finite filtration in a~Lebesgue space with continuous measure there is a~well-defined cocycle on the tail equivalence relation.
\end{lemma}

\begin{proof}
Consider two paths~$s$,~$t$ lying in the same element of the partition~$\xi_n$, and define~$c(t,s)$ to be the ratio of their conditional measures (see~\eqref{eq3.1}). As follows from the definition, conditional measures have the transitivity property, i.\,e., the ratio does not change if we consider the conditional measures of the partition~$\xi_m$,
$m>n$. The lemma is proved.
\end{proof}

Assume that we have an equipment of a~graded graph~$\Gamma$ (or a~Markov compactum); the cocycle corresponding to the equipment is defined as follows: for two paths
$$
t=(t_1,t_2,\dots,t_n,\dots),\quad
s=(s_1,s_2,\dots,s_n,\dots),
$$
with $t_i=s_i$ for $i\ge n+1$, we have
$$
c(t,s)=\prod_{i=1}^{n-1}
\frac{\lambda_{t_{i+1}}^{t_i}}{\lambda_{s_{i+1}}^{s_i}}\,.
 $$
Such a~cocycle will be called a~\textit{Markov} cocycle. Of course, there exist non-Markov cocycles.

\textit{Every Borel probability measure defined on the path space of a~graph (the space of trajectories of a~Markov chain) determines a~system of cotransition probabilities; namely, it is the system of conditional measures of the measurable partitions~$\xi_n$}. One says that a~measure \textit{agrees} with a~given system~$\Lambda$ of cotransition probabilities (with a~given cocycle) if the collection of the corresponding cotransition probabilities for all vertices coincides with~$\Lambda$.

\begin{lemma}
\label{lem2}
If the tail filtration on the path space of an equipped graph is ergodic with respect to a~given measure (or simply the measure is ergodic), then the cocycle constructed from the equipment is nontrivial.
\end{lemma}

Indeed, the triviality of the cocycle would mean the existence of measurable functions constant on tail equivalence classes, which contradicts the ergodicity.

The canonical equipment gives rise to the cocycle identically equal to~$1$, and measures that agree with this cocycle are central, or invariant, measures.

The Markov property of central measures on an arbitrary graded graph follows automatically from their definition, as was observed in old papers by the author and S.\,V.~Kerov~\cite{86},~\cite{30}.

Recall that a~system of cotransition probabilities does not in general allow one to uniquely determine a~system of transition probabilities
$$
\operatorname{Prob} \{x_{t+1}=v\mid x_t=u\};
$$
in other words, this system does not in general uniquely determine the Markov chain. On the contrary,
the transition probabilities uniquely determine the Markov chain if an initial distribution is given.
 An analog of the initial distribution for cotransition probabilities is the final distribution, i.\,e., a~measure on the boundary (see below), which is the collection of all ergodic measures with given cotransitions.

A~measure on the path space of a~graph will be called \textit{ergodic} if the tail $\sigma$-algebra (i.\,e., the intersection of all $\sigma$-algebras of the tail filtration) is trivial $\operatorname{mod} 0$.

\textit{One of our purposes is to enumerate all ergodic Markov measures with~a given system of cotransition probabilities. The list of such measures will be called the absolute boundary, or the absolute, of an equipped Markov compactum, or the absolute of a~graded multigraph. This is a~topological boundary, and, as we will see, the Choquet boundary of a~certain simplex (a~projective limit of finite-dimensional simplices).}

\subsection{Examples of graded graphs and Markov compacta}
\label{ssec3.6}
The list of graded graphs and multigraphs includes a~series of classical examples: the Pascal graph (already known to ancient Chinese mathematicians), which is the lattice~${\mathbb Z}^2_+$ graded along the diagonal (see Fig.~\ref{fig4}), and its multidimensional generalizations~${\mathbb Z}^d_+$, $d>2$, the Euler graph, the Young graph (see Fig.~\ref{fig5}), and, more generally,  Hasse diagrams of lattices, primarily distributive lattices. Among the more recent examples, we can mention the graph of unordered pairs, or the tower of dyadic measures
 (\cite{69},~\cite{82}, see Sec.~\ref{ssec7.3} below), the graph of ordered pairs~\cite{91}, etc., which appeared in connection with the theory of filtrations and the adic realization of ergodic automorphisms~\cite{67},~\cite{37},~\cite{89}.

\begin{figure}[h!]
\centering
 \includegraphics{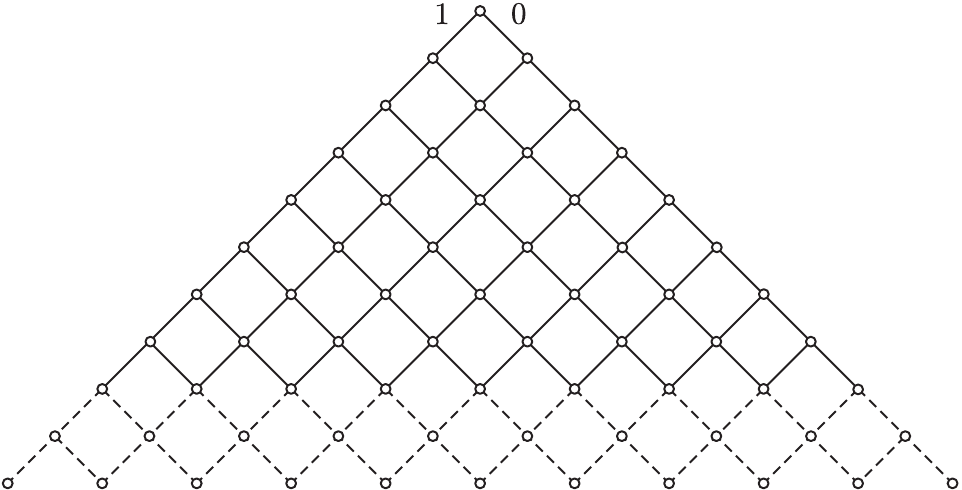}
\caption{The Pascal graph}
\label{fig4}
\end{figure}

\begin{figure}[h!]
\centering
 \includegraphics{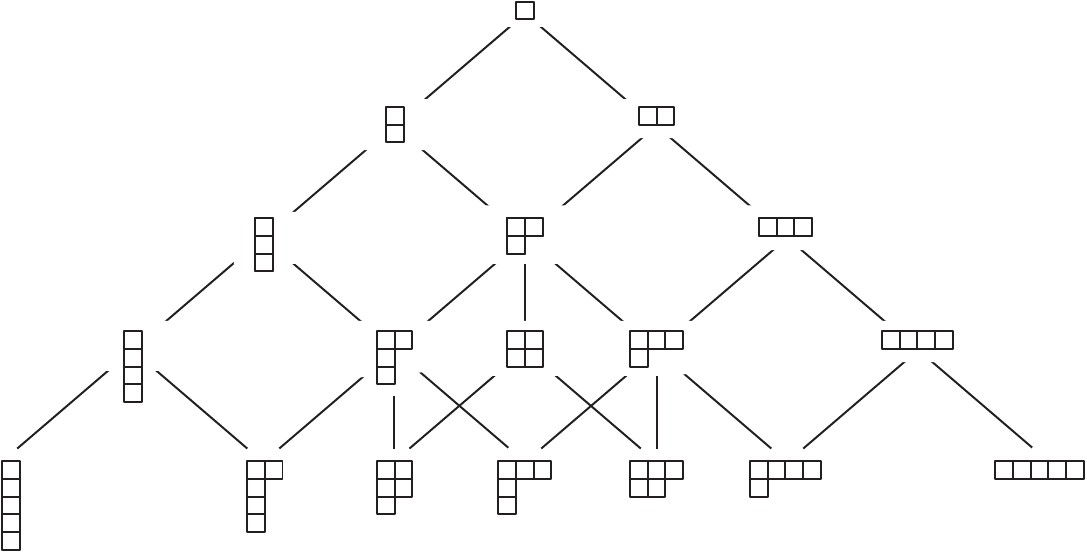}
\caption{The Young graph}
\label{fig5}
\end{figure}

A~graph is said to be of dyadic type if the number of paths from the initial vertex to every vertex of the $n$th level is equal to~$2^n$ for all $n=1,2,\dots$\,. This class includes the graphs of ordered and unordered pairs, the minimal dyadic graph, and the minimal polyadic (Glimm) graph (see Fig.~\ref{fig6}).

We emphasize once more that the path space in any graph can be regarded as the space of trajectories of a~topological (in general, nonstationary) Markov chain. This space (of paths or trajectories), regarded as a~Borel space, supports the tail filtration. The graph (chain) can be endowed with an equipment. And if we have a~Borel measure (not necessarily Markov) on the space of paths (trajectories), then we obtain a~\textit{locally finite} filtration on the resulting measure space. The local finiteness of this filtration follows from the local finiteness of the graph.

\begin{figure}[h!]
\centering
 \includegraphics[scale=0.9]{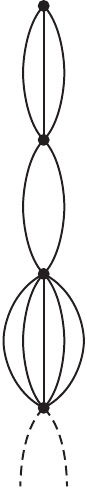}
\caption{The Glimm graph (beads)}
\label{fig6}
\end{figure}

\section{The Markov realization of locally finite filtrations}
\label{sec4}

We have approached one of the main theorems of the paper, which connects the theory of filtrations with the combinatorics of Markov chains and graded graphs.

\subsection{The theorem on a~Markov realization}
\label{ssec4.1}

\begin{theorem}
\label{th5}
Every locally finite filtration is isomorphic to the tail filtration of a~Markov chain, or to the tail filtration of the path space of a~graded graph.\footnote{Here we more often associate the theory of filtrations with Markov chains rather than with graded graphs, because the term ``Markov property'' and the language of Markov chains is better known and more widely used  than the equivoluminar language of graded graphs. One of the purposes of this survey is to popularize the equivalence of these languages and to report the new things brought into the field by graded graphs.}
\end{theorem}

\begin{proof}
Let
$$
\tau=\{\mathfrak A\}_{n=0}^{\infty}\simeq \{\xi_n\}_{n=0}^{\infty}
$$
be a~locally finite filtration in a~Lebesgue space
$(X,{\mathfrak A}_0,\mu)$ with continuous measure~$\mu$. Starting from~$\tau$, we will construct a~Markov chain with finite state spaces whose tail filtration is isomorphic to~$\tau$.

The finite partitions introduced below are labeled, i.\,e., their elements are assigned labels (for example, positive integers). The (ordered) product of two or more labeled partitions is again a~partition of the same type, whose elements are labeled by the ordered tuples of labels of the factors. In what follows, it is convenient to assume that for finite partitions $\eta \succ \eta'$, the set of labels of the larger partition~$\eta$
includes the set of labels of the smaller partition~$\eta'$.

Denote by $\phi=\{\xi_n\}_n$ the sequence of measurable partitions corresponding to a~filtration of $\sigma$-algebras~$\mathfrak F$. Choose a~basis of the $\sigma$-algebra~${\mathfrak A}_0$, i.\,e., an arbitrary increasing sequence of finite labeled partitions~$\{\eta_n\}_{n=1}^{\infty}$ that tends to the partition into singletons of the space~$(X,\mu)$.

We start an inductive construction with the base case, namely, with defining the first state space of the future Markov chain. For this, we first choose an arbitrary partition \textit{complementary to~$\xi_1$}, i.\,e., a~labeled measurable partition~$\zeta_1$ such that almost every element of~$\zeta_1$ intersects almost every element of~$\xi_1$ in at most one point (this can be done, since the elements of~$\xi_1$ are finite, and hence~$\zeta_1$ can be chosen finite).\footnote{The choice of a~partition~$\zeta_1$ complementary to a~given partition whose elements are finite or countable is a~standard procedure developed in~\cite{45}: it proceeds by successively constructing  sets of
maximal measure for which the intersection with almost all elements
of the given partition consists of at most one point. This is the main step in the classification of measurable partitions.}

Consider the product of partitions $\sigma_1=\eta_1\vee\zeta_1$; it is a~finite partition. We declare the set of its elements, i.\,e., the quotient space $Y_1=X/\sigma_1$ regarded as a~measure space, to be the set of labeled states of the Markov chain at time $t=1$.

The induction step is a~construction of the set~$Y_n$ of states at time $t=n$ provided that we have already constructed a~finite Markov chain of length~$n-1$ as the quotient  of the space $(X,\mu)$ by the finite partition
$\bigvee\limits_{i=1}^{n-1} \sigma_i$. We construct a~partition~$\sigma_n$, consider the quotient space $X/{\xi_{n-1}}$, and define on this space a~finite labeled partition~$\zeta_n$ complementary (in the sense described above) to the partition $\xi_n/\xi_{n-1}$ in the space $X/\xi_{n-1}$.

Note that if the intersection of an element of the partition~$\zeta_n$
with an element~$C$ of the partition $\xi_n/\xi_{n-1}$ is nonempty, then this intersection consists of a~single point, and hence every point from~$C$ obtains a~uniquely defined label, an element $D \in \zeta_n$.
The same is true if (as below) we replace the partition~$\zeta_n$
with a~finite subpartition.

Denote by~$\pi_n$ the canonical projection to the quotient space:
$$
\pi_n\colon X\to X/\xi_n.
$$

Define a~finite partition~$\sigma_n$ of the space~$X/\xi_{n-1}$
as a~subpartition of~$\zeta_n$ ($\sigma_n\succ \zeta_n$) by the following rule:

$(*)$ by definition, two points~$y_1$ and~$y_2$ of an element~$D$ of the partition~$\zeta_n$ lie in the same element of~$\sigma_n$ if and only if the elements $C_1=\pi_n^{-1}(y_1)$ and
$C_2=\pi_n^{-1}(y_2)$ endowed with the conditional measures of~$\xi_{n-1}$ are isomorphic as finite measured trees with points labeled by the elements of the partitions~$\sigma_i$, $i=1,2,\dots,n-1$; the partition~$\eta_n$ is obtained by restricting each of the finite partitions~$\sigma_i$,~$\eta_i$, $i=1,\dots,n-1$, to the elements~$C_1$ and~$C_2$.

Note that, as explained above, an isomorphism between two elements~$C_1$
and~$C_2$, if it exists, is unique.

Now we can define the $n$th state space~$Y_n$ as the quotient $\{X/\xi_1\}/\sigma_n$, and the path space of the $n$-step Markov chain, as the quotient  $\{X/\xi_1\}/\!\bigvee\limits_{i=1}^n \sigma_i$; denote these spaces by~$({\mathscr M}_n,\nu_n)$.

The Markov property immediately follows from the construction: an element of the partition contains points with equal conditional measures.

As also follows from the construction, these spaces constitute, in a~natural sense, an inverse spectrum of spaces, i.\,e.,
$$
({\mathscr M}_n,\nu_n)\leftarrow ({\mathscr M}_{n+1},\nu_{n+1}),
$$
and the projection preserves the Markov structure (a~new space is being added), as well as the measure. Thus we can consider the inverse limit of these spaces, which is already the space of infinite trajectories of a~one-sided Markov chain; denote this space, endowed with the limiting measure, by $({\mathscr M}_{\infty},\nu_{\infty})$.

Thus we have constructed a~homomorphism of the Lebesgue space  with filtration
$(X,\mu,\mathfrak F)$ to the space
$({\mathscr M}_{\infty},\nu_{\infty},\mathfrak T)$, where $\mathfrak T$ is the tail filtration of the Markov chain.

By construction, $\sigma_n \succ \eta_n$ for all~$n$, therefore,
$$
\bigvee_n \sigma_n \succ \bigvee_n\eta_n=\epsilon,
$$
i.\,e., the limiting homomorphism $X \to \mathscr M$ is an isomorphism of the measure spaces $(X,\mu)$ and $({\mathscr M},\nu)$, where $\nu$ is the image of the measure (the quotient measure). Denote the $n$th coordinate partition of the space~$\mathscr M$, i.\,e., the partition into the classes of trajectories that coincide from the
$(n+1)$th position, by~$\theta_n$. The tail filtration~$\mathfrak T$ gives rise to the sequence of partitions~$\{\theta_n\}_n$.

Let us show that this isomorphism sends the filtration~$\mathfrak F$
to~$\mathfrak T$. For this, it suffices to observe that our construction (by the same reason as above) sends the space~$X/\xi_1$ isomorphically to the quotient of $({\mathscr M},\nu)$ by the partition into the classes with respect to the first coordinate, and, more generally, the space $X/\xi_n$, to the analogous quotient with respect to the $n$th coordinate. Hence the bases of the corresponding partitions are sent to each other isomorphically, and, together with our choice of the partitions~$\zeta_n$
(complementary to~$\xi_n$), which are sent to the partitions into the classes of sequences that
 coincide from the~$n$th coordinate,
this yields a~desired isomorphism of filtrations. Theorem~\ref{th5} is proved.
\end{proof}

\subsection{Remarks on Markov models}
\label{ssec4.2}

1) Recall that in measure theory there are several theorems on Markov realizations of objects similar to the theorem proved above. For example: every transformation with invariant or quasi-invariant measure has an adic realization on the path space of a~graded graph with Markov measure.\footnote{An adic transformation~\cite{67},~\cite{68} (a~generalized odometer) is defined for graphs in which for each vertex there is a~linear order on the edges going to this vertex; using these orders, one can  define in a~natural way a~lexicographic order on the sets of cofinal paths; then the adic automorphism sends a~path to the next path with respect to this order if it is defined, which is the case on a~set of full measure.} The theory of adic transformations and adic realizations of automorhisms became a~new source of examples in ergodic theory. Even the first example suggested by the author, that of the Pascal automorphism, still intrigues researchers; it is not yet known whether its spectrum is continuous, and many other properties are also unknown
(see~\cite{75},~\cite{80},~\cite{41},~\cite{24}).

2) Another example is the theory of tame (hyperfinite) partitions (see~\cite{63}). Every discrete tame partition (in the category of spaces with quasi-invariant measure) is isomorphic to the tail partition on the space of trajectories of a~stationary Markov chain with finite state spaces and a~measure quasi-invariant with respect to the shift.

3) Let us give a~corollary of Theorem~\ref{th5}.

\begin{corollary}
\label{cor2} For an arbitrary locally finite filtration, the cocycle corresponding to the tail equivalence relation (see the definitions in Sec.~\ref{ssec3.5}) is isomorphic to a~cocycle of a~Markov chain with finite state spaces.

In particular, every hyperfinite (tame) partition with countable elements in a~Le\-bes\-gue space is isomorphic to the tail partition of a~Markov chain. In other words, every hyperfinite equivalence relation is isomorphic to the tail equivalence of a~Mar\-kov chain; this is also a~corollary of Theorem~\ref{th5} on the Markov property of a~locally finite filtration; this fact is known for a~long time (see~\cite{60}).
\end{corollary}

4) Of course, the model constructed in Theorem~\ref{th5} is by no means unique: there are many graded graphs whose tail filtrations are isomorphic. In spite of this, the language of graded graphs and their tail filtrations on spaces of paths, or on spaces of trajectories of Markov chains, is extremely convenient. Note, however, that when passing from the original filtration (even a~Markov one) to a~new realization, we can lose some additional (noninvariant) properties, e.\,g., the stationarity of the approximation: the original Markov chain could be defined as a~stationary one, but the approximation constructed in the proof of Theorem~\ref{th5} is rarely stationary. As we will see from examples in what follows, this loss is compensated by the simplicity of the analysis of the constructed approximation. The proof of Theorem~\ref{th5} contains a~method of constructing a~graph in whose path space the given filtration can be realized. We obtain a~class of rather difficult problems related to constructing graded equipped graphs with tail filtrations on path spaces isomorphic to a~given filtration. Let us consider an important example.

5) Consider a~\textit{dyadic filtration}, for which all partitions
$$
\xi_1,\xi_2/\xi_1,\dots,\xi_n/\xi_{n-1},\dots
$$
have two-point elements with the measures $(1/2,1/2)$. Our construction proves that every such filtration is isomorphic to the tail filtration of a~dyadic graded graph endowed with a~central measure on its path space; in this graph, all vertices of all levels except the initial vertex have two predecessors with equal cotransitions.

Thus the problem of metric classification of dyadic filtrations is reduced to the essentially combinatorial problem of classification of dyadic graded graphs (more exactly, path spaces of such graphs) endowed with central measures. Rephrasing  results obtained in the 1970s--1990s, as well as recent facts and theorems, in these, essentially combinatorial, terms allows one to take a~fresh look on the problem. The standardness criterion for dyadic filtrations can be rephrased as a~criterion saying when a~dyadic graph endowed with a~central measure can be brought to the simplest Glimm graph,
see Fig.~\ref{fig6}. Specific Markov models of filtrations will be considered in more detail in Sec.~\ref{sec6} devoted to nonstandard filtrations.

\section{Standardness, criteria, finite isomorphism}
\label{sec5}

\subsection{The minimal model of a~filtration}
\label{ssec5.1}
In this section, with every infinite locally finite filtration
$\tau=\{\xi_k\}_{k=1}^\infty$ of an arbitrary Lebesgue space~$(X,\mu)$
with continuous measure~$\mu$ we associate its canonical realization as the tail filtration of an equipped graded multigraph, or, equivalently, of a~Markov chain.
\textit{In contrast to the result of Theorem~\ref{th5}, the obtained filtration will not, in general, be isomorphic to the original filtration, but only finitely isomorphic.}

As we will see, this construction leads to a~special class of equipped graded multigraphs, which we call minimal. The model itself will be called the
\textit{canonical minimal model} of a~filtration.\footnote{Another natural term for minimal filtrations is ``finitely determined filtrations,'' since they are completely determined by their finite invariants (the same term ``finitely determined'' should have been used for standard filtrations, see the next section~\ref{ssec5.2}). \label{fnt7}}

Note that a~finite isomorphism does not preserve even the ergodicity of a~filtration (this is obvious from examples of dyadic filtrations), so we add the ergodicity requirement  each time when we compare  a~finite isomorphism and a~true one.

\textit{Given an ergodic filtration $\tau=\{\xi_k\}_{k=1}^\infty$ on a~Lebesgue space~$(X,\mu)$, we will construct an equipped graded multigraph~$\Gamma(\tau)$ and a~measure~$\nu$ on the path space~$T(\Gamma(\tau))$ that agrees with the equipment.}

The construction, as in Theorem~\ref{th5}, is inductive: we successively build a~multi\-graph (by levels), an equipment of this multigraph, and measures on its levels that will determine a~measure on the whole path space.

We start from the vertex~$\varnothing$ of the zero level and associate it to the whole space, more exactly, to the quotient of~$(X,\mu)$ by the trivial partition. Consider the first partition~$\xi_1$ and the finite (by the local finiteness of the filtration) partition~$\delta_1$ of the quotient space $X/\xi_1$ into the sets $A_1,A_2,\dots,A_{k_1}$ satisfying the following property: any two points  $C,D \in X/\xi_1$, regarded as elements of~$\xi_1$, have isomorphic conditional measures if and only if they lie in the  same element of~$\delta_1$. With the elements of the partition~$\delta_1$ of the space $X/{\xi_1}$, i.\,e.,
with the sets~$A_i$, $i=1,2,\dots,k_1$, associate the vertices of the first level of the graph~$\Gamma$. In other words, the vertices of the first level~$\Gamma_1$ correspond to the types of conditional measures of the partition~$\xi_1$. The number of edges going from every vertex to the vertex~$\varnothing$ is equal to the number of points in the corresponding element of the partition~$\xi_1$, and the cotransition probabilities are exactly the conditional measures of points in this element. Finally, the measures of the vertices themselves are equal to the measures~$\mu_{\xi_1}$ of the sets~$A_i$,
$i=1,\ldots,k_1$.

The process continues in exactly the same way. It is easiest to say that the construction of the vertices, edges, and cotransition probabilities of the level~$\Gamma_{m+1}$, provided that we already have the previous level~$\Gamma_m$, proceeds as above with the original filtration replaced by the filtration
$\{\xi_{m+s}/\xi_m\}_{s=1}^{\infty}$.

The sequence of partitions $\delta_1,\delta_2,\dots$ of the spaces $X/\xi_1, X/\xi_2,\dots$, respectively, corresponds to the partition of the vertices of the graph~$\Gamma(\tau)$ into levels, and the elements of each partition~$\delta_n$
correspond to the vertices of the $n$th level.

Thus, from the invariants of the filtration~$\tau$, we have constructed a~multigraph~$\Gamma(\tau)$ with an equipment and a~Markov measure on cylinder sets that agrees with this equipment and thus defines a~Markov measure~$\mu_{\Gamma(\tau)}$ on the space~$T(\Gamma_{\tau})$ of infinite paths in~$\Gamma(\tau)$. The obtained equipped locally finite graded multigraph~$\Gamma(\tau)$ (we will call this multigraph \textit{minimal},
and the corresponding tail filtration~$\overline\tau$, the
\textit{minimal model of the original filtration~$\tau$}) has the following characteristic property, which obviously follows from the construction.

\begin{proposition}[{\rm(theorem-definition of a~minimal graph)}]
\label{pr3}
An equipped multigraph is a~minimal model of a~locally finite filtration on a~Lebesgue space with continuous measure if and only if for any two distinct vertices~$u$,~$v$ of an arbitrary level, the trees, with roots in
$u$ and $v$, of paths leading from $\varnothing$ to these vertices are not isomorphic as measured trees.
\end{proposition}

The filtration~$\tau$ is finitely isomorphic to~$\overline\tau$
(this is obvious from the construction), but, in general, not isomorphic to it.

Clearly, the above construction can be applied to the tail filtration of an arbitrary equipped multigraph to obtain a~new multigraph with an equipment and a~measure. If the equipment of the original graph is canonical
 (i.\,e., the filtration is semihomogeneous), then its standard model will also be a~semihomogeneous filtration, since the image of a~canonical equipment is, obviously, also canonical. Thus our construction defines, in particular, an operation on graded graphs, by associating with an arbitrary multigraph a~\textit{minimal multigraph}. Obviously, the minimal model of a~minimal multigraph is the multigraph itself. In other words, this operation is a~projection in the space of all graded multigraphs to the space of minimal multigraphs.

We emphasize that in the basis method of defining a~filtration, the basis of the algebra of all measurable sets in the path space of the constructed multigraph~$\Gamma(\tau)$ is the algebra of cylinder sets, i.\,e., sets of paths determined by conditions on finite segments of paths. But when we map the tail filtration of one multigraph to the tail filtration of another one, the inverse image of the $\sigma$-algebra of all measurable sets of the second multigraph does not in general coincide with the $\sigma$-algebra of all measurable sets of the first one in the space of the original filtration. In other words, we obtain homomorphisms, and not isomorphisms, of the spaces $X/\xi_n$, in contrast to Theorem~\ref{th5} from the previous section, where the construction was aimed at obtaining an isomorphism. Namely, in the proof of Theorem~\ref{th5}, we had to subdivide the spaces $X/\xi_n$ and introduce additional vertices in graphs corresponding to the same conditional measures on the path space of the tree of the original graph. In the new construction, we identify these vertices and glue them together.

\begin{example}
\label{ex1}
The standard model of every ergodic dyadic filtration is the tail filtration of the Glimm multigraph of beads (see Fig.~\ref{fig6});
this graph is minimal, and its path space is the dyadic compactum~$\{0,1\}^{\infty}$ with the ordinary tail equivalence relation. Note that in this example there is a~unique central measure, the Bernoulli measure with probabilities
$(1/2,1/2)$. For brevity, in what follows, filtrations finitely isomorphic to Bernoulli ones will be called finitely Bernoulli.
\end{example}

\begin{figure}[b!]
\centering
 \includegraphics{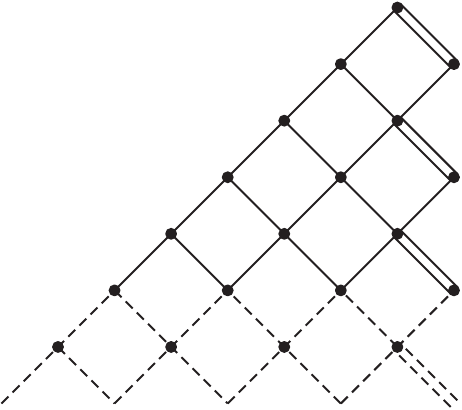}
\caption{A~minimal graph: the half of the Pascal graph}
\label{fig7}
\end{figure}

\begin{example}
\label{ex2} The Pascal graph is not minimal, since symmetric vertices in the graph have isomorphic trees (see Fig.~\ref{fig7}, showing the quotient graph). Moreover, if we introduce the ergodic central measure on the path space (this is the Lebesgue measure on paths regarded as sequences of zeros and ones), then the set of paths passing through one of the two first vertices does not satisfy the conditions of Theorem~\ref{th5}. Therefore, the filtration on the path space endowed with this measure is not standard. A~simple example of an inhomogeneous nonstandard Markov chain is given at the end of Sec.~\ref{ssec5.5}. An example of two measured trees that are isomorphic as trees but nonisomorphic as measured trees is shown in Fig.~8.
\end{example}

\begin{figure}[t!]
\centering
 \includegraphics{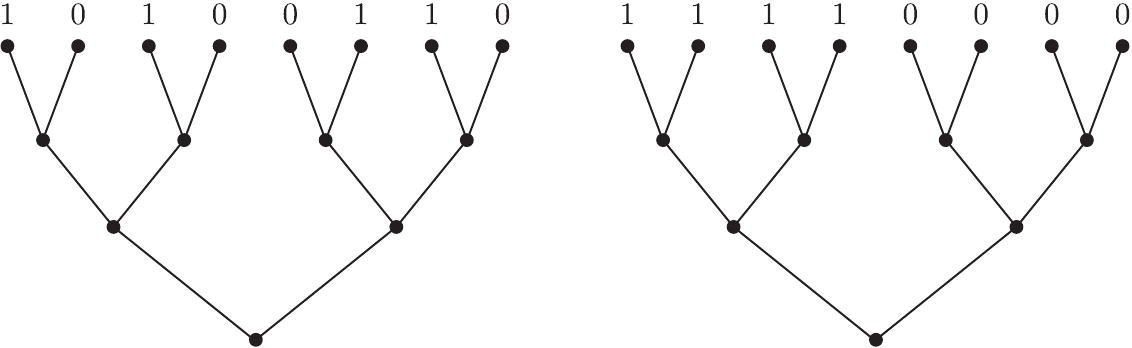}
\caption{Nonisomorphism of words with respect to the group of symmetries of a~tree}
\label{fig8}
\end{figure}

Let us summarize the main conclusions.

\begin{proposition}
\label{pr4}
1) Every locally finite filtration~$\tau$ is finitely isomorphic to its minimal model~$\overline\tau$, i.\,e., the tail filtration on a~minimal multigraph~$\Gamma(\tau)$, but in general is not isomorphic to it.

2) The minimal model, regarded as an equipped graph with a~measure on the path space, is a~complete finite isomorphism invariant of filtrations, i.\,e., two filtrations are finitely isomorphic if and only if their minimal models coincide.

3) The minimal model of an $\{r_n\}_n$-adic and, more generally, finitely Ber\-noul\-li filtration is a~Bernoulli filtration. Thus an example of a~minimal ergodic filtration is provided by any Bernoulli filtration.
\end{proposition}

\begin{proof}
Indeed, every finite fragment of the filtration has been isomorphically and measure-preservingly mapped to a~finite fragment of the graph.
\textit{And since we have used only conditional measures of all partitions and measures of sets with a~fixed type of conditional measures, finitely isomorphic filtrations lead to the same standard models}. Conversely, if two filtrations
$\tau=\{\xi_n\}$ and $\tau'=\{\xi'_n\}$ are not finitely isomorphic, then for some~$n$ the partitions~$\xi_n$ and~$\xi'_n$ are not isomorphic, and hence their invariants contained in the minimal model will differ.

It is clear from the above that for every Markov chain with finite state spaces there exists a~minimal Markov chain generating (as the tail filtration) a~filtration finitely isomorphic to the given one. The proposition is proved.
\end{proof}

We see that the \textit{problem of finite classification of locally finite filtrations is equivalent to the problem of constructing the space of minimal equipped graded graphs with a~measure on the path space that agrees with the equipment,
i.\,e., the space of complete invariants of such filtrations} (for more on the space of filtrations, see Sec.~\ref{ssec5.6} below).

Unfortunately, passing to minimal models and minimal filtrations does not simplify the study of filtrations, since every filtration can be arbitrarily closely, in a~natural sense, approximated by minimal ones. However, minimality brings us nearer to the fundamental notion of standardness.

\subsection{Minimality of graphs, standardness, and the general standardness criterion}
\label{ssec5.2}
The following important definition will be often used in what follows.

\begin{definition}
\label{def7}
A~locally finite ergodic filtration is called standard if it is isomorphic to a~minimal filtration.\footnote{See footnote\,\raisebox{1pt}{{\scriptsize{\ref{fnt7}}}} about the term ``finitely determined filtration.''}
\end{definition}

For example, a~$p$-adic filtration is standard if it is isomorphic to a~Bernoulli filtration; in other words, in the class of filtrations finitely isomorphic to Ber\-noul\-li ones, only the Bernoulli filtrations themselves and those isomorphic to them are standard.

The difficulty is that a~filtration isomorphic to a~minimal one
(i.\,e., a~standard filtration) is not necessarily minimal, in other words, is not necessarily the tail filtration of a~minimal multigraph (or a~minimal Markov chain). Thus the
\textit{minimality property for filtrations is not invariant under isomorphism, and the ``invariant hull'' of this property is exactly ``standardness.''}

Hence the problem arises of how to describe standard filtrations in invariant (intrinsic) terms,
i.\,e., how to check that a~given filtration is isomorphic to a~standard one.

In the next theorem we give an invariant characterization of a~standard filtration (the ``standardness criterion'').

Let $\tau=\{\xi_n\}_n$ be an ergodic locally finite filtration, and let $\eta_n$ be the finite partition of the space~$X/\xi_n$
into the maximal classes of elements $C\in
\xi_n$ on which the restrictions of the finite filtration~$\{\xi_k\}_{k=1}^{n-1}$ are isomorphic (equivalently, the corresponding measured trees are isomorphic). Fix an arbitrary isomorphism of these elements with some typical element~$\overline C$ regarded as a~measured tree. Then, obviously, every element~$D_i$,
$i=1,\dots,r_n$, of the partition~$\eta_n$ is a~direct product:
$$
D_i\simeq \overline C \times (D_i/\xi_n), \qquad D_i \in \eta_n.
$$
Thus we have defined an independent complement~$\xi^-_{n,i}$ to the restriction of the partition~$\xi_n$ to~$D_i$; this is a~finite partition of~$D_i$.

\begin{theorem}
\label{th6}
A~locally finite ergodic filtration $\tau=\{\xi_n\}_n$
of a~Lebesgue space $(X,\mu)$ with continuous measure is standard if and only if, in the above notation, for every
$\varepsilon >0$ and every measurable set~$A$ there exists a~number~$n$ and a~set~$A'$ with $\mu(A \bigtriangleup A')<\varepsilon$ such that
$$
A'=\bigcup_{i=1}^{r_n} A'_i
$$
and~$A'_i$ is measurable with respect to the independent complement~$\xi^-_{n,i}$.
\end{theorem}

The condition of the theorem means that every measurable subset in~$X$, for sufficiently large~$n$ and
up to a modification
on a~set of small measure,  breaks into a~union of sets  measurable with respect to the maximal possible independent complement to the restrictions of the partition~$\xi_n$. An important fact here is the maximality of the restrictions of~$\xi_n$ for which such a~complement exists. However, the choice of these independent complements is subject to additional considerations.

For filtrations finitely isomorphic to Bernoulli ones (for example, for dyadic filtrations), the partitions~$\xi_n$ themselves have independent complements, so  for such filtrations the criterion consists in verifying the possibility of constructing coherent independent complements constituting a~basis of the space.

In this sense, the condition of Theorem~\ref{th6} simply reduces the situation to the case of a~homogeneous filtration, where it is \textit{tautological}.

\begin{proof}[of Theorem~\ref{th6}]
The ``if'' part is equivalent to the assertion that the partitions into cylinders in the path space of the graph from the definition of a~standard model (see Definition~\ref{def7}) constitute a~basis of the $\sigma$-algebra of measurable sets. The ``only if'' part follows from the fact that the basis of cylinders is exactly the basis of independent complements to the (maximal) restrictions of the partitions of the filtration mentioned above. The theorem is proved.
\end{proof}

The condition of Theorem~\ref{th6} shows that the standardness of a~filtration is in a~sense the property of its decomposability into Bernoulli components. For finitely Bernoulli filtrations, this is exactly the Bernoulli property,
i.\,e., independence. Hence the notion of standardness is a~generalization of the notion of independence.

In the next section we will formulate a~concrete standardness criterion, which uses and generalizes the criterion for finitely Bernoulli filtrations suggested by the author~\cite{59} in the 1970s.

It follows from the conditions of the theorem that a~possible obstacle to standardness is the existence of a~special group of symmetries preserving the measure in the path space of the graph.

\subsection{The combinatorial standardness criterion}
\label{ssec5.3}
The general standardness criterion given above (Theorem~\ref{th6}) requires checking the measurability of sets with respect to a~system of independent complements, but does not show how one can do this. In this section we give a~more constructive, in fact combinatorial, method of verifying standardness.

\subsubsection{The combinatorial criterion in terms of partitions or metrics}
\label{sssec5.3.1}
We start with the standardness criterion for finitely Bernoulli filtrations; as we will see, the general case will be obtained by ``mixing'' them. The criterion to be stated differs from the standardness criterion suggested by the author in the 1970s for dyadic sequences only by a~more universal statement. Note that in the form given below this criterion holds also for filtrations with continuous conditional measures.

Let $\tau=\{\xi_n\}$ be a~locally finite filtration of a~Lebesgue space $(X,\mu)$ with continuous measure that is finitely isomorphic to a~Bernoulli filtration, i.\,e., for every~$n$ the quotient partition
$\xi_n/\xi_{n-1}$, $n=1,2,\dots$, is isomorphic to a~partition whose all elements are finite and all conditional measures coincide and are equal to the distribution determined by a~finite probability vector~$p^{(n)}$, $n=1,2,\dots$, possibly different for different~$n$. The leading special case is the dyadic one:
$$
p^{(n)}\equiv (1/2,1/2)\quad\text{for all~$n$}.
$$

In our case, an element~$C$ of the partition~$\xi_n$, regarded as a~finite set endowed with the conditional probability measure,  has the structure of a~direct product of finite spaces endowed with the measure
$\prod\limits_{k=1}^n p_k$, but it is more convenient to represent it as a~homogeneous tree of rank~$n$ with the product measure on leaves.

Assume that on the space $(X,\mu)$ we have a~measurable bounded function~$f$ with real values, or a~metric, or a~semimetric~$\rho$; then on the space $X/\xi_n$ we can introduce a~distance $d^n_f(\,\cdot\,{,}\,\cdot\,)$ ($d^n_{\rho}$), which should be interpreted as the distance between elements~$C_1$,~$C_2$
of the partition~$\xi_n$ regarded as measured trees. This distance is a~version of the Kantorovich metric for probability measures in a~metric space, but in our case, there is an additional condition.

Namely, we consider so-called \textit{couplings~$\psi$},
i.\,e., measures on the direct product  $C_1\times C_2$
with the given marginal projections~$\mu^{C_1}$ and~$\mu^{C_2}$ equal to the conditional measures. The additional condition is that these couplings are not simply measures, but also couplings of trees,
i.\,e., $\psi$ should be concentrated on some tree, as a~subset in $C_1\times C_2$,
whose projections to both coordinates are~$C_1$ and~$C_2$ as trees. An example of a~coupling of trees is the graph of a~measure-preserving isomorphism of one tree onto another one.

The measures~$\psi(\,\cdot\,{,}\,\cdot\,)$ are called \textit{couplings}, or
\textit{transportation plans}; the set of all couplings will be denoted by~$\Psi_{C_1,C_2}$.\footnote{The metric or semimetric~$d^n_{\rho}$ is exactly the Kantorovich metric on the set of probability measures of the (semi)metric space
 $(Y,\rho)$; in our case, we take the semimetric $\rho(x,y)=|f(x)-f(y)|$
on $Y=C_1\cup C_2$.}

Then the distance~$d^n_f$ ($d^n_{\rho}$) is given by the formula
$$
d^n_f(C_1,C_2)=\min_{\psi \in \Psi_{C_1,C_2}} \sum_{x\in C_1,y\in C_2}
|f(x)-f(y)|\psi(x,y)
$$
(in case of metrics, one should replace $|f(x)-f(y)|$ with~$\rho(x,y)$).

\begin{theorem}
\label{th7}
An ergodic filtration finitely isomorphic to a~Bernoulli one is standard, i.\,e., isomorphic to a~Bernoulli filtration, if and only if the following condition holds: either

-- for every measurable function~$f$ on the space $(X,\mu)$, or

-- for every admissible metric~$\rho(\,\cdot\,{,}\,\cdot\,)$
on the space $(X,\mu)$, or

-- for every cut semimetric\footnote{That is, a~semimetric $\rho(x,y)$ equal to~1 if~$x$,~$y$ lie in different elements of a~finite partition of the space $(X,\mu)$, and to~$0$ if~$x$,~$y$ lie in the same element.}

\noindent the following equality holds:
$$
\lim_{n\to \infty}\,\int_{X/\xi_n \times X/\xi_n}
d^n_f(C_1,C_2)\,d\mu_{\xi_n}\,d\mu_{\xi_n}=0.
$$
\end{theorem}

For what follows, it is convenient to give an equivalent $\varepsilon$-form of the theorem condition:
\begin{align*}
&\forall\,\varepsilon>0,\quad
\exists N,\quad
\forall\,n>N,\quad
\exists D_n\subset X/{\xi_n}\colon
\\
&\qquad\mu_{\xi_n}(D_n)>1-\varepsilon\quad\text{and}\quad
\sup_{C_1,C_2 \in D_n} d^n_f(C_1,C_2)<\varepsilon.
\end{align*}

The theorem condition is especially clear for dyadic filtrations of cut semimetrics
with respect to partitions into~$k$ subsets: in this case, a~substantial simplification is due to the fact that we may take couplings to be the graphs of isomorphisms of trees; i.\,e., we may consider the Hamming metric on the set of vertices of the cube~$\mathbf{k}^{2^n}$, and then (instead of couplings~$\psi$) take the distance between the orbits of the action of the group of automorphisms of the dyadic tree on this cube. In other words,
\textit{the standardness criterion reduces to the fact that the space of orbits of the action of the group of automorhisms of the tree on the cube~$\mathbf{k}^{2^n}$ collapses in this metric to a~single point} (see~\cite{59},~\cite{69}).

The proof of the criterion for homogeneous, i.\,e.,
$\{r_n\}$-adic, filtrations exactly coincides with the proof for dyadic ($r_n\equiv 2$) filtrations and was suggested in~\cite{65},~\cite{69}. Later, several other expositions appeared, see, e.\,g.,~\cite{7},~\cite{13},~\cite{14}. The scheme of the proof remains exactly the same also for finitely Bernoulli filtrations, as well as for continuous filtrations (in which all conditional measures of all partitions $\xi_n/\xi_{n-1}$, $n=1,2,\dots$,
are continuous).

We will explain only the main idea of the proof of Theorem~\ref{th7};
for details, see~\cite{69}. First of all, the ``only if'' part follows from the fact that, by definition,  for a~Bernoulli filtration there exists a~sequence of independent partitions~$\{\eta_n\}_n$  constituting a~basis of the space such that
$$
\xi_n=\bigvee_{k=n}^{\infty}\eta_k,\qquad
n=0,1,2,\dots\,.
$$
Namely, it follows that for all finite partitions
$\bigvee_{k=1}^n \eta_k$, the theorem condition follows from their independence with
the corresponding partitions~$\xi_n$. On the other hand,~$\{\eta_n\}_n$ is a~basis, and hence the condition is satisfied for any finite partitions.

The proof of the ``if'' part is more laborious, one should prove that the theorem condition implies the existence of a~sequence of independent partitions~$\{\eta_n\}_n$ associated with~$\{\xi_n\}$ as described above that is a~basis of the measure space. This can be done again by induction, and the main lemma is that the theorem condition can be extended to the quotient filtrations
$\{\xi_{n+k}/\xi_k\}_{n=0}^\infty$, $k=1,2,\dots$, hence one can construct a~basis~$\{\eta_n\}_n$ by successively passing to quotient spaces and approximating some basis chosen in advance. For the first time, this scheme was used in the proof of the lacunary isomorphism theorem for dyadic sequences~\cite{58}.

\subsubsection{Comments on the criterion}
\label{sssec5.3.2} The statement of Theorem~\ref{th7} involves the notion of an admissible metric (see~\cite{71}).
Recall that an admissible metric in a~Lebesgue space is a~metric defined on a~set of full measure as a~measurable function of two variables such that the $\sigma$-algebra generated by all balls of positive radius in this metric generates the whole $\sigma$-algebra of measurable sets (see~\cite{92}). If a~function of two variables satisfies the conditions of a~metric (nonnegativity, symmetry, triangle inequality) as conditions $\operatorname{mod} 0$ on classes of metrics, as well as the above condition on the $\sigma$-algebra of balls, then the $\operatorname{mod} 0$ class of this function contains an admissible metric (see~\cite{61}).

\begin{theorem}
\label{th8}
Let $\rho$ be an arbitrary admissible metric on a~Lebesgue space with continuous measure. An ergodic filtration finitely isomorphic to a~Bernoulli one is standard,
i.\,e., isomorphic to a~Bernoulli filtration, if and only if the following condition holds: for every $\varepsilon>0$ there exists a~positive integer~$N$ such that for every $n>N$ and some set
$$
A_n \subset X/\xi_n, \qquad
\mu_{\xi_n}(A_n)>1-\varepsilon,
$$
the Kantorovich distance between the conditional measures on arbitrary elements
$C_1,C_2\subset A_n$ of the partition~$\xi_n$ in the space~$(X,\rho)$
regarded as a~metric space does not exceed~$\varepsilon$. This condition is or is not satisfied independently of the choice of~$\rho$.
\end{theorem}

Note that if in the condition of the last theorem  we take $\rho$ to be the cut semimetric corresponding to a~partition, then we obtain the condition of the theorem for functions with finitely many values. Hence the conclusion of the  theorem for cut semimetrics follows from the conclusion for functions. On the other hand, every semimetric is the limit of sums of cut semimetrics, hence the conclusion of the theorem for functions implies the conclusion for a~dense family of semimetrics.
Finally, it is not difficult to check that the fact that the elements of the partition endowed with the conditional measures approach one another in the introduced metric remains true when passing to the limit along a sequence
of metrics.\footnote{More exactly, the Kantorovich metric on the space of probability measures of a~given Borel space is continuous with respect to the weak topology in the space of admissible metrics (regarded as functions of two variables). The same considerations imply that the conclusion of the theorem is independent of the choice of a~metric.} Note, by the way, that the idea of varying metrics and the notion of an admissible metric
for a~fixed measure in a~space was suggested by the author in the 1980s and has been repeatedly used, for example, in the definition of scaled entropy (see~\cite{71},~\cite{82}).

The conclusion of the theorem can be equivalently stated as the convergence of distances to zero in measure, etc.

\subsection{The standardness criterion for an arbitrary filtration}
\label{ssec5.4}
How one should modify the condition of Theorem~\ref{th6} to obtain a~standardness criterion for an arbitrary locally finite filtration? In fact, the condition for the general case reduces to a~mixture of conditions on each finitely Bernoulli component of the filtration.

We will need the finite partitions~$\delta_n$ on the spaces
$(X/\xi_n, \mu_{\xi_n})$ introduced when defining minimality in Sec.~\ref{ssec5.1}. Recall that two elements~$C_1$,~$C_2$ of the partition~$\xi_n$ lie in the same element of~$\delta_n$ if the finite filtration~$\{\xi_k\}_{k=1}^{n-1}$ induces on~$C_1$,~$C_2$ isomorphic filtrations, or, equivalently, if these elements~$C_1$,~$C_2$ are isomorphic as trees endowed with the conditional measures. It is essential that the partition~$\delta_n$ is determined only by the conditional measures. For homogeneous filtrations, this is the trivial partition of the space $X/\xi_n$, and for minimal graphs, this is the partition into singletons at each level.

Then the statement in the $\varepsilon$-language  is as follows.

\begin{theorem}
\label{th9}
Let $\tau=\{\xi_n\}_{n=0}^{\infty}$ be an arbitrary finite filtration in a~Lebesgue space $(X,\mu)$
with continuous measure. The filtration~$\tau$ is isomorphic to a~standard filtration if and only if the following condition holds: for every $\varepsilon >0$ there exists~$N$ such that for every $n>N$ the partition of~$X/\xi_n$ into the types of elements of~$\xi_n$ isomorphic with respect to the previous fragment of the filtration,
$$
X/\xi_n=\bigcup_{i=1}^{k_n} A^n_i, \qquad
\{A^n_i\}_i=\delta_n,
$$
has the following property:
\begin{align*}
&\text{there exists~$D$}, \ D\subset X/{\xi_n},\
\mu_{\xi_n}(D)>1-\varepsilon,
\\
&\qquad\text{such that} \ \sup_{C_1,C_2 \in A^n_i\cap D}
d_f(C_1,C_2)<\varepsilon,
\end{align*}
in other words, inside each element of the partition~$\delta_n$, the distances between most elements of the partition~$\xi_n$ are small with respect to the metric~$d_f$ or~$d_{\rho}$.
\end{theorem}

It should be clear from this statement in what sense standardness is a~generalization of independence.

In fact, it differs from Theorem~\ref{th8} (i.\,e., the theorem for finitely Bernoulli filtrations) only in that the pairwise distances should be small only for \textit{most} pairs of isomorphic elements, and not for all such pairs. Thus the proof of the ``if'' part, which is the main part of the theorem, does not differ from the previous case. However, checking the condition is by no means easy.

It makes sense to state the ``only if'' part of the criterion as a~separate proposition.

\begin{proposition}
\label{pr5}
1) The filtrations satisfying the standardness criterion form an invariant class,
 i.\,e., a~filtration isomorphic to a~filtration satisfying the standardness criterion also satisfies this criterion.

\nopagebreak
2) A~minimal filtration satisfies the standardness criterion.
\end{proposition}

\begin{proof}
The first claim directly follows from the statement of the criterion.

The proof of the second claim breaks into two cases. In the first (obvious) case, assume that a~minimal multigraph is a~graph, i.\,e., has no multiple edges. Then the partition into the types of elements of the partitions~$\xi_n$ coincides with the partition into cylinder sets, and the $\sigma$-algebra generated by all complements coincides with the whole $\sigma$-algebra. In the second case, the partition into types reduces to Bernoulli components, for which the ``only if'' part of the criterion is proved.

The proposition is proved.
\end{proof}

The standardness criterion is of interest in the case of graphs essentially different from minimal ones. For canonical equipments, i.\,e., central measures, the criterion can be somewhat simplified, and in this simplified form it was stated in our previous papers; but for filtrations that are not semihomogeneous, the simplified statement is not equivalent to the general one given above, and it is in this sense that the term ``standardness'' is used. We will return to this question in connection with the so-called internal metric (see~\cite{82},~\cite{79},~\cite{78},~\cite{81}, and
Sec.~\ref{sec7}).

\subsection{The martingale form of the standardness criterion  and intermediate conditions}
\label{ssec5.5}
It is natural to ask how, given a~Markov, or even arbitrary, one-sided random process, one can determine whether or not its tail filtration is standard. The problem is to state the standardness criterion in terms of the process.

The most natural form of the answer to this question, as well as to other questions related to the interpretation of weakenings of the independence condition in the theory of random processes,  looks as a~strengthening of martingale convergence theorems. The classical theorem for a~random process~$\{x_n\}_{n \leqslant 0}$ in the form we need it looks as follows.

\textit{If the tail filtration of the process
(i.\,e., the intersection of the past $\sigma$-algebras) is trivial,
$$
\bigcap_n{\mathfrak A}_n=\mathfrak N,
$$
then for every real functional~$F$ in~$k+1$ coordinates of this process, the conditional distributions of~$F$ given a~fixed past (starting from $-n-1<k$) converge almost everywhere to the unconditional distribution of this functional:
\begin{align*}
&\lim_n \operatorname{Prob}(F(x_0,x_{-1},\dots,x_{-k})\geqslant
t\mid x_{-n-1},\dots)\hspace*{3cm}\\
&\qquad\qquad\qquad\qquad= \operatorname{Prob}(F(x_0,x_{-1},\dots,x_{-k})\geqslant
t),\quad t \in \mathbb R;
\end{align*}
in other words, as $n$ tends to infinity, the conditional measures on~$k+1$ coordinates of the process weakly converge almost everywhere (or in measure) to the corresponding unconditional measures.} Here it is essential that~$k$ is fixed and $n\to \infty$. Can~$k$ be allowed to go to infinity along with~$n$? Literally, this does not make sense, since the limit of the conditional measures for growing~$k$ is not defined, but one may speak about the
\textit{conditional measures approaching one another as~$n$ goes to infinity.} For this, one should define some metric in the space of conditional measures for a~given~$k$ and require, for instance, that the expectation of the distances between the conditional measures with respect to different conditions converge to zero. Here the choice of a~metric plays a~crucial role, and the classes of processes satisfying such requirements depend substantially on this choice. But in any case, the validity of conditions of this type shows that the process is kind of similar to a~Bernoulli process.

Let us state the standardness criterion in terms of convergence of a~sequence of conditional measures
(i.\,e., measures on spaces of measured trees); it can be regarded as a~``martingale'' definition of standard random processes.

\begin{theorem}[{\rm(Martingale standardness criterion)}]
\label{th10}
Consider an arbitrary (not necessarily Markov) random process~$\{x_n\}_{n<0}$ with values in~$[0,1]$
(or in an arbitrary Borel space). Consider its tail filtration~${\mathfrak A}_n$, $n=0,1,2,\dots$,
where~${\mathfrak A}_n$ is the $\sigma$-algebra generated by the random variables~$x_k$,
$k\leqslant -n$, $n=0,1,2,\dots$\,. Assume that the filtration is ergodic,
i.\,e., the intersection of the $\sigma$-algebras satisfies the zero--one law:
$$
\bigcap_{n=0}^{\infty}{\mathfrak A}_n=\mathfrak N.
$$

Take either

-- an arbitrary cylinder function~$f$ on the space of trajectories of the process, or

-- an arbitrary metric~$\rho$ on the space of trajectories of the process.

Fix two trajectories~$C_1$,~$C_2$ of the process for time from~$-\infty$ to~$-n$ with isomorphic conditional measures~$\mu^{C_1}$,~$\mu^{C_2}$ on the finite segments
$$
(x_0,x_{-1},x_{-2},\dots,x_{-n}|C_1)\quad\text{and}\quad
(x_0,x_{-1},x_{-2},\dots,x_{-n}|C_2),
$$
i.\,e., two points of the quotient space $X/\xi_n$ lying in the same element of the partition~$\delta_n$ (see the standardness criterion, Theorems~\ref{th7}
and~\ref{th8}), and consider the trees of trajectories corresponding to the elements~$C_1$ and~$C_2$ endowed with the conditional measures on the leaves. Then the filtration is standard if and only if the following condition is satisfied:
$$
\lim_{n \to \infty}\,\int_{X/\xi_n}\,\int_{X/\xi_n}
d_f(\mu^{C_1},\mu^{C_2})\,d\mu_{\xi_n}\,d\mu_{\xi_n}=0
$$
(in the case of a~metric, one should replace~$d_f$ with~$d_{\rho}$) for all measurable functions~$f$  on the space of trajectories of the process (it suffices to check this only for cylinder functions with finitely many values), or for all admissible metrics (actually, it suffices to check this only for one metric).
\end{theorem}

This statement is a~verbatim reproduction of the standardness criterion, but it leads to an important conclusion:
\textit{between the ordinary martingale convergence theorem, which holds for all processes determining an ergodic filtration (i.\,e., regular processes), on the one hand, and the standardness criterion, which has less applicability, on the other hand, there is a~natural spectrum of intermediate conditions on processes.} They are determined by the relation between the moment~$n<0$  that fixes a~condition in the past and the moment~$k(n)>n$ up to which we compare the conditional measures on trees of rank~$|k(n)|$. If~$k(n)$ does not depend on~$n$ and $n\to \infty$, this is the ordinary martingale convergence theorem, and if $k(n)=n$, this is the standardness criterion. Hence standardness is the greatest possible strengthening of the zero--one law. The author knows nothing on the intermediate cases, which can be called ``higher zero--one laws,'' but one can expect that processes for which the conditions hold for different intermediate growth rates of the sequence~$k(n)$ have different properties.

The most interesting part of the theory is the study of filtrations that are finitely Bernoulli but not isomorphic to Bernoulli filtrations. Deviations from the conditions of the standardness criterion can be measured by various invariants, for instance, the entropy of filtrations.
We will return to this problem in Sec.~\ref{sec6}.

The simplest example of an inhomogeneous finitely Bernoulli nonstandard filtration is the Markov chain with two states and the transition matrix
$$
\begin{pmatrix}
p & q
\\
q & p
\end{pmatrix},\qquad
p \ne q,\quad
p+q=1,\quad
p,q>0
$$
(see~\cite{82}); the cylinder set $\{x=\{x_n\}:x_1=0\}$ does not satisfy the standardness criterion. This filtration is a~double cover of the Bernoulli filtration with probabilities~$(p,q)$.

Conversely, the Fibonacci Markov chain with the transition matrix
$$
\begin{pmatrix}
\lambda^2 & \lambda \\
1 & 0 \\
\end{pmatrix},\qquad
\lambda^2+\lambda=1,
$$
which is not locally Bernoulli, determines a~standard homogeneous tail filtration, since the partitions~$\delta_n$, $n=1,2,\dots$, are the partitions into singletons of each level of the graph.

\subsection{The spaces of filtrations, Markov compacta, and grad\-ed multigraphs}
\label{ssec5.6}
The theorem on the Markov property of filtrations allows one to take  the space of equipped graded multigraphs, or, equivalently, the space of Markov compacta with Markov measures, as the space of all locally finite filtrations in a~Lebesgue space. In other words, we give the following definition.

\begin{definition}
\label{def8}
We identify the space of locally finite filtrations in a~continuous Lebesgue space with the space of equipped graded multigraphs (respectively, with the space of Markov chains with cotransition probabilities) endowed with a~fixed continuous measure that agrees with the equipment (respectively, with the cotransitions).
\end{definition}

With this definition, a~filtration is identified with the tail filtration of the path space of a~multigraph (respectively, with the space of trajectories of a~Mar\-kov chain). We emphasize that an important role in the convention to identify the collection of filtrations in measure spaces and the collection of equipped graded multigraphs  is played by the measure on the path space that agrees with the equipment: the equipped graph alone does not determine, for instance, whether the filtration is ergodic, this is determined by the measure.

For definiteness, we will speak about multigraphs; reproducing all statements in terms of Markov chains brings nothing new.

As we have seen, an equipped multigraph can be defined by a~sequence of finite matrices whose nonzero entries are the probabilities of cotransitions of adjacent vertices. Each such sequence determines an equipped  graded multigraph~$\Gamma$. The path space of~$\Gamma$ will be denoted by~$T(\Gamma)$. Besides, given a~multigraph~$\Gamma$, we introduce the simplex~$\Sigma_{\Gamma}$ of all probability measures on the space~$T(\Gamma)$ that agree with the equipment.

\textit{Denote by~$\mathfrak F$ the space of all pairs $(\Gamma,\Sigma_{\Gamma})$ where~$\Gamma$ runs over the set of all infinite equipped graded locally finite multigraphs and $\Sigma_{\Gamma}$ are the corresponding simplices of measures.} This space can also be viewed as the space of all locally finite filtrations of a~continuous Lebesgue space, and simultaneously as the space of all locally finite Markov compacta with Markov measures. We endow this space with the weak topology,
i.\,e., the topology of an inverse spectrum.

We may consider the following important subspaces of this space.

1. \textit{The subspace
of ergodic multigraphs}, i.\,e., multigraphs for which there exists at least one continuous ergodic measure on the path space that agrees with the equipment. This class does not contain, for example, trees and other decomposable multigraphs. It is of most interest for the needs of the theory of filtrations and ergodic theory. The combinatorial structure of such multigraphs must be completely described. In fact, this is a~question about hyperfinite equivalence relations on a~standard Borel space (the path space) with a~fixed cocycle of measures (cf.~\cite{6},~\cite{78},~\cite{52}).

2. \textit{The subspace
 of multigraphs with canonical equipments}
(the tail filtration in this case is homogeneous, i.\,e., the conditional measures on the elements~$\xi_n$, $n=1,2,\dots$, are uniform). The geometry of the projective limit of simplices and the limiting simplex of central measures ${\mathfrak F}_0$ are quite special; in the dyadic case, the limiting simplex is the simplex of unordered pairs, or the so-called ``tower of dyadic measures'' (see Fig.~\ref{fig11}
in Sec.~\ref{ssec7.3}). In the language of Markov chains (in the stationary case), central measures are measures of maximal entropy. The canonical equipment is determined by the very structure of the graph, which is why this space is of special interest.

3. \textit{The space~${\mathfrak F}^{\rm st}$
 of stationary filtrations}
(i.\,e., filtrations invariant under the shift, which is the passage to the quotient by the first partition:
$\tau\equiv \{\xi_n\}_{n=0}^{\infty}\simeq
\{\xi_n/\xi_1\}_{n=1}^{\infty}\equiv \tau'$), or, equivalently,
\textit{the space
 of stationary multigraphs} (i.\,e., multigraphs with constant cotransition matrices) with an additional symmetry, etc. The analysis of this space is of special importance for ergodic theory (of one automorhism).

4. Besides, in all cases one can give up considering measures on path spaces and pass to the Borel theory of graphs and filtrations. This means that we study only the space~$\mathscr G$, without fixing the second component in pairs $(\Gamma,\Sigma_{\Gamma})$,
i.\,e., a~simplex of measures~$\Sigma_{\Gamma}$, and consider only cotransitions.
(See Sec.~\ref{sec7}.)

It is more important to analyze the space~${\mathfrak F}$ from the point of view of the problem of finite and general isomorphism. Let us select in~${\mathfrak F}$ filtrations corresponding to minimal multigraphs (minimal Markov chains) with an equipment and a~measure on the path space. By the previous analysis, this is a~subspace of~${\mathfrak F}$ consisting of the minimal locally finite filtrations.

By Theorem~\ref{sec5}, this is the
\textit{space of complete invariants of filtrations up to finite isomorphism: no two different standard filtrations are finitely isomorphic, and any filtration is finitely isomorphic to a~standard one}.

Our definitions immediately imply the following assertions.

\begin{proposition}
\label{pr6}
1) The space of all equipped graded multigraphs (respectively, the space of all Markov compacta with cotransitions) is fibered over the space of minimal equipped graphs (respectively, over the space of minimal Markov compacta with cotransitions). Thus the space of all filtrations is fibered over the space of minimal filtrations. The space of minimal equipped graphs is dense in the space of filtrations.\footnote{At the same time, in the space, say, of dyadic filtrations there is only one minimal ergodic filtration and only one minimal graph. In general, one can introduce a~more complicated topology in which the set of minimal filtrations is closed.}

2) The set of standard ergodic filtrations is dense in the space of all filtrations.
\end{proposition}

Basically, the space of all minimal multigraphs (Markov chains) is the space of all locally finite types of filtrations.

Thus, in contrast to the problem of finite isomorphism, the problem of isomorphism of locally finite filtrations is wild: passing to the quotient of~${\mathfrak F}$ by the classes of isomorphic filtrations does not result in a~space with a~separable Borel structure. Moreover, the latter remains true even if we consider the problem of isomorphism of ergodic dyadic filtrations, i.\,e., in one (in fact, an arbitrary nontrivial) class of finite isomorphism. The situation is similar to many isomorphism problems in algebra and analysis, e.\,g., the problem of isomorphism of actions of infinite groups with invariant measure, or the problem of classification of irreducible unitary representations of the infinite symmetric group, etc. This question arises in the theory of dynamical systems, representation theory, asymptotic combinatorics, etc. The wildness of the classification problem follows from the fact that the classification includes a~typical wild problem;  we will return to this in another paper.

In what follows, we, on the one hand, propose to consider the isomorphism problem for some classes of filtrations close to standard ones and, on the other hand, suggest types of invariants for general filtrations substantially different from standard ones.

\section{Nonstandardness: examples and a~sketch of a~general theory}
\label{sec6}

The most difficult problems of the theory of filtrations are related to nonstandard filtrations. But, as follows from the standardness criterion, nonstandardness for inhomogeneous filtrations can have different causes. Recall that a~minimal filtration is uniquely determined by its finite type. The standardness of a~filtration means that it does not differ much from its finite type; roughly speaking, its homogeneous parts must be standard, i.\,e., Bernoulli.

Thus the nature of true nonstandardness shows itself already in homogeneous (and, moreover, even dyadic) filtrations. As we have already mentioned, the group of symmetries of the filtration itself and its quotients by the partitions~$\xi_n$ with arbitrarily large~$n$ must be infinite, otherwise the filtration cannot have a~complicated structure at infinity.

In this section we continue constructing Markov chains and graphs for nonstandard dyadic filtrations. Actually, these chains with finite state spaces, i.\,e., path spaces of graded multigraphs, should be regarded as an approximation or, more exactly, a~simple model of continual, as well as more complicated, Markov chains. But perhaps the most instructive thing in these constructions is that the graphs obtained from these approximations are interesting in themselves and, possibly, of even more importance than the approximation. This fact should be emphasized, because these graded graphs, regarded as Bratteli diagrams, represent important $C^*$-algebras which seem to have never been studied in the literature. For example, the algebras generated by the graph of words for the Abelian group~${\mathbb Z}^d$ (see below) look mysterious.

The method of proving nonstandardness contained in the criterion becomes clearer if one endows it with a~combinatorial meaning. This is exactly what we do below in several examples. In fact, the problem is reduced to the analysis of a~measure on the orbits of an action of the group of automorphisms of a~tree
on some subsets of configurations. The most transparent case is the action of this group on the set of vertices of the cube~$2^{2^{n}}$ with the uniform measure or with an arbitrary product measure with equal factors. The standardness criterion requires that the measure concentrates (with respect to the Hamming metric) in a~neighborhood of one orbit, and we need advanced methods of proving this fact. Unfortunately, such a~clarity is not yet achieved in the examples given below.

\subsection{Graphs of random walks over orbits of groups, graphs of words}
\label{ssec6.1}
The construction of a~Markov approximation, whose existence follows from Theorem~\ref{th5}, as a~Markov chain with finite state spaces is a~rich source of new graded graphs. On the other hand, these constructions allow one to study filtrations given a~priori as Markov filtrations with continual state spaces. Specializing the proof of Theorem~\ref{th5}, we will give an explicit construction of a~Markov chain with finite state spaces and obtain a~``graphic'' interpretation of the filtration. We mean models of random walks over the orbits of an action of a~countable group in a~space with continuous invariant measure and the tail filtration of the corresponding Markov chain. This subject is sometimes called the theory of
\textit{random walks in random scenery} (RWRS), and it is a~further step in the theory of random walks on groups,  which is currently the subject of hundreds of works. It should be noted that this relatively new field deals with problems of different nature: for instance, problems related to boundaries of walks, entropy, etc.\ (see, e.\,g., the circle of problems in~\cite{26}) are replaced with problems about the relation between the action of the group
and properties of filtrations (see~\cite{77}).

The question of when the corresponding filtrations (even dyadic ones) are standard is not sufficiently studied. It is only known that the filtration of the Markov process~$(T,T^{-1})$, that is, the random walk over the orbits of an automorphism~$T$, is not standard if the entropy is positive,
 $h(T)>0$ (see~\cite{22}).

For automorphisms~$T$ with discrete spectrum, this filtration is presumably standard. For the rotation~$T_{\lambda}$ of the circle, this is proved by the author
(unpublished) for well-approximable~$\lambda$, and by W.~Parry (unpublished)  in full generality.

But these are only first examples of random walks over orbits of actions; nonsimple walks  (i.\,e., walks not only over generators) offer much more possibilities. One can hope that the machinery of graphs will provide new tools for investigating such problems.

We begin the construction with the following general example. Let  $G$ be a~group and
$S$, with $|S|=s<\infty$, be a set of generators of~$G$; consider the set~$\mathbf{k}^{G}$ where
$$
\mathbf{k}=\{0,1,\dots,k-1\};
$$
let $m$ be a~Bernoulli measure on~$\mathbf{k}^{G}$, i.\,e., the direct product $\prod_G p$, where
$$
\{p_g\},\qquad
g\in S,\quad
\sum_g p_g=1, \quad
p_g\geqslant 0,
$$
is a~probability vector of dimension~$s$.
Assume that there is a~free action of~$G$  on a~Lebesgue space $(X,\mu)$ with continuous invariant measure. Consider the (one-sid\-ed) Markov process~$(y_n)_{n\leqslant 0}$ with the state space~$X$ and probabilities of transitions (since we consider negative moments of time) equal to the probabilities of cotransitions:
$$
\operatorname{Prob}(y|x)=\begin{cases}
p_g& \text{if} \ y=T_gx, \ g\in S,
\\
0 & \text{otherwise};
\end{cases}
$$
the measure~$\mu$ is an invariant initial measure of the Markov process; we also have the Markov measure~$M$ defined on the space~$\mathscr X$ of all trajectories of the process.

Consider  the tail filtration
$\tau=\{\mathfrak A\}_{n=0}^{\infty}$ on the space $(\mathscr X,M)$ determined by this stationary Markov process. It is, obviously, locally finite (since~$S$ is finite), but the Markov process has a~continual state space. However, by Theorem~\ref{th5}, the corresponding filtration can be defined as the tail filtration in the path space of some graded graph, or as the tail filtration of a~Markov chain with \textit{finite state spaces}. Certainly, this Markov chain will no longer be stationary, for instance, since the cardinality of state spaces, i.\,e., the number of vertices on levels of the graph, grows. But, surely, the filtration itself remains stationary, only the approximation is not shift-invariant. Of course, this chain is not unique, and, as was already mentioned, in general there is no canonical choice of such a~chain.

We will give only a~typical  example of constructing such a~graph.

Let $B_n$ be the set of all elements of the group~$G$ that can be written as words of length at most~$n$ in the generators~$S$; here $B_1=S$.
Observe that
$$
gB_{n-1}\subset B_n\quad\text{if }
g \in S.
$$
We introduce the following series of graded graphs $\Gamma_{G,S}\equiv \Gamma$.

The set~$\Gamma_k$ of vertices of the $k$th level of the graph~$\Gamma$, $k\geqslant 1$,
consists of all functions on the ball~$B_k$ with values in~$\mathbf{k}$;
an edge joins a~vertex $v\in \Gamma_k$ with a~vertex $u\in \Gamma_{k-1}$ if~$u$, regarded as a~function on~$B_{k-1}$, is the restriction of~$v$, regarded as a~function on~$B_k$, to $B_{k-1} \subset B_k$. An equipment on the edge $(u,v)$ is defined in this case as the product of the numbers~$p_g$ corresponding to the product of generators completing the word~$u$ to the word~$v$.

Thus we have constructed an equipped graded graph~$\Gamma_{G,S}$. We will call it the
 \textit{graph of words of the given group with generators $(G,S)$}. This construction admits useful generalizations, e.\,g., one can consider the graph of words corresponding to the sequence of balls~$B_{n_k}$, where
$\{n_k\}$ is a~sequence of positive integers going to infinity.

We define a~homomorphism of the space~$\mathscr X$ onto the path space~$T(\Gamma_{G,S})$ of the constructed graph as follows: a~trajectory of the process is a~sequence of configurations~$\{\phi_n\}_{n<0}$ on the group. Consider the map
$$
\Phi\colon \{\phi_n\}_n \to \{\phi_n|B_n\}_n.
$$

\begin{theorem}
\label{th11}
1) The image of the space~$\mathscr X$ under the homomorphism~$\Phi$ is the path space of the graph of words~$\Gamma_{G,S}$.

2) The homomorphism~$\Phi$ agrees with the equipment,
i.\,e., the parameters of the Markov chain (transition probabilities) coincide with the cotransition probabilities on the path space. Hence~$\Phi$ maps the measure~$M$ to a~measure on the path space that agrees with the equipment.

3) The homomorphism~$\Phi$ of the space~$\mathscr X$ to the path space~$T(\Gamma_{G,S})$ defines a~homomorphism of the space of configurations,
i.\,e., trajectories of the Markov chain, onto the path space of the graph, and the filtration is homomorphically and measure-preservingly mapped onto the tail filtration of the path space.
\end{theorem}

\begin{proof}
Follows immediately from our definition of the graph and its structure.
\end{proof}

The answer to the important question of  \textit{when this homomorphism is an isomorphism} depends on the group, on the choice of a~sequence~$\{n_k\}_k$, and on the transition probabilities, more exactly, on whether the shifts of the balls in the random walk cover the whole group. In the general case, one should increase the size of the balls, which dictates the choice of the sequence~$\{n_k\}_k$. We will return to this question later, and now, in the upcoming sections, consider concrete examples.

It is useful to observe that this construction can also be used for semigroups; for instance, one can take~$S$ to be a set of generators without inverses. The resulting graph corresponds to the filtration of the process constructed from the group, but  it is simpler, and it sometimes suffices to analyze this simpler graph in order
to study  the process itself.

Below we give examples of applying this scheme.

\subsection{Walks over orbits of a~free group}
\label{ssec6.2}
Historically the first example of a~nonstandard dyadic filtration~\cite{59} involved the free group
$G=F_2$ with two generators $(a,b)$
and the space $X=2^{F_2}$ endowed with a~Bernoulli measure (the infinite product of the measures
 $(1/2,1/2)$). We considered the random walk over the trajectories (orbits)
 of the Bernoulli action of the free semigroup with two generators, i.\,e., the Markov process with the transitions $y(g)\to y(ag)$ and $y(g)\to y(bg)$
having the probabilities $(1/2,1/2)$.

Let us cite the paper~\cite{59}: ``\textit{The constructed Markov process has the following paradoxical property, which can be loosely stated as follows. Assume that the process of development of mathematics is the same as the constructed Markov process. If each year,  starting from~$-\infty$, \textit{Abstract Journal of Mathematics}\footnote{Which, unfortunately, has not survived to modern times.}  publishes a~volume containing all essentially new results (i.\,e., results independent of the previous ones) discovered in this year, then reading all these volumes does not allow one to fully recover the picture of mathematics for each given year. Of course, it is understood that in year~$-\infty$ there were no discoveries.}''

Let us add that the above is true for any way of writing the volumes,  i.\,e., for any choice of an independent complement to the previous volumes. In other words, there are essential events
(i.\,e., events of positive measure) occurred in the given year that are not measurable with respect to any system of independent complements to the partitions of this dyadic filtration.

The construction of the graph
described above for the general case determines an approximation of this process by finite Markov chains. One of the corresponding multigraphs looks as follows: the number of vertices of level~$n$ is equal to~$2^{2^n}$, which corresponds to the set of all functions on all words of length~$n$ with values~$0$,~$1$. Every function-vertex~$v$ is joined by an edge with two functions-vertices~$u$ and~$u'$ of the previous level if~$u$ (respectively,~$u'$) is the function on words of length~$n-1$ obtained from~$v$ by considering it on the indices beginning as~$(0,*,*,\dots,*)$ (respectively,~$(1,*,*,\dots,*)$). But since functions are arbitrary, this means that the set of vertices of the $n$th level is the set of all pairs of vertices of the $(n-1)$th level. In other words, the graph of words
in this case is the
\textit{graph of ordered pairs}~\cite{82},~\cite{91}
(see Fig.~\ref{fig9}).\footnote{The possibility to use this graph in the study of  random walks on free semigroups and the corresponding filtrations was not observed in~\cite{91}; this link will be considered in another paper.} From here it is not difficult to deduce that the corresponding filtration is nonstandard (see below) and to obtain a~number of other properties.

\begin{landscape}

\begin{figure}[h]

\vspace*{2cm}

\hspace*{-12cm}

\centering
 \includegraphics[scale=0.9]{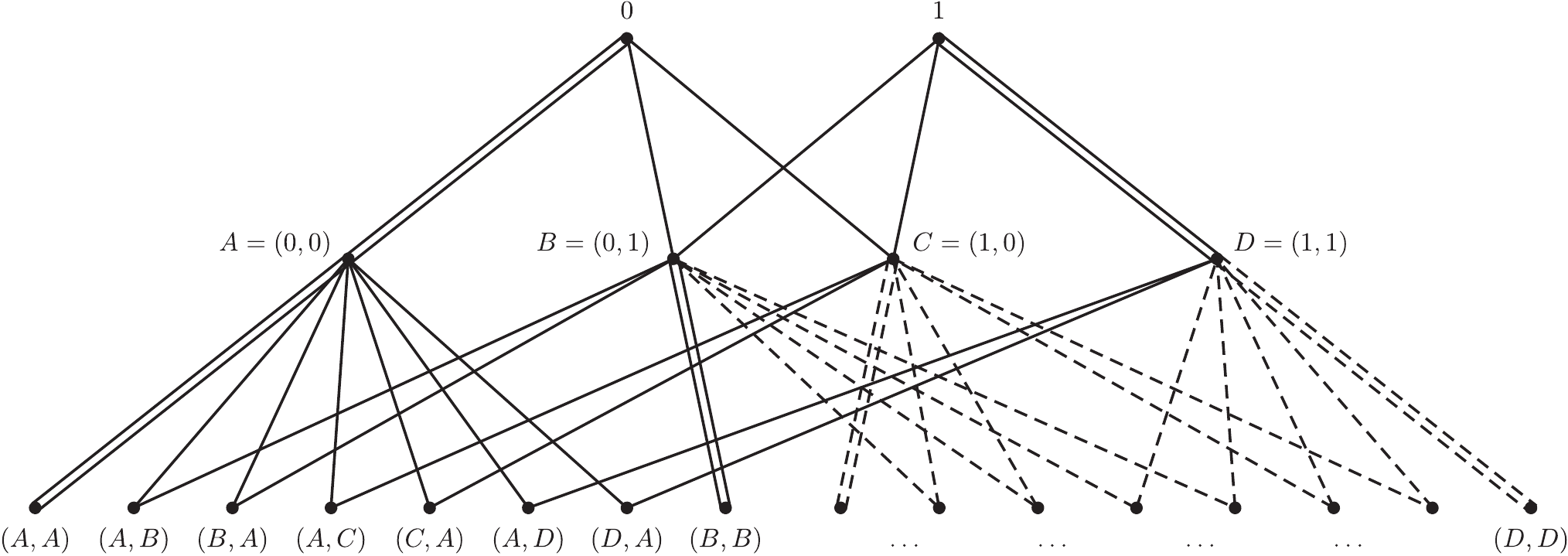}

\vspace*{3mm}

\caption{The graph of ordered pairs}
\label{fig9}

\end{figure}

\end{landscape}

Exactly the same numerical analysis was used in~\cite{59} to prove the nonstardardness of filtrations arising in the study of actions of locally finite groups, e.\,g., a~Bernoulli action of the group of roots of unity of orders of the form~$2^n$, or a~direct sum of groups of order~$2$. The main argument proving the nonstandardness in this case was that the orders of orbits of the groups of automorphisms of the dyadic tree in these examples was substantially smaller than the order sufficient for the measure to concentrate near one orbit. This argument was used by the author to introduce the notion of the entropy of a~dyadic filtration. After the author's paper~\cite{59} was published, where nonstandard dyadic filtrations and the standardness criterion appeared for the first time and the notion of entropy of dyadic filtrations was defined (see~\cite{69}), A.\,M.~Stepin observed that in examples related to actions of locally finite groups (like the infinite sum of groups of order~$2$), one can also use the entropy of the action as an invariant of a~trajectory dyadic sequence. In~\cite{64},~\cite{65},~\cite{69}, an estimate on the growth of the orders of groups is given for which the entropy of the action coincides with the entropy of the filtration. Further research~\cite{22} showed that this estimate is sharp. However, the entropy of filtrations in the sense of~\cite{62},~\cite{69} is of a~much more general nature.

Let us also add that the tree of paths~$T_u$ with root at a~given vertex~$u$ of level~$n$ together with the sequence of zeros and ones on the leaves of the tree saying at which of the two vertices~$a$,~$b$ of the first level the given path ends, is exactly what in~\cite{59} was called the
\textit{universal projection} of the cylinder set~$\xi_a$
or~$\xi_b$ of vertices of the first level to the initial fragment of the filtration,
i.\,e., to the first $n$ partitions. The fact that the universal projection does not stabilize, i.\,e., that for different vertices~$u$ of level~$n$ the  trees do not approach one another in measure, means, according to the standardness criterion, exactly that the filtration is nonstandard.

In conclusion of this example, note that although (by the theorem on a~Mar\-kov realization) filtrations generated by different RWRS actions can be represented as filtrations of Markov chains with finitely many states, an explicit form of this realization is known only in a~small number of examples. This is also true for adic realizations of the corresponding groups and subgroups. Such a~realization is not even known  for all Abelian groups, which is discussed in the next section.

\subsection{Walks over orbits of free Abelian groups}
\label{ssec6.3}
Walks over orbits of infinite Abelian groups provide
an example of a~filtration that is simpler in one respect and more complex in another. We will also consider a~model example.

If $G={\mathbb Z}^d$ and $S=(e_1,e_2,\dots,e_d)$, where $e_i$ are the unit coordinate vectors, then the corresponding graph looks as follows: the set of vertices of  level~$n$ is the set of all functions on the integer simplex
$$
\Sigma_d^n=\biggl\{\{r_1,r_2,\dots,r_d\}\colon \sum_i r_i=n,\
r_i\in {\mathbb Z}_+\biggr\}
$$
with values~0 or~1, and an edge joins a~vertex-function~$u$ with the vertices-functions of the previous level that are the restrictions of~$u$ to the subsimplices obtained by removing one of the faces.

Consider the case $d=1$, $G=\mathbb Z$,  $S=\{-1,+1\}$; the corresponding graph is the
\textit{graph of words for~$\mathbb Z$}. The vertices of the $n$th level are the words of length~$n$:
$$
(\epsilon_1,\epsilon_2,\dots,\epsilon_{n-1},\epsilon_n),\qquad
\epsilon \in 0,1,
$$
and the edges lead to the following two subwords:
$$
(\epsilon_1,\epsilon_2,\dots,\epsilon_{n-1}),\quad
(\epsilon_2,\dots,\epsilon_{n-1},\epsilon_n),\qquad
n=2,3,\dots\,.
$$

The graph of words for~$\mathbb Z$ is an Abelian analog of the graph of ordered pairs. This is a~dyadic (multi)graph, the $n$th level  has~$2^n$ vertices, each having two predecessors and four successors (see Fig.~\ref{fig10}). The~natural central measure is the uniform measure on the vertices of each level.

\begin{figure}[h!]

\centering
 \includegraphics{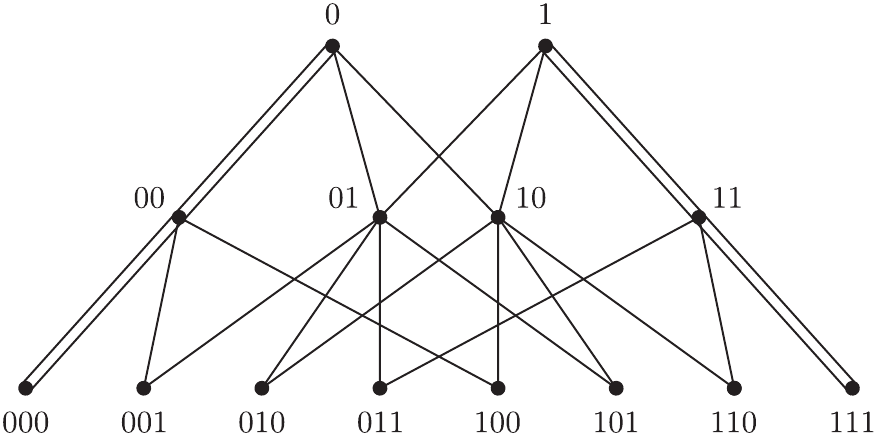}
\caption{The graph of binary words for the group $\mathbb{Z}$}
\label{fig10}
\end{figure}

The adic shift on this graph is quite interesting, it somewhat resembles the
$(T,T^{-1})$-transformation (where $T$ is a~Bernoulli automorphism). In the well-known paper~\cite{27}, it is proved that this transformation is not Bernoulli and even not weakly Bernoulli. But here we are interested in the properties of the filtration; the relation of these properties to ergodic characteristics of transformations is not yet sufficiently studied. The nonstandardness of the filtration for $(T,T^{-1})$ was conjectured by the author and first proved in~\cite{22} (see also~\cite{84}).
An extension of the results of~\cite{27} to the groups~${\mathbb Z}^d$ was obtained in~\cite{4}. The nonstandardness of the filtrations for all free Abelian groups with $d \geqslant 1$ and for nilpotent groups was proved in~\cite{84}, as well as the fact that the scaled entropy of these filtrations depends on~$d$; we will return to these graphs and their generalizations to other groups in subsequent papers.

\subsection{Lacunarity}
\label{ssec6.4}
The theorem on lacunary isomorphism~\cite{58} says that every dyadic ergodic filtration~$\{\xi_n\}_n$ is lacunary standard, i.\,e., there exists an increasing subsequence~$\{n_k\}$ of positive integers such that the filtration~$\{\xi_{n_k}\}_k$ is a
$\{2^{n_k-n_{k-1}}\equiv r_k\}$-adic standard filtration. The sequences~$\{n_k\}$ for which this is true constitute the so-called \textit{scale}, which is an invariant of the original filtration. The scale of a~standard dyadic sequence consists of all sequences and is called complete. Of course, together with every sequence, the scale contains all sequences cofinal with it (i.\,e., differing with it in finitely many positions); also, the scale is monotone, since a~subsequence of a~standard filtration is obviously standard. It seems that these two properties exhaust the conditions on a~scale. In~\cite{69},~\cite{40}, nonstandard dyadic sequences with various scales are constructed. For instance, $\{\xi_{2k}\}$ can be a~standard 4-adic filtration, while the filtration~$\{\xi_k\}_k$ is nonstandard, etc.

By the same method one can prove a~general theorem on arbitrary inhomogeneous filtrations, not necessarily locally finite
(see~\cite{82}).

\begin{theorem}
\label{th12}
Every ergodic filtration is lacunary standard.
\end{theorem}

Since standardness can be violated along an arbitrary sequence, classification of filtrations can hardly be efficient.

However, the role of lacunary theorems is different: they measure the degree of closeness of a~filtration to a~standard one.

In~\cite{64}, the notion of an automorphism with a complete (dyadic) scale was introduced; in terms of this paper, this is an automorphism for which there exists a~dyadic realization with a~standard tail filtration (with respect to the given measure).\footnote{In~\cite{64}, there was an attempt to apply the results related to dyadic filtrations (standardness, scale, etc.) to the study of ergodic properties of automorphisms. However, in contrast to such applications to actions of locally finite groups, for groups that are not locally finite the requirement of dyadicity or homogeneous periodicity in the spirit of Rokhlin's lemma complicates all constructions. As the further development of the subject showed, one should pass from homogeneous filtrations to semihomogeneous ones; it is in this way that adic realizations of group actions are constructed, and the notions of standardness and complete scale become more natural.}
Later, A.\,B.~Katok~\cite{28} (see also~\cite{29}) introduced and studied the notion of a~standard automorphism as an automorphism that is monotonically equivalent to an automorphism with universal discrete spectrum. It would be interesting to find out whether these notions coincide.

\subsection{Scaled and secondary entropy}
\label{ssec6.5}
Another measure of closeness to standardness is the scaled entropy of a~filtration (and its analog, the scaled entropy of a~group action).

Recall that the condition of the standardness criterion is that the metric~$d^n_{\rho}$ on the space of measured trees converges in measure to a~degenerate metric.

\begin{proposition}
\label{pr7}
If a~filtration is ergodic, then the sequence of metrics~$d^n_{\rho}$ either converges in measure to a~degenerate metric, or has no limit.
\end{proposition}

The reason of this simple, but important fact is that the existence of a~limit that is not reduced to a~one-point space contradicts the ergodicity, since otherwise nontrivial limiting functionals would appear.

Thus in the case of a~nonstandard filtration there arises a~sequence of metric spaces that has no limiting metric space, but one can track the $\epsilon$-entropy of this metric space as a~function of~$n$. It turns out that the asymptotics of this sequence for some normalization does not depend on the choice of a~metric, and thus its class is an invariant of the filtration. It is in this way that this invariant was defined in~\cite{65},~\cite{59} in the case of a~dyadic filtration with exponential (Kolmogorov) normalization. In~\cite{69},~\cite{84}, it
was defined in the general case and, for reasons related to the stationary case, called secondary entropy. In the sense of this definition and under a correct normalization, the secondary entropy of the filtration corresponding to the transformation $(T,T^{-1})$,
by a~theorem from~\cite{23}, equals the entropy~$h(T)$.
In~\cite{91}, it was proved that the scaled entropy of the adic automorphism\footnote{The definition of this entropy was given in the author's papers~\cite{73},~\cite{74}, where its main properties were formulated, proved later by P.~B.~Zatitskiy~\cite{100},~\cite{101}.} for some measures on the graph of ordered pairs is equal to the entropy (scaling) sequence (with respect to the same measure) of the tail filtration. These theorems should be extended to arbitrary groups.

\subsection{A project of classification of nonstandard filtrations}
\label{ssec6.6}
Kolmogorov's zero--one law (the triviality of the $\sigma$-algebra of the intersection of the $\sigma$-algebras of a~filtration) is the simplest condition on the behavior of a~filtration at infinity. Standardness, or Bernoulli property, is, on the contrary, the strongest such condition: it guarantees that at infinity there is no complexity, since all conditional distributions converge in the strongest sense. Here, in contrast to the traditional way of measuring the closeness of a~process to a~process of independent random variables, this closeness is expressed not by the rate of convergence to zero of the correlation coefficients or other characteristics, but by the smallness of the difference (in one or another metric) between  the tree structures of conditional measures. Thus one can speak about ``higher zero--one laws,'' meaning assertions that conditional distributions approach one another in some metric, which are intermediate between the ordinary zero--one law and standardness (Bernoulli property).

As observed in the previous section, such a~scale of conditions apparently exists.

Ornstein's very weak Bernoulli (VWB) condition is also a~law of this type, but, as shown by examples, it does not coincide with the standardness condition: in~\cite{16}, an example is given of a~non-Bernoulli process with a~standard ergodic dyadic filtration, and our examples show that a~Bernoulli automorphism can have a~Markov generator with respect to which the past is not standard; another example is given in~\cite{23}. However, these are only examples, and a~more complete analysis of the situation is yet to be carried out. Note also that conditions may depend on what metric on the space of measured trees we choose.

Recall that  the scheme of the standardness criterion involves considering functions on
$X/\xi_n$ with values in the space of measured trees of rank~$n$, and the condition  is that  for sufficiently large~$n$ the values of the functions approach one another with respect to a~metric on the space of trees,
i.\,e., the sequence of functions converges to a~constant in measure.
 In other words, the images of the measures on the spaces of trees converge to a~$\delta$-measure. Instead of functions on the quotient space $(X/\xi_n,\mu_{\xi})$, it is convenient to consider their liftings to the space $(X,\mu)$ itself, and then on $(X,\mu)$ we obtain semimetrics~$\rho_n$ in which the distance between points is the distance between their images, the values of the functions,
i.\,e., the corresponding trees.

The condition of the standardness criterion is  that the sequence of these metrics collapses to a~point, and if this condition is satisfied, then two finitely isomorphic filtrations are isomorphic. If there is no standardness, then the functions do not converge to a~constant in measure, and this means that the sequence $(X,\mu,\rho_n)$, $n=1,2,\dots$, where $\rho_n$ is the sequence of semimetrics, does not converge and there is no limiting metric on
$(X,\mu)$.

What criteria can we use to establish an isomorphism of filtrations? Necessary conditions are the coincidence of the entropy, of the scale, etc.\ (see~\cite{64}).  But can we state sufficient conditions? For this, one should understand what does it mean that  two sequences of (semi)metric spaces for which there is no limiting space are asymptotically isomorphic.

The conjecture we formulate below is that sometimes finite isomorphism supplemented by a~similar asymptotic isomorphism implies isomorphism. This conjecture reduces the study of all invariants of nonstandard filtrations (like, e.\,g., entropy) to a~unified context.

To state this conjecture more precisely, recall the following general idea related to asymptotic isomorphism suggested in~\cite{82}.

Consider a~so-called metric triple, or, in another terminology, an admissible triple $\theta=(X,\mu,\rho)$, that is, a~space~$X$ endowed with a~measure~$\mu$ and a~metric~$\rho$ that agree with each other: the metric~$\rho$ is measurable as a~function of two variables on the space $(X,\mu)$,
the space $(X,\rho)$ is separable as a~metric space, and~$\mu$ is a~Borel probability measure on this metric space $(X,\rho)$. By the way, note that, in contrast to the traditional practice of fixing a~metric on a~space and considering various Borel measures on the resulting metric space, the author has repeatedly advocated the usefulness of another approach: fixing a~measure and varying metrics on the resulting measure space~\cite{71}.

In~\cite{19},~\cite{71},~\cite{72}, a~theorem is proved on a~complete system of invariants of metric triples with respect to the group of measure-preserving isometries.

\begin{definition}
\label{def9}
The matrix distribution~$D_{\theta}$ of a~metric triple
$\theta=(X,\mu,\rho)$ is the measure in the space of infinite nonnegative real matrices defined as the image of the Bernoulli measure~$\mu^{\infty}$ on~$X^{\infty}$ under the map
\begin{align*}
M_{\theta}\colon (X^{\infty},\mu^{\infty}) &\to
({\mathbb M}_{\mathbb N }(\mathbb{R}_+),D_{\theta}),
\\*
\{x_n\}_n&\mapsto \{\rho(x_i,x_j)\}_{(i,j)}.
\end{align*}
\end{definition}

\begin{theorem}
\label{th13}
Two metric triples
$$
\theta=(X,\mu,\rho)\quad\text{and}\quad
\theta'=(X',\mu',\rho')
$$
with nondegenerate measures ($\mu(U)>0$, $\mu(U')>0$ if $U$,~$U'$ are nonempty open sets in~$X$,~$X'$) are isomorphic with respect to the above equivalence  (i.\,e., there exists an invertible isometry
$T\colon X\to X'$ preserving the measure: $T\mu=\mu'$) if and only if their matrix distributions coincide:
$$
D_{\theta}=D_{\theta'}.
$$
\end{theorem}

In~\cite{70},~\cite{76},~\cite{85}, a~similar theorem is proved on the classification of arbitrary measurable functions of several variables via matrix distributions.

\begin{definition}
\label{def10}
Assume that for a~sequence of admissible metric triples on the same measure space $(X,\rho_n,\mu)$ there exists a~weak limit   $\lim_n D_{\theta_n}$ of  matrix distributions (regarded as measures on the space of matrices~${\mathbb M}_{\mathbb N}(\mathbb{R}_+)$ with the ordinary weak topology). This limit will be called the virtual matrix distribution of the sequence of metric triples. It is an invariant with respect to a~sequence of isometries~$I_n$,
$n=1,2,\dots$\,.
\end{definition}

Note that if the metric triples do not converge to any metric triple, then the limit of the matrix distributions, if it exists, is not the matrix distribution of any metric triple. On the other hand, the set of all matrix distributions of metric triples is not weakly closed, hence one can use virtual matrix distributions for classification of some diverging sequences of metric spaces.

Now we can state a~conjecture on isomorphism of filtrations.

\begin{conjecture}
\label{con1}
Consider two finitely isomorphic ergodic filtrations~$\tau$,~$\tau'$
and assume that
{
\begin{enumerate}
\renewcommand{\labelenumi}{{\rm\arabic{enumi})}}
\renewcommand{\theenumi}{{\rm\arabic{enumi})}}

\item\label{lab1} they both are nonstandard;

\item\label{lab2} they both have virtual matrix distributions of the sequences of (semi)metric triples.
\end{enumerate}
}

If these matrix distributions coincide, then the filtrations are isomorphic.
\end{conjecture}

The validity of this conjecture, along with the standardness criterion, would allow us to believe that the classification of nonstandard filtrations is completed in the above terms and under the above assumption on the validity of condition~\ref{lab2} for an arbitrary nonstandard filtration. The answer would be that the virtual matrix distribution together with the finite invariants exhaust all invariants of ergodic locally finite filtrations. However, condition~\ref{lab2}
(more exactly, the existence of a~virtual matrix distribution) can hardly always be  satisfied.

Hence the main problem is to find a~complete system of invariants for a~diverging sequence of metrics on a~measure space. For the dyadic case, it suffices to consider invariants of sequences of metrics arising from iterations of metrics on the graph of unordered pairs, which we discuss in the next section.

But in the example of a~Markov chain with a~nonstandard filtration from Sec.~\ref{ssec6.2}, the virtual matrix distribution does exist. Namely, the sequence of matrix distributions of triples on the two-point space,
$$
\rho_{n}(a,b)=0, \quad
n\equiv 0 \ (\operatorname{mod} 2); \qquad
\rho_{n}(a,b)=1, \quad
n\equiv 1 \ (\operatorname{mod} 2),
$$
converges in the sense of Ces\`aro to the virtual matrix distribution of this sequence, which is a~measure on the set of infinite symmetric matrices with zeros on the diagonal and with independent entries~$r_{i,j}$, $i>j$,
taking the values~0 and~1 with equal probabilities.

\section{Projective limits of simplices, invariant measures, and the absolute of a~graph}
\label{sec7}

This section is devoted to the theory of filtrations in Borel spaces: in this case,  a~priori there is no measure, but only an equipment, i.\,e., conditional measures. The main problem is to describe all Borel probability measures that agree with the equipment.

We describe the fourth language of the theory of locally finite filtrations: that of projective limits of finite-dimensional simplices. One may say that the category of projective limits of finite-dimensional simplices is equivalent to the category of locally finite filtrations in Borel spaces, or the category of multigraphs with equipments, or the category of Markov compacta with cotransition probabilities. Although this is an extremely natural relation, nevertheless, as far as the author knows, until recently the functorial equivalence of the theory of Borel and metric filtrations on the one hand and the convex geometry of projective limits of simplices on the other hand almost has not been discussed in the literature in this context. We use some definitions and facts from the recent paper~\cite{78} by the author (see also~\cite{36},~\cite{5},~\cite{98},~\cite{52}). The analysis of the theory of locally finite filtrations in Borel spaces, as well as the theory of central and invariant measures on path spaces of multigraphs, or the theory of Markov measures with given cotransitions, shows that the adequate and geometric language of projective limits of finite-dimensional simplices is the most appropriate language for all these theories. Below we will describe a~correspondence between the problem of describing the measures that agree with an equipment and, in particular, of describing the central measures, and the geometry of projective limits of simplices (for more details, see~\cite{78}).

\subsection{Projective limits of simplices, extreme points}
\label{ssec7.1}
First, we will show how, given a~pair $(\Gamma,\Lambda)$,
i.\,e., an equipped graded graph, one can define, in a canonical way, a~projective limit of finite-dimensional simplices. Later we will see that there is also a~reverse transition from projective limits to equipped graphs.

Denote by~$\Sigma_n$ the finite-dimensional simplex of formal convex combinations of vertices
 $v \in\Gamma_n$ of level~$n$. It is natural to view the simplex as the set of all probability measures on the set of vertices of~$\Gamma_n$. We define affine projections
$$
\pi_{n,n-1}\colon \Sigma_n \to \Sigma_{n-1};
$$
it suffices to define them for every vertex $v \in \Gamma_n$. Obviously, these projections can be regarded as a~system of cotransition probabilities~$\Lambda$, and the images of vertices~$v$ are points of the previous simplex,
i.\,e., probability vectors:
\begin{equation}
\label{eq7.1}
\pi_{n,n-1}(\delta_v)=\sum_{u: u\prec v} \lambda_v^u \delta_u;
\end{equation}
this map can be extended by linearity to the whole simplex~$\Gamma_n$. The vertex~$\varnothing$ gives rise to the zero-dimensional simplex consisting of a~single point. Degeneracies are allowed (i.\,e., projections may glue together different vertices). Now we define projections of simplices with arbitrary numbers:
\begin{gather*}
\pi_{n,m}\colon \Sigma_n \to \Sigma_m,\qquad
m<n;
\\
\pi_{n,m}=\prod_{i=m}^{n+1}\pi_{i,i-1},\qquad
m>n,\quad
m,n \in \mathbb N.
\end{gather*}
The collection of data $\{\Sigma_n, \pi_{n,m}\}$ allows one to define a~projective limit on the one hand, and to recover the graph on the other hand: the vertices of the graph~$\Gamma_n$ are the vertices of~$\Sigma_n$, and the edges (and then also the paths in the graph) are recovered from the nonzero coordinates of the vectors~$\pi_{n,n-1}$, $n \in \mathbb{N}$.

Let ${\mathscr M}=\prod\limits_{n=0}^{\infty}\Sigma_n$ be the direct product of the simplices~$\Sigma_n$, $n\in \mathbb N$, endowed with the product topology.

\begin{definition}
\label{def11}
The space of the projective limit of the family~$\{\Sigma_n\}_n$ of simplices with respect to the system of projections~$\{\pi_{n,m}\}$ is the following subset of the direct product~${\mathscr M}$:
\begin{align*}
\lim_{n\to\infty}(\Sigma_n,\pi_{n,m}) &\equiv
\bigl\{\{x_n\}_n; \ \pi_{n,n-1}(x_n)=x_{n-1}, \ n=1,2,\dots\bigr\}
\\
&\equiv (\Sigma_{\infty},\Lambda) \subset \prod_{n=0}^{\infty}\Sigma_n
=\mathscr M.
\end{align*}
\end{definition}

\begin{proposition}
\label{pr8}
The space of the projective limit~$\Sigma_{\infty}$ is always a~nonempty, convex, closed (and hence compact) subset of the direct product~$\mathscr M$, which is a~(possibly infinite-dimensional) Choquet simplex.
\end{proposition}

The affine structure of the direct product~$\mathscr M$ determines an affine structure on the limiting space; the fact that the set is nonempty and closed is obvious. It remains to check only that an arbitrary point of the limit has a~unique decomposition over its Choquet boundary (see below).

We distinguish between the space of the projective limit and the ``structure of the projective limit,'' meaning that not only the limiting space itself, i.\,e., some infinite-dimensional simplex, is important for us,
 but also the structure of the pre-limit simplices and their projections. In other words, we consider the category of projective limits of simplices in which an object is not a~limiting simplex, but a~sequence of finite-dimensional simplices.

Let us show that the collection of data of a~projective limit of simplices allows one to recover the graph, the path space, and the system of cotransition probabilities, and the projective limit corresponding to this system and constructed according to the above rule coincides with the original one. This will establish a~tautological relation between the two languages:  the language of pairs (a~Bratteli diagram, a~system of cotransitions) on the one hand, and the language of projective limits of finite-dimensional simplices on the other hand.

Indeed, assume that we have a~projective limit of finite-dimensional simplices~$\{\Sigma_n\}$, $n\in \mathbb N$,
and a~coherent system of affine projections
$$
\{\pi_{n,m}\},\qquad
\pi_{n,m}\colon\Sigma_n \to \Sigma_m,\quad
n\geqslant m,\quad
n,m\in \mathbb N.
$$
Take the vertices of~$\Sigma_n$ as the vertices of the $n$th level of the graph~$\Gamma$; a~vertex~$u$ of level~$n$ precedes a~vertex~$v$ of level~$n+1$ if the projection~$\pi_{n+1,n}$ sends~$v$ to a~point of the simplex~$\Sigma_n$ for which the barycentric coordinate with respect to~$u$ is positive. As a~system of transition probabilities, we take the system of  vectors~$\{\lambda_v^u\}$ related to the projections~$\pi_{n+1,n}$ by formula~\eqref{eq7.1}.

In what follows, having a~projective limit of simplices, we will use the graph (of vertices of all simplices) canonically associated with it, the space of paths in this graph, etc., and in the same way will speak about projections of simplices canonically associated with an equipped graph.

Consider an arbitrary projective limit of finite-dimensional simplices:
$$
\Sigma_1 \leftarrow \Sigma_2 \leftarrow \dots\leftarrow
\Sigma_n\leftarrow \Sigma_{n+1}\leftarrow
\dots \leftarrow\Sigma_{\infty}\equiv \Sigma(\Gamma,\Lambda).
$$

First of all, we define limiting projections
$$
\pi_{\infty,m}\colon\Sigma_{\infty}\to \Sigma_m
$$
as the limits $\lim_n \pi_{n,m}$ for each~$m$: obviously, the images~$\pi_{n,m}\Sigma_n$, regarded as subsets in the simplices~$\Sigma_m$, monotonically decrease as~$n$ grows, and their intersections are some sets denoted by
$$
\Omega_m= \bigcap_{n:n>m} \pi_{n,m}\Sigma_n;
$$
these are convex closed subsets of the finite-dimensional simplices~$\Sigma_m$,
$m=1,2,\dots$, and the limiting projections send the limiting compactum~$\Sigma_{\infty}$ epimorphically on these sets:
$$
\pi_{\infty,m}\colon\Sigma_{\infty}\to \Omega_m.
$$

It would be more economical to consider also the projective limit
$$
\Omega_1\leftarrow \Omega_2 \leftarrow \dots\leftarrow \Omega_n
\leftarrow \dots\leftarrow \Omega_{\infty}=\Sigma(\Gamma,\Lambda)
$$
itself with the epimorphic projections~$\pi_{n,m}$ restricted to the subset~$\Omega_n$ and, by definition, with the same limiting space. However, to find the sets $\Omega_n$ explicitly is an interesting and difficult problem equivalent to the main problem of describing all invariant measures.\footnote{In particular, an explicit form of the compacta~$\Omega_n$ is known in very few cases, even among those where the central measures are known. Even for the Pascal graph, we obtain interesting and rather complicated convex compacta; but, for example, for the Young graph, the author does not know the form of these compacta with the same degree of precision.}

Every point of the limiting compactum, i.\,e.,  a~sequence
$$
\{x_m\}\colon x_m \in \Sigma_m,\quad
\pi_{m,m-1}x_m=x_{m-1},
$$
determines, for every~$m$, a~sequence of measures~$\{\nu_n^m\}_n$
on the simplex~$\Sigma_m$, namely,
$$
\nu_n^m=\pi_{n,m}(\mu_n),
$$
where the measure~$\mu_n$ is the (unique) decomposition of the point~$x_n$ into the extreme points of the simplex~$\Sigma_n$. Of course, the barycenter of each of the measures~$\nu_n^m$
in~$\Sigma_m$ is $x_m$, and this sequence of measures itself becomes coarser, in a~clear sense, and weakly converges in~$\Sigma_m$
as $n \to \infty$ to a~measure~$\nu_{x_m}$ supported by
$\Omega_m \subset \Sigma_m$. Obviously, in this way we obtain all points of the limiting compactum, i.\,e., all measures with given cotransition probabilities.

A point of an arbitrary convex compactum~$K$ is called \textit{extreme} if it cannot be represented by any nontrivial convex combination of points of~$K$; the collection of extreme points is called the Choquet boundary of~$K$ and denoted by~$\operatorname{ex}(K)$. A~point is called
\textit{almost extreme} if it lies in the closure~$\overline{\operatorname{ex}(K)}$ of the Choquet boundary. Recall that an affine compactum in which every point has a~unique decomposition into extreme points (over the Choquet boundary) is called a~Choquet simplex.

Let us give general criteria of extremality and almost extremality for points of a~projective limit of simplices.

\begin{proposition}
\label{pr9}
1) A~point~$\{x_n\}$ of a~projective limit of simplices is extreme if and only if for every~$m$ the weak limit of the measures~$\nu_{x_m}$
in the simplex~$\Sigma_m$ is the $\delta$-measure at the point~$x_m$:
$$
\lim_n \nu_n^m \equiv \nu_{x_m}=\delta_{x_m}.
$$

2) A~point $\{x_n\}$ is almost extreme if for every~$m$ and every neighborhood~$V(x_m)$ of the point $x_m \in \Sigma_m$ there exists an extreme point~$\{y_n\}$ of the limiting compactum for which $y_m \in V(x_m)$.

3) For every point~$\{x_n\}$ of the limiting compactum of simplices there exists a~unique decomposition into extreme points (Choquet decomposition), which is determined by the measures~$\nu_{x_m}$.
\end{proposition}

\begin{corollary}
\label{cor3}
The limiting compactum of a~projective limit of fi\-nite-di\-men\-sion\-al simplices is a~Choquet simplex (in general, infinite-dimensional).
\end{corollary}

In the literature (see~\cite{98}), the question was discussed why a~limit of simplices is a~simplex; even in the finite-dimensional case, this problem is not quite obvious, it was raised by A.\,N.~Kolmogorov. The most correct way to explain this is to refer to the result, formally much more general but in fact equivalent (see, e.\,g.,~\cite{52}), on the uniqueness of a~decomposition of measures with a~given cocycle into ergodic components, i.\,e., use the probabilistic interpretation of simplices.

It is easy to prove that the converse is also true: every separable Choquet simplex can be represented as a~projective limit of finite-dimensional simplices, but, of course, such a~representation is not unique.

Slightly digressing from the main focus, note that the simplex of invariant measures with respect to an action of a~nonamenable group on a~compactum is separable, though its possible approximation is not generated by finite approximations of the action; thus there arises a~non-trajectory finite-dimensional approximation of the action, which, apparently, has not been considered in the literature.

\begin{remark}
\label{rem1}
The first two claims of Proposition~\ref{pr9} can perhaps be extended to projective limits of arbitrary convex compacta.
\end{remark}

Recall that the class  of separable Choquet  simplices contains
\textit{Poulsen simplices}, for which the set of extreme points is dense; such a~simplex is unique up to an affine isomorphism and universal in the class of all affine separable simplices (see~\cite{72} and the references therein).

\begin{proposition}
\label{pr10}
Consider a~projective limit of simplices with the following property: for every~$m$ the union of the projections
$$
\bigcup_{n,t}\{\pi_{n,m}(t); \ t\in \operatorname{ex}(\Sigma_n),\
n=m,m+1,\dots\}
$$
over all vertices of the simplices~$\Sigma_n$ and all $n>m$ is dense in the simplex~$\Sigma_m$. Then the limiting simplex is a~Poulsen simplex.
\end{proposition}

It is clear that this construction can be carried out by induction, and the extremality criterion obviously implies that in this case the extreme points are dense. With such arbitrariness in the construction, the uniqueness seems surprising. Nevertheless, verifying that a~given projective limit is a~Poulsen simplex is not so easy and is similar to problems related to filtrations.

A simplex whose Choquet boundary is closed is called a~ \textit{Bauer simplex}. Between the classes of Bauer and Poulsen simplices, there are many intermediate types of simplices. In the literature on convex geometry and the theory of invariant measures, this has been repeatedly discussed. However, it seems that these and similar properties of infinite-dimensional simplices in application to projective limits and the theory of graded graphs and corresponding algebras have never been considered. Each of these properties has an interesting interpretation in the framework of these theories. The author believes that the following class of simplices (or even convex compacta) is useful for applications: an \textit{almost Bauer simplex} is a~simplex whose Choquet boundary is open in its closure.

Parallelism between the study of pairs $(\Gamma,\Lambda)$ (a graded graph and an equipment of this graph) on the one hand and of projective limits of simplices on the other hand means that the latter has a~probabilistic interpretation. It is useful to describe this interpretation directly. Recall that in the context of projective limits a~path is a~sequence~$\{t_n\}_n$ of vertices ${t_n \in \operatorname{ex}(\Sigma_n)}$
that agrees with the projections in the following way: the point $\pi_{n,n-1}t_n$
for all $n\in {\mathbb N}$   has a~nonzero barycentric coordinate with respect to the vertex~$t_{n-1}$. First of all, every point $x_{\infty} \in \Sigma_{\infty}$ of the limiting simplex
 is a~sequence of points of simplices that agrees with the projections:
$$
\{x_n\}\colon \pi_{n,n-1}x_n=x_{n-1},\quad
n \in \mathbb N.
$$
As an element of a~simplex,~$x_n$ defines a~measure on the vertices of this simplex, and since all these measures agree with the projections,~$x_{\infty}$ defines a~measure~$\mu_x$ on the path space with fixed cotransition probabilities. Conversely, each such measure defines a~sequence of projections. Thus the limiting simplex is the simplex of all measures on the path space with given cotransitions.

If a~point $\mu\in \operatorname{ex}(\Sigma_{\infty})$ is extreme, this means that the measure~$\mu$ is ergodic,
 i.\,e., the tail $\sigma$-algebra is trivial with respect to the measure~$\mu$ on the path space.

We end this section with the following conclusion.

\textit{The theory of equipped graded multigraphs and the theory of Markov compacta with systems of cotransitions are identical to the theory of Choquet simplices regarded as projective limits of finite-dimensional simplices. All three theories can be viewed as theories of locally finite filtrations in a~fixed basis representation}.

\subsection{The fundamental problem: description of the absolute}
\label{ssec7.2}
\begin{problem}[Main problem]
Given a~projective limit of finite-dimensional simplices, find the Choquet boundary of the limiting simplex.

It will be convenient to call this boundary (in the context of path spaces of graded graphs) the absolute of a~projective limit (of a~Markov chain, a graded graph, etc.). Sometimes the absolute consists of a~single point; in this case it is trivial, but this happens very rarely.

By above, this problem is equivalent to the following one:
\begin{enumerate}
\item[--] given a~Markov compactum with a~system of cotransitions, describe all ergodic measures that agree with these cotransitions,
\end{enumerate}
or
\begin{enumerate}
\item[--] given a~graded graph with an equipment, describe all ergodic measures on its path space that agree with this equipment,
\end{enumerate}

\noindent
or, finally,

\begin{enumerate}
\item[--] given a~filtration on a~Borel space, describe all ergodic Borel measures that agree with the cocycle of this filtration (see Sec.~\ref{sec3}).
\end{enumerate}
\end{problem}

Actually, as we have seen, all these problems are tautologically equivalent, but still they belong to different areas of mathematics. The formulation with Markov chains is essentially due to E.\,B.~Dynkin (see~\cite{11}).\footnote{In~\cite{11}, the terms ``exit boundary,'' ``entrance boundary,'' etc.\ were used. Note that the Martin boundary can also be easily defined in our context, see~\cite{78}.}

For graphs with the canonical equipment, this is the problem of describing the central measures, or traces of
$\operatorname{AF}$-algebras, or characters of locally finite groups; besides, this is also the problem of describing the invariant measures for actions of countable groups, provided that an adic approximation is chosen. This is an enormous circle of problems which covers a~substantial part of asymptotic representation theory, combinatorics, algebra, and, of course, the theory of dynamical systems.

In the case of filtrations, fixing some ergodic measure allows one to regard a~filtration as a~filtration of a~measure space (rather than of merely a~Borel space), with all the problems considered above, whose solution substantially depends on the choice of a~specific measure. But the absolute, in a sense, combines different measures with respect to which one can consider filtrations, equivalence relations, adic actions, etc. Thus the problem is not only to describe the absolute. Here is a~list of natural important additional questions (we state them in terms of graded graphs).

The setting of the problem can be, for instance, as follows:

-- On which graded graphs every measure from the absolute determines a~standard tail filtration on the path space endowed with this measure (it is natural to call such graphs standard)?

-- On which graphs there exists a~measure from the absolute with a~standard tail filtration?

-- On which graphs the absolute contains no nontrivial measures with a~standard tail filtration?

The same series of questions can be posed in connection with entropy and other invariants of filtrations. From the point of view of the metric theory of filtrations and the theory of Markov chains, these questions look especially unconventional, since we ask to what extent can measures with the same conditional (cotransition) probabilities be dissimilar.

\subsection{Several examples. Relation to the theory of locally finite filtrations}
\label{ssec7.3}
We choose several examples of problems of describing the absolute.

1. \textsl{Absolutes of groups}. Consider a~countable finitely generated group~$G$ and choose a~symmetric system of generators: $S=S^{-1}$. By the
\textit{dynamical Cayley graph}, or the \textit{Markov compactum of the random walk}, we mean the graph in which the vertices of level~$n$ are all elements of~$G$ representable as words of length at most~$n$, and edges correspond to the right multiplication by generators. The equipment is canonical,
i.\,e., the set of word representations of a~given element is endowed with the uniform measure. The problem of describing the absolute,
i.\,e., the set of ergodic central measures, is a~considerable extension of the problem of describing the Poisson--Furstenberg boundary of a~random walk. While the Poisson--Furstenberg boundary reduces to a~single point for many groups, the absolute does it very rarely; it has a~structure of a~generalized fiber bundle over the Poisson--Furstenberg boundary. For the free group, the absolute is calculated in~\cite{90}. In this case,  a~natural phase transition from ergodicity to nonergodicity  is discovered. Presently, the absolute of nilpotent groups is under study. A~quite useful tool here is a~generalization of the ergodic method (see~\cite{82}), which reduces finding the central measures to establishing their additional symmetry properties. In this sense, describing the absolute can be viewed as a far-reaching generalization of de~Finetti's theorem on measures invariant with respect to groups of permutations. In all group examples considered so far, the filtrations turned out to be standard, i.\,e., the dynamical graph can be called standard.

2. On the contrary, graphs arising in the approximation theory of dynamical systems
(i.\,e., graphs providing adic realizations of automorphisms with invariant measure) are often nonstandard.

\begin{conjecture}
\label{con2}
A graph providing an adic realization of an automorphism with positive entropy is not standard. The corresponding limiting simplex is a~Poulsen simplex.
\end{conjecture}

The graph of ordered pairs (see Sec.~\ref{sec4}) has many central measures with respect to which the tail filtration is not standard (see~\cite{91}).

Another, more interesting, example of this type is the  ``tower of measures'' graph.

This is the graph of unordered pairs, which appeared in 1969 as a~tower of binary measures. The simplest way to introduce this graph is to define it as a~projective limit of simplices in which the set of vertices of each simplex is the set of all vertices and middle points of all edges of the previous simplex, the initial simplex being the interval $[0,1]$
(see Fig.~\ref{fig11}). The limiting  simplex provides nice interpretations of invariants of sets with respect to nonstandard dyadic filtrations. Moreover, the definition of the tower of measures was motivated by an interpretation of the standardness criterion (more exactly, an interpretation of the values of the universal projection of a~set) for nonstandard situations. The construction of the tower of measures is related to
 constructions of universal spaces in the set-theoretic and algebraic topology.

\begin{figure}[h!]

\centering
 \includegraphics{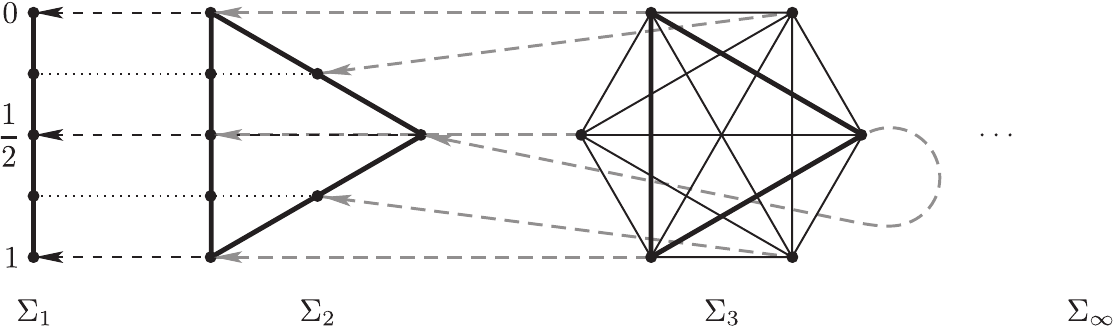}
\caption{The beginning of the tower of measures}
\label{fig11}
\end{figure}

3. An enormous supply of graded graphs, and hence of absolutes, filtrations and their invariants, is provided by the representation theory of $\operatorname{AF}$-algebras and locally finite groups. Each of these objects gives rise to its own Bratteli diagram, the canonical equipment determines a~projective limit of simplices and the limiting simplex of central measures, etc. These links have been studied (but yet without taking into account tail filtrations) both in Borel path spaces and in path spaces endowed with central measures. We only mention that the computation of the absolute,
i.\,e., of the traces and characters, looks isolated in each specific case; up to now, there no sufficiently general methods for solving this problem. For example,  one has not yet managed to prove
the well-known Kerov's conjecture~\cite{30, 31, 32} on central measures in the Macdonald graph, though it differs little from known solved problems. An attempt to separate solvable problems related to the absolute from ``intractable'' ones,  i.\,e., those that  have  no reasonable solution, was made in~\cite{77},~\cite{79},~\cite{82}; it is also reproduced in part in this paper. This is the reason for bringing the notion of standardness of filtrations into these problems. In the papers mentioned above (in more detail, in~\cite{82}), the so-called internal metric on paths of a~graph is introduced, whose role is exactly to separate the two cases. For want of space, we postpone a~further analysis of these relations to future publications.

\section{Historical commentary}
\label{sec8}

In conclusion of this survey, it is appropriate to add several remarks on the history of the problem and some papers, examples, and counterexamples useful in the study of filtrations.

1. In~1958, the famous A.\,N.~Kolmogorov's paper~\cite{34} was published which contained a~definition of entropy as an invariant of the metric theory of dynamical systems. In that paper, Kolmogorov defined entropy in the spirit of Shannon's concept of information transmission and rested on a~filtration (in the old terminilogy, the ``sequence of past $\sigma$-algebras'') that he had often considered  earlier, acknowledging the importance of Rokhlin's theory of partitions used in his considerations. As one can deduce from recollections of various persons, in an earlier lecture course, given by A.\,N.\ in 1957 and, unfortunately, lost, he apparently proved that Bernoulli schemes with different entropies are nonisomoprhic using a~combinatorial description of entropy, which, due to the subsequent definition  suggested by Ya.\,G.~Sinai~\cite{53}, later became the main one. The more so since this definition was at once easily extended to arbitrary groups.\footnote{Here one should mention the master thesis by D.\,Z.~Arov (Odessa University, 1957), mentioned by A.\,N.\ in his second paper~\cite{35} and published only recently (see~\cite{1}), which contains a~combinatorial definition of an invariant of an automorphism close to Sinai's definition. A~full comparison of Kolmogorov's entropy and Arov's entropy was carried out by B.\,M.~Gurevich~\cite{20}.} Kolmogorov's definition from the first paper was much closer to the traditional theory of stationary processes and Shannon's ideas. But a~small error forced A.\,N.\ to write the second paper (see~\cite{35}). Here we quote a~fragment and a~footnote from this second paper by A.\,N., in which the error contained in the first paper was corrected,\footnote{For some reason, the publishers of A.\,N.'s collected works decided not to publish this paper in full, but to combine a~small part of it with the first paper.} since this is directly related to difficulties of the theory of filtrations.

``V.\,A.~Rokhlin pointed out that the proof of Theorem~2 in my paper~(1) tacitly uses the following assumption.
It follows from
$$
{\mathfrak A}_1 \supseteq {\mathfrak A}_2 \supseteq \cdots \supseteq
{\mathfrak A}_n \supseteq\dotsb,\qquad
\bigcap_n {\mathfrak A}_n={\mathfrak N},
$$
that
$$
\bigcap_n({\mathfrak B}\vee{\mathfrak A}_n)=\mathfrak B.\textrm{''}
$$

``By the courtesy of V.~A.~Rokhlin, I reproduce this example. Let
${\mathfrak G}^m$ be the additive group of numbers of the form~$\alpha m^{-\beta}$
($m$~is a~positive integer, $\alpha$~is an integer,
$\beta$~is a~nonnegative integer).
Denote by~$U$ the automorphism of~${\mathfrak G}^6$ that consists in dividing by~$6$; by $M$, the group of characters of~${\mathfrak G}^6$; and by $T$, the automorphism of~$M$ conjugate to~$U$. The subgroups~${\mathfrak G}^2$,~${\mathfrak G}^3$
of~${\mathfrak G}^6$ satisfy the following obvious relations:
\begin{alignat*}{3}
{\mathfrak G}^2&\subset U{\mathfrak G}^2, &\quad
\bigvee_n U^n{\mathfrak G}^2&={\mathfrak G}^6,&\quad
\bigcap_n U^n{\mathfrak G}^2&=0,
\\*
{\mathfrak G}^3&\subset U{\mathfrak G}^3, &\quad
\bigvee_n U^n{\mathfrak G}^3&={\mathfrak G}^6,&\quad
\bigcap_n U^n{\mathfrak G}^3&=0.
\end{alignat*}

Therefore, the corresponding subalgebras~${\mathfrak G}^2$,~${\mathfrak G}^3$
of the algebra~${\mathfrak G}^6$ of measurable sets of the space~$M$ satisfy the quasi-regularity conditions. At the same time,~${\mathfrak G}^2$ has index~3 in~$U{\mathfrak G}^2$ and~${\mathfrak G}^3$ has index~2 in~$U{\mathfrak G}^3$, so
$$
MH(T{\mathfrak G}^2|{\mathfrak G}^2)=\lg 3, \qquad
MH(T{\mathfrak G}^3|{\mathfrak G}^3)=\lg 2.
$$

Here
$$
\bigcap_n({\mathfrak G}^2\vee T^n{\mathfrak G}^3)\ne {\mathfrak G}^2,\qquad
\bigcap_n({\mathfrak G}^3\vee T^n{\mathfrak G}^2)\ne
{\mathfrak G}^3.\textrm{''}
$$

This example shows the lack of continuity when passing from a~filtration to its supremum with a~given partition. It is interesting that this difficulty is absent in the theory of increasing sequences of partitions ($\sigma$-algebras). I mention this also because the theory of invariant partitions (in our terms, the theory of stationary filtrations) became an object of interest for Rokhlin,  Sinai, and many subsequent authors (see~\cite{47}). In our view, this deep subject is not yet exhausted, and its relations with the theory of groups of unitary or Markov operators, with scattering theory should be studied systematically. On the other hand, as one can see from the survey, the applicability of the theory of filtrations is much wider than the stationary case.

2. The question of whether there exists an ergodic filtration finitely isomorphic, but not isomorphic to a~Bernoulli one appeared as early as in the 1950s, e.\,g., in the work by M.~Rosenblatt (see~\cite{49}). In slightly different terms, it was discussed in N.~Wiener's book \textit{Nonlinear Problems in Random Theory}~\cite{97} (chapters \textit{Coding} and \textit{Decoding}). Actually, \textit{Decoding} contains a~negative answer to this question along with tedious calculations. More exactly, it is claimed there that if a~stationary Gaussian random sequence~$\alpha_n$, $n<0$, has the property that for every~$n$, for the $\sigma$-algebra~${\mathfrak A}_n$ generated by the variables with indices~$<n$, there exist~$|n|$ Gaussian random variables independent of each other and independent of~${\mathfrak A}_n$ that together with~${\mathfrak A}_n$ generate the $\sigma$-algebra of all measurable sets, then the original stationary sequence is isomorphic to a~one-sided Bernoulli shift. However, this is not true, since there exist nonstandard (and hence nonisomorphic to Bernoulli ones) stationary filtrations.

I must say that in my first paper~\cite{58} I also claimed (for dyadic filtrations) that nonstandard filtrations do not exist (without giving any arguments or calculations to support this claim), but soon~\cite{59} corrected this error and suggested the first series of examples of nonstandard dyadic sequences, as well as the standardness criterion. Thus classification of filtrations became a~meaningful problem.

3. Speaking about the prehistory of the theory of filtrations, one should begin with the remark that the problem of analyzing the structure of stationary one-sided filtrations was raised in the 1960s by V.\,A.~Rokhlin in the context of the study of endomorphisms (noninvertible measure-preserving transformations): in this field, many important results are due to him (see~\cite{48}). He attracted to it B.\,G.~Vinokurov and his  Tashkent pupils (B.\,A.~Rubshtein, N.\,N.~Ganikhodzhaev, etc.), who published a~series of interesting papers on this subject
(see~\cite{95},~\cite{94},~\cite{50},~\cite{51}). On the other hand, a~direct necessity to study homogeneous (e.\,g., dyadic) filtrations appeared later in connection with the problem of trajectory isomorphism of automorphisms, also raised by Rokhlin irrespectively of filtrations. It was successfully studied by R.\,M.~Belinskaya~\cite{2}. In fact, by that time, the positive answer to the problem of trajectory isomorphism of all ergodic transformations with invariant measure had been already obtained by H.~Dye~\cite{8}, a~fact that became known to the participants of Rokhlin's seminar on ergodic theory only later, when I proved the theorem on lacunary isomorphism, which immediately implied Dye's theorem. The techniques used in~\cite{8} were based on the theory of $W^*$-algebras, in contrast to an approximation (essentially combinatorial) proof in~\cite{58},~\cite{2}. Much later, the same result was proved for all actions of amenable groups with invariant measure~\cite{43},~\cite{3}. I was interested in trajectory theory in connection with the theory of factors even earlier, and some results for the more general case of quasi-invariant measures were obtained simultaneously with results of W.~Krieger and A.~Connes. For set-theoretic intersections of filtrations, I used the term ``locally measurable partitions,'' V.\,A.~Rokhlin suggested the term ``tame partitions,'' and the generally accepted term is ``hyperfinite partitions.''

4. Similar problems have also been raised in probability theory itself, long ago and repeatedly. For instance, the question of what stationary processes can be obtained by encoding Bernoulli schemes was studied by M.~Rosenblatt, K.~It\^o~\cite{45}, and many others. The clearest formulation of the problem was developed in the framework of information theory and ergodic theory: what automorphisms or endomorphisms are factors of Bernoulli automorphisms? For autormophisms, the problem was solved in the remarkable paper by D.~Ornstein~\cite{42}, who proved that a~necessary and sufficient condition is the condition incongruously called VWB (very weak Bernoulli), while it would be more appropriate to call it ``Ornstein's condition.'' For endomorphisms, it is not yet completely solved. Ornstein's condition resembles in nature the standardness condition, which appeared in another connection in approximately the same time~\cite{59},~\cite{66}. The similarity is in the fact that both are requirements on the behavior of the collection of conditional measures of finite segments of a~process given a~fixed past starting from some time, namely, one requires that these conditional measures approach one another as this time goes to~$-\infty$. Obviously, any interpretation of a~weakening of the independence condition can be expressed exactly in these terms, but still this method differs from various conditions of weak dependence which existed earlier in traditional schemes of the theory of random processes (Kolmogorov's, Rosenblatt's, Ibragimov's conditions, etc.)\ and are usually stated as a~decrease of correlation coefficients, measure densities, etc. But in our approach, the leading role is played by the system of conditional measures and a~metric that controls the convergence of systems of conditional measures. In Rosenblatt's condition, this is the Kantorovich metric (rediscovered by him specifically in this case) constructed from the Hamming metric on trajectories. In the case of standardness, this is a~modified Kantorovich metric taking into account the (tree-like) structure of conditional measures. The difference between them is substantial, but the problems and classes of processes are also different: Ornstein's condition concerns stationary processes on~$\mathbb Z$ (though is stated for their restrictions to~${\mathbb Z}_+$) and deals with the properties of a~two-sided process, while standardness characterizes the properties of a~one-sided filtration.\footnote{Later, D.~Ornstein and B.~Weiss~\cite{43},~\cite{44}, D.~Rudolph, and others extended this theory from the group~$\mathbb Z$ to actions of arbitrary amenable groups with invariant measure, but this generalization has no parallel in the theory of filtrations in the form it is used for the group~$\mathbb Z$
(the existence of a~notion of ``past''). However, filtrations that arise as a~tool in trajectory approximation (see one of the examples of filtrations in Sec.~\ref{sec1}) are undoubtedly related also to Ornstein's condition for general groups.}

It is worth mentioning that instead of the Hamming metric (on trajectories), other metrics have also been considered, and the Kantorovich metric on measures constructed from a~metric on trajectories can also be replaced with other ones: there is a~wide room for experiments and analysis of asymptotic properties of processes and filtrations depending on metrics.

It seems that the most important idea is that of measuring the degree of ``being non-Bernoulli'' and ``being nonstandard'' with the help of what was called ``secondary entropy'' and higher zero--one laws.

\section{Acknowledgements}
\label{sec9}

The author must remember and thank the people with whom he discussed the subject of this paper many years ago. First of all, this is V.\,A.~Rokhlin, one of my teachers, who posed a~series of stimulating questions on similar subjects and who was very attentive and sympathetic to my ideas in this field (see~\cite{47},~\cite{48}); A.\,N.~Kolmogorov, who was interested in filtrations and communicated several my papers on this subject to \textit{Proceedings of the USSR Academy of Sciences}; R.~Belinskaya, who asked me one of the first specific questions about dyadic filtrations; S.~Yuzvinsky, who took an active interest in the subject and wrote an important paper on applications of the scale of an automorphism~\cite{99}.

In those years, I had useful fierce disputes with Ya.~Sinai, B.~Gurevich, V.~Alexeev, A.~Katok, S.~Yuzvinsky, A.~Stepin, V.~Vinokurov, B.~Rubshtein, L.~Abramov, M.~Gordin, I.~Ibragimov, and others, as well as with A.~Kirillov, who suggested a~meaningful comparison of this subject to some previous works.

Later, after a~long pause, the interest to the subject was revived abroad; M.~Yor's question about filtrations related to the Brownian motion (see also~\cite{57}) was partially solved in~\cite{7}, where my standardness criterion was rediscovered. A~remarkable flurry of activity in the theory of filtrations was observed in 1997--1998 during the semester on ergodic theory at the Israel Institute for Advanced Studies (Jerusalem): H.~Furstenberg, B.~Weiss, J.~Feldman, D.~Rudolph, D.~Hoffman, D.~Heicklen,
J.-P.~Thouvenot published a~number of papers and participated in discussions of this subject. In the 2000s, M.~\'Emery, who gave a Bourbaki~talk~\cite{12} on this and related subjects, his students W.~Schachermayer and others also worked in the theory of filtrations. In the 1990s, the author had many discussions of the theory of filtrations with J.~Feldman (see~\cite{15},~\cite{16})
and B.~Tsirelson (see, in particular,~\cite{17}).

Most recently, especially after the subject has been combined with the theory of graded graphs,  the theory of filtrations and the notion of standardness attracted much interest from S.~Laurent, T.~de~la~Rue, and others; they obtained many new results mentioned in the survey (see~\cite{38, 39, 40}, \cite{25}). Finally, I have repeatedly discussed the subject with my students A.~Gorbulsky, F.~Petrov, P.~Zatitskiy, and others, which, hopefully, will lead to the discovery of new facts and a~new understanding of this interesting and important subject. The author thanks all people mentioned above and  those he may have forgotten to mention. Special thanks are due to A.\,A.~Lodkin for his help with formatting the bibliography, and to A\,M.~Minabutdinov for the figures.

\end{document}